\documentclass[11pt, oneside, reqno]{amsart}
\usepackage[usenames,dvipsnames,svgnames,table]{xcolor}
\usepackage[utf8]{inputenc}
\usepackage[margin = .95in]{geometry}     
\usepackage[mathscr]{euscript}
\usepackage{mathrsfs}
\usepackage{mathtools}
\usepackage{graphicx}
\usepackage{dsfont}
\usepackage{booktabs, multirow}
\usepackage{pst-node}
\usepackage{tikz-cd} 
\usepackage{amssymb}
\usepackage[utf8]{inputenc}
\usepackage{bm}
\usepackage{amsthm}
\usepackage[square,sort,comma,numbers]{natbib}
\usepackage{doi}
\usepackage[capitalize,nameinlink,noabbrev]{cleveref}
\usepackage[inline]{enumitem}
\newtheorem{theorem}{Theorem}[section]
\newtheorem{lemma}[theorem]{Lemma}
\newtheorem{prop}[theorem]{Proposition}
\newtheorem{cor}[theorem]{Corollary}

\newtheorem*{theorem*}{Theorem}
\theoremstyle{definition}
\newtheorem{definition}[theorem]{Definition}

\newtheorem{example}[theorem]{Example}
\newtheorem{remark}[theorem]{Remark}
\numberwithin{equation}{section}
\crefname{equation}{}{}
\Crefname{equation}{Equation}{Equations}
\date{\today}
\DeclareMathOperator{\Z}{\mathbb{Z}}
\DeclareMathOperator{\NN}{\mathbb{N}}
\DeclareMathOperator{\R}{\mathbb{R}}
\DeclareMathOperator{\C}{\mathbb{C}}
\DeclareMathOperator{\Q}{\mathbb{Q}}
\DeclareMathOperator{\Des}{Des}

\newcommand{\defn}[1]{{\color{RoyalBlue}\bf#1}}

\def \A{{\mathcal{A}}}
\def \BB{{\mathcal{B}}}
\def \GG{{\mathcal{G}}}
\def \MR {{\sf{MR}}}
\def \SS{{\mathcal{S}}}
\def \TT{{\mathcal{T}}}
\def \cE{{\mathcal{E}}}
\def \frakg{{\mathfrak{g}}}
\def \kk{\C}
\def \lat{{\mathcal{L}}}
\def \ol{\overline}
\def \sym{\mathfrak{S}}
\def \uu{{\mathfrak{u}}}
\def \ee{{\mathfrak{e}}}
\def \wh{\widehat}

\def \I {{\mathcal{I}}}
\def \ztwo {{C_2}}
\def \gsetcomposition{{\Sigma_n[G]}}
\def \gsetpartition{{\Pi_n[G]}}
\def \gcomposition{{\Gamma_n[G]}}
\def \gpartition{{\Lambda_n[G]}}
\def \ztwodual {{\widehat{\ztwo}}}

\def \ztwosetcomposition{{\Sigma_n[\ztwo]}}
\def \ztwosetpartition{{\Pi_n[\ztwo]}}
\def \ztwocomposition{{\Gamma_n[\ztwo]}}
\def \ztwopartition{{\Lambda_n[\ztwo]}}

\def \ztwosetcompositiondual{{\Sigma_n[\ztwodual]}}
\def \ztwosetpartitiondual{{\Pi_n[\ztwodual]}}
\def \ztwocompositiondual{{\Gamma_n[\ztwodual]}}
\def \ztwopartitiondual{{\Lambda_n[\ztwodual]}}

\def \hsiao {{\gsetcomposition}}
\DeclareMathOperator{\chleq}{\unlhd}
\DeclareMathOperator{\chgeq}{\unrhd}

\def\qqiff{\qquad\text{if and only if}\qquad}
\def\qqand{\qquad\text{and}\qquad}

\def \supp{\mathsf{supp}}
\def \type{\mathsf{type}}
\def \sgn{\mathbf{sgn}}
\def \trv{\mathbf{trv}}

\DeclareMathOperator{\sol}{Sol} 		
\DeclareMathOperator{\co}{Co} 		
\DeclareMathOperator{\rad}{rad} 		% radical

\newcommand{\B}[1]{\mathtt{#1}} 			% for basis elements
\newcommand{\SB}[1]{\mathcal{#1}} 		% for symmetrized basis elements
\newcommand{\abs}[1]{|{#1}|} 				% for underlying composition/partition
\title[Left Regular Bands of Groups and the Mantaci--Reutenauer algebra]{Left Regular Bands of Groups and \\ the Mantaci--Reutenauer algebra}
\author{Jose Bastidas, Sarah Brauner, and Franco Saliola}
\address{LACIM - Universit\'e du Qu\'ebec \`a Montr\'eal, Montr\'eal, Canada}
\email{bastidas\_olaya.jose\_dario@uqam.ca, saliola.franco@uqam.ca}
\urladdr{\url{https://sites.google.com/view/bastidas}, \url{https://saliola.github.io/}}
\address{Max Planck Institute for Mathematics in the Sciences, Leipzig, Germany}
\email{sarah.brauner@mis.mpg.de}
\urladdr{\url{https://www.sarahbrauner.com/}}



\usepackage{xargs}
\usepackage[colorinlistoftodos,prependcaption,textsize=tiny,linecolor=red,backgroundcolor=red!25,bordercolor=red]{todonotes}
\newcommandx{\sarah}[2][1=]{\todo[linecolor=blue,backgroundcolor=blue!25,bordercolor=blue,#1]{Sarah: #2}}
\newcommandx{\franco}[2][1=]{\todo[linecolor=orange,backgroundcolor=orange!25,bordercolor=orange,#1]{Franco: #2}}
\newcommandx{\jose}[2][1=]{\todo[linecolor=green,backgroundcolor=green!25,bordercolor=green,#1]{Jose: #2}}

\begin{document}
\begin{abstract}
We develop the idempotent theory for algebras over a class of semigroups called
left regular bands of groups (LRBGs), which simultaneously generalize group
algebras of finite groups and left regular band (LRB) algebras.  Our techniques
weave together the representation theory of finite groups and
LRBs, opening the door for a systematic study of LRBGs in an
analogous way to LRBs. We apply our results to construct complete systems of
primitive orthogonal idempotents in the Mantaci--Reutenauer algebra $\MR_n[G]$
associated to any finite group $G$. When $G$ is abelian, we give closed form
expressions for these idempotents, and when $G$ is the cyclic group of order
two, we prove that they recover idempotents introduced by Vazirani.
\end{abstract}
\dedicatory{To the memory of Georgia Benkart.}

\maketitle

\section{Introduction}
This paper studies a class of semigroup algebras called \defn{left regular
bands of groups} (\defn{LRBG}), which simultaneously generalize \defn{left
regular band} (\defn{LRB}) algebras and group algebras of finite groups. Throughout, all semigroups are assumed to be finite.

A LRB is a semigroup $\BB$ such that for any $x,y \in \BB$, one has $x^2 = x$ and $xyx= xy$.
To obtain a LRBG, one relaxes these conditions as we now explain.
Following the standard semigroup theory conventions, for any $s \in \SS$, we let $s^{\omega} \in \SS$ denote the unique idempotent that is a positive power of $s$.
This element necessarily exists because $\SS$ is finite, and the particular power $N \in \Z_{> 0}$ that makes $s^N = s^\omega$ depends on $s$.
A semigroup $\SS$ is a LRBG if for any $s,t \in \SS$ one has
\[ s^{\omega}s = s \qqand st s^{\omega} = st.\]
Note that if $s^\omega = s$ for all $s \in \SS$, this recovers the definition of a LRB;
while if $s^{\omega} = t^{\omega}$ for all $s,t \in \SS$, then $\SS$ is a group.

We obtain from a LRBG $\SS$ two important pieces of data:
\begin{enumerate}
    \item a LRB $E(\SS)$, formed by the idempotent elements of $\SS$:
    \[ E(\SS)= \{ x \in \SS: x^2 = x \};\]
    \item a collection of groups $\{ G_{x} \}$, one for each $x \in E(\SS)$, where 
    \[G_x = \{ s \in \SS : s^\omega = x \}. \]
\end{enumerate}
Together, the $G_x$ ``glue together'' to form $\SS$, in the sense that 
\[
    \SS = \bigsqcup_{x \in E(\SS)} G_x,
\]
where $\bigsqcup$ denotes the disjoint union of sets.

Both LRBs and group algebras are important in their own right; the latter are fundamental to the representation theory of finite groups, while the former algebras play a key role in the theory of hyperplane arrangements, Coxeter theory, and certain Markov chains such as card shuffling. One way of understanding the role of LRBGs in this context is that they are to the representation theory of wreath products with symmetric groups what LRBs are to the representation theory of Coxeter groups. 

\subsection*{Idempotent theory for LRBGs}
Constructing complete systems of primitive orthogonal idempotents (CSoPOI) of any finite dimensional algebra is a crucial piece of understanding its structure and representation theory. 
In this paper, we produce such systems for the $\kk$-algebra\footnote{Because we will be dealing with group representation theory, we will take our coefficients throughout to be in $\C$. However, in general any field where Maschke's Theorem holds would be adequate.} $\kk \SS$ generated by the LRBG $\SS$. Our work builds up the general theory of LRBGs, whose study was initiated by Margolis and Steinberg in \cite{ms11HomologySemigroups}, where they compute the quiver\footnote{Though they compute the quiver of $\kk \SS$, they do not construct CSoPOI to do so.} of $\kk \SS$.

Our approach intertwines the idempotent theory of LRBs with the idempotent theory of $\kk G$ for a $G$ a finite group.  We review these theories briefly.
\begin{enumerate}
\item \emph{Idempotents for LRBs algebras}: In \cite{saliola07quiverLRB, sal2008quiv, saliola2009face, saliola2012eigenvectors}, the third author constructed families of complete systems of primitive orthogonal idempotents for a LRB algebra $\kk \BB$ using the \defn{support map} 
\[ \supp: \kk \BB \longrightarrow \kk \lat,\]
where $\lat$ is the \defn{support lattice} of $\BB$ (see \ \cref{ss:CSoPOI-LRB}).
He show that one obtains a CSoPOI $\{\ee^\circ_{X}\}_{X \in \lat}$
parameterized by certain sections of $\supp$ (see \cref{thm:CSoPOIforLRB}).
Moreover, Aguiar and Mahajan showed that all CSoPOIs for
a LRB algebra arise in this way \cite[Prop. 11.9]{am17}.

    \item \emph{Idempotents for group algebras}:
In general, there is no canonical way to produce a CSoPOI for $\kk G$ when $G$ is an arbitrary finite group, although some case-by-case constructions exist (for example, Young's idempotents for the symmetric group, $\sym_n$). However, the character theory of finite groups does give a method of constructing complete family of orthogonal idempotents that lie in the center $Z(\kk G)$ of the group algebra; these idempotents will be a CSoPOI if and only if $G$ is abelian.

\end{enumerate}

Our main result is to prove that one can obtain complete families of primitive orthogonal idempotents for any LRBG algebra using the idempotents for the LRB and group algebras it contains. 
    \begin{theorem}[\cref{thm:CSoPOI-LRBG}]\label{thm:mainthmintro}
        Let $\SS$ be a left regular band of groups and $E(\SS)$ its left regular band of idempotents.
        \begin{enumerate}
            \item Let $\{ \ee^\circ_X \}_{X \in \lat}$ be a complete system of primitive orthogonal idempotents for $\C E(\SS)$ and,

            \item
                for each $X \in \lat$, let
                $\{ E_X^{(i)} \}_{i}$ be a complete system of primitive orthogonal idempotents for $\C G_x$, where $x \in E(\SS)$ is a fixed element with $\supp(x) = X$.
        \end{enumerate}
        Then the elements
        \[ \ee_{(X,i)} := \ee^\circ_X E_X^{(i)} \ee^\circ_X \]
        form a complete system of primitive orthogonal idempotents for $\C \SS$.
    \end{theorem}
Our techniques are in the spirit of Aguiar--Mahajan's monograph \cite{am17}, which develops a vast theory for LRBs
associated to a hyperplane arrangement $\A$. Our work thus opens
the door to a systematic study of LRBGs in an analogous way. This approach is especially fruitful when $\SS$ is a \defn{left regular band of abelian groups} (\defn{LRBaG}), meaning that $G_x$ is abelian for all $x \in E(\SS)$. In this case, we produce several bases for $\kk \SS$, and describe how to multiply them (see \cref{sec:LRBaG}). 

\subsection*{Connections to the Mantaci--Reutenauer algebra}

Our interest in LRBGs is motivated in part by a beautiful connection, due to Hsiao \cite{hsiao2007semigroup}, between certain LRBG algebras $\kk \Sigma_n[G]$ (henceforth, Hsiao's algebra---see \cref{ss:Hsiao's}) and the \defn{Mantaci--Reutenauer algebra} $\MR_n[G] \subseteq \kk \sym_n[G]$ for $G$ a finite group (see \cref{thm/def:mralg}). The Mantaci--Reutenauer algebra has been studied in a variety of contexts related to the representation theory of the wreath product $\sym_n[G]$ by Baumann--Hohlweg \cite{BauHoh2008solomon} and Douglass--Tomlin \cite{DouTom2018decomposition}, the descent and peak algebras by Aguiar--Bergeron--Nyman \cite{abn04peak}, Hopf algebras and shuffling by Pang \cite{pang2021eigenvalues}, and the Whitehouse representations and configuration spaces by the second author \cite{brauner2023type}.

Hsiao showed in \cite{hsiao2007semigroup} that the $\sym_n$-invariant subalgebra of $\kk \Sigma_n[G]$ is anti-isomorphic to $\MR_n[G]$. This establishes a LRBG analogue to a celebrated result of Bidigare for LRBs connecting the face algebra of a reflection arrangement to Solomon's Descent algebra.

Motivated by Hsiao's result, we apply \cref{thm:mainthmintro} to any LRBaG algebra $\kk \SS$ with the action of a finite group $W$. In doing so, we consequently produce families of CSoPOI for the $W$-invariant subalgebras $(\kk \SS)^W$ (\cref{thm:CSoPOI-invar-subalg}).

We then specialize to the case of Hsiao's algebra $\kk \Sigma_n[G]$ when $G$ is abelian. In this setting, we are able to give closed form expressions for several bases of $\kk \Sigma_n[G]$ and their change-of-basis formulae (\cref{s:change-of-basis-hsiao}). Moreover, we obtain an explicit expression of a CSoPOI for $\kk \Sigma_n[G]$ and $(\kk \Sigma_n[G])^{\sym_n}$:

\begin{theorem}[\cref{cor:uniform-CSoPOI-Hsiao}]\label{thm:intro_mridem}
When $G$ is abelian, the collection $\{ \ee_{\lambda} \}_{\lambda \in \Lambda_n[\wh{G}]}$,
where
\begin{align*}
 \ee_{\lambda} &= \frac{1}{\ell(\lambda)!} \sum_{\substack{ \alpha \in \Gamma_n[\wh{G}] \\ \supp(\alpha) = \lambda}} \sum_{\substack{\beta \in \Gamma_n[\wh{G}] \\ \alpha \chleq \beta}} \dfrac{(-1)^{\ell(\beta)-\ell(\alpha)}}{\deg(\beta/\alpha)} \cE_\beta \\
 &= \dfrac{1}{|G|^{\ell(\lambda)} \cdot \ell(\lambda)!} \sum_{\substack{ \alpha \in \Gamma_n[\wh{G}] \\ \supp(\alpha) = \lambda}} \sum_{\substack{p \in \Gamma_n[G] \\ \abs{p} = \abs{\alpha} }} \ol{\alpha(p)} \sum_{\substack{q \in \Gamma_n[G] \\ p \leq q }} \dfrac{(-1)^{\ell(q)-\ell(p)}}{\deg(q/p)} \SB{H}_q,
\end{align*}
gives a complete system of primitive orthogonal idempotents
for $(\kk \Sigma_n[G])^{\sym_n}$.

Applying Hsiao's map gives a complete system of primitive orthogonal idempotents for $\MR_n[G]$; when $G = \{ \pm 1 \}$ this system recovers the idempotents introduced by Vazirani in \cite{vazirani}. 
\end{theorem}

Here, $\wh{G}$ is the group of characters of $G$, and $\ol{\alpha(p)} \in \C^{\times}$ is an evaluation of a character of $G^{\ell(\lambda)}$. For any group $H$, the sets $\Gamma_n[H]$ and $\Lambda_n[H]$ are \defn{$H$-colored (integer) compositions} of $n$ and \defn{$H$-colored (integer)  partitions} of $n$, respectively  (see \cref{ex:invariantshsiao}), while $\supp(\alpha) \in \Lambda_n[H]$ is the $H$-colored partition obtained by forgetting the order of the blocks in $\alpha \in \Gamma_n[H]$.
The set $\{ \SB{H}_p \}_{p \in \Gamma_n[G]}$ is the ``canonical'' basis of
$(\kk \Sigma_n[G])^{\sym_n}$ indexed by elements of the semigroup, and  $\{
\SB{E}_{\beta} \}_{\beta \in \Gamma_n[\wh{G}]}$ is another basis (see
\cref{sec:invarLRBaG}). See \cref{def:orderforMR} for the definitions
of the order relations $\leq$ and $\chleq$ and
\cref{def:deg} for the definition of $\deg(\beta/\alpha)$.

We are not aware of other results constructing idempotents for $\MR_n[G]$ beyond the case of $G =\{\pm 1\}$ considered by Vazirani. One advantage of our work is that the idempotents in \cref{thm:intro_mridem} can be written explicitly in terms of a well-known basis for $\MR_n[G]$. In the case that $G =\{\pm 1\}$, this is a vast simplification compared to other expressions of Vazirani's idempotents appearing in the literature (e.g. \cite{vazirani, DouTom2018decomposition}); see \cref{cor:vazclosedformexpression} for our precise formula.

\cref{thm:intro_mridem} introduces many interesting questions. For instance, Vazirani's idempotents have been shown by the second author to generate representations describing the symmetries of cohomology rings of certain orbit configuration spaces \cite{brauner2023type}. Our work provides a new perspective to understand these idempotents, and suggests there may be connections to other orbit configuration spaces.

\subsection*{Summary of the remainder of the paper}
The remainder of the paper proceeds as follows:
\begin{itemize}
    \item  \cref{sec:backgroundLRB} reviews the relevant theory of LRBs, including several interesting examples and a construction of a CSoPOI for any LRB algebra (\cref{thm:CSoPOIforLRB}).
    \item   \cref{sec:LRBG} introduces LRBGs and develops their properties using lattices of semigroups. Hsiao's algebra serves as a running example throughout, beginning with \cref{ss:Hsiao's}. We introduce a basis called the $\B{R}$-basis for the semigroup algebra generated by any LRBG $\SS$.
     \item  \cref{sec:idemforLRBGs} constructs families of CSoPOI for any LRBG, thereby proving \cref{thm:mainthmintro}. We then study properties of LRBGs of abelian groups (LRBaG) in \cref{sec:LRBaG}, introduce several additional bases (the $\B{E}$-basis and the $\B{Q}$-basis), and describe how they can be used to construct CSoPOIs in the abelian case. 
     \item \cref{sec:LRBGsymmetry} considers the case where the LRBG algebra $\kk \SS$ is acted on by a group $W$, and produces bases (the so called $\SB{R}$, $\SB{E}$ and $\SB{Q}$ bases) and CSoPOIs for the $W$-invariant subalgebra $(\kk \SS)^W$ (\cref{thm:CSoPOI-invar-subalg}).
     \item  \cref{sec:hsiaosalgebra} specializes to the case that $\SS$ is Hsiao's algebra over an abelian group. In this case, we give explicit formulae for the $\SB{R}$, $\SB{E}$ and $\SB{Q}$ bases, as well as change-of-basis formulas between them (\cref{s:change-of-basis-hsiao}). We then give an explicit construction of a CSoPOI, thereby proving \cref{thm:intro_mridem}.
 In \cref{sec:externalproductonHsiao} we introduce an external product on Hsiao's algebra and show it is compatible with the aforementioned bases.
 \item In \cref{sec:MRalg}, we define the Mantaci--Reutenauer algebra $\MR_n[G]$ and use Hsiao's map to construct CSoPOI for $\MR_n[G]$ using the results in \cref{sec:hsiaosalgebra}. We then specialize in \cref{s:typeBMR} to Type $B$, and make precise the relationship between our construction and the ``classical'' objects in $\MR_n[\pm 1]$, including Vazirani's idempotents. 

    \item  \cref{s:appendix} clarifies some ambiguous language in the theory of semigroups by outlining the difference between a \defn{band of groups} and a \defn{left regular band of groups}. We characterize the LRBGs that are in fact bands of groups, which we call \defn{strict LRBGs}.

\item \cref{s:appendix-notation} summarizes the notation in the paper. 
\end{itemize}

\subsection*{Acknowledgements}
The authors are very grateful to Marcelo Aguiar, Stuart Margolis, Simon Malenfant, Vic Reiner and Benjamin Steinberg for illuminating discussions. They thank an anonymous referee for helpful suggestions to improve the exposition. The second author acknowledges support from NSF Graduate Research Fellowship while part of this research was conducted.
The first and third authors acknowledge the support of the Natural Sciences and Engineering Research Council of Canada (NSERC).
This research was facilitated by computer exploration using the open-source
mathematical software system \textsc{SageMath}~\cite{SageMath} and its algebraic
combinatorics features developed by the \textsc{Sage-Combinat} community 
\cite{sage-combinat}. 

\subsection*{Historical Note}
After posting our preprint on the arXiv, we were contacted by Benjamin Steinberg, who informed us that he was also working on a paper related to LRBGs (which he calls left dual monoids).
His preprint has now appeared in \cite{steinberg23topology}. Steinberg uses topological methods to compute Ext between certain representations of LRBGs. He describes CSoPOIs for an LRBG algebra, in the same spirit as \cref{thm:mainthmintro}. Here, we additionally study LRBGs with symmetry and compute CSoPOIs for the Mantaci-Reutenauer algebra, something his work does not consider.
\section{Background on Left Regular Bands}\label{sec:backgroundLRB}

\subsection{Left regular bands}
Recall from the introduction the definition of a left regular band:
\begin{definition}\label{def:lrb}
    A \defn{left regular band (LRB)} is a semigroup $\BB$ such that for any $x,y \in \BB$
    \begin{equation*}
        x^2 = x
        \quad\text{and}\quad
        xyx = xy.
    \end{equation*}
\end{definition}
We will restrict our attention to finite semigroups that contain a unit, in which case
the identity $x^2 = x$ in \cref{def:lrb} follows from the identity $xyx = xy$ by setting $y$ equal to the unit.
LRBs play an important role in the theory of hyperplane arrangements~\cite{am17}, in determining mixing times of certain important Markov chains~\cite{BHR}, and in connection with Solomon's Descent algebra~\cite{bidigare}.
We refer the reader to \cite[Appendix B]{BrownOnLRBs} for the basics of left regular bands, which we summarise next.

\begin{example}[Hyperplane arrangement]\label{ex:hyperplanearrangement} \rm
Let $\A = \{ H_i \}_{i \in I} \subseteq \R^k$ be a real, central hyperplane arrangement, meaning that we have a finite
collection of $\R^{k-1}$-dimensional subspaces $H_i$ of $\R^k$.
Each hyperplane $H_i$ divides $\R^k$ into two open half spaces, and the
non-empty intersections of (the closure of) half-spaces defines the \defn{faces}
of $\A$.
An example of a hyperplane arrangement in $\R^3$ is shown in \cref{fig:rank4braid},
intersected with the sphere; here ${\sf G}$ and ${\sf H}$ are instances
of \defn{chambers} (faces of maximal dimension), while ${\sf F}$ is a 2-dimensional
face, and ${\sf F \cap \sf G}$ is a $1$-dimensional face.

\begin{figure}[ht]
\centering
      \includegraphics[width=0.25\textwidth]{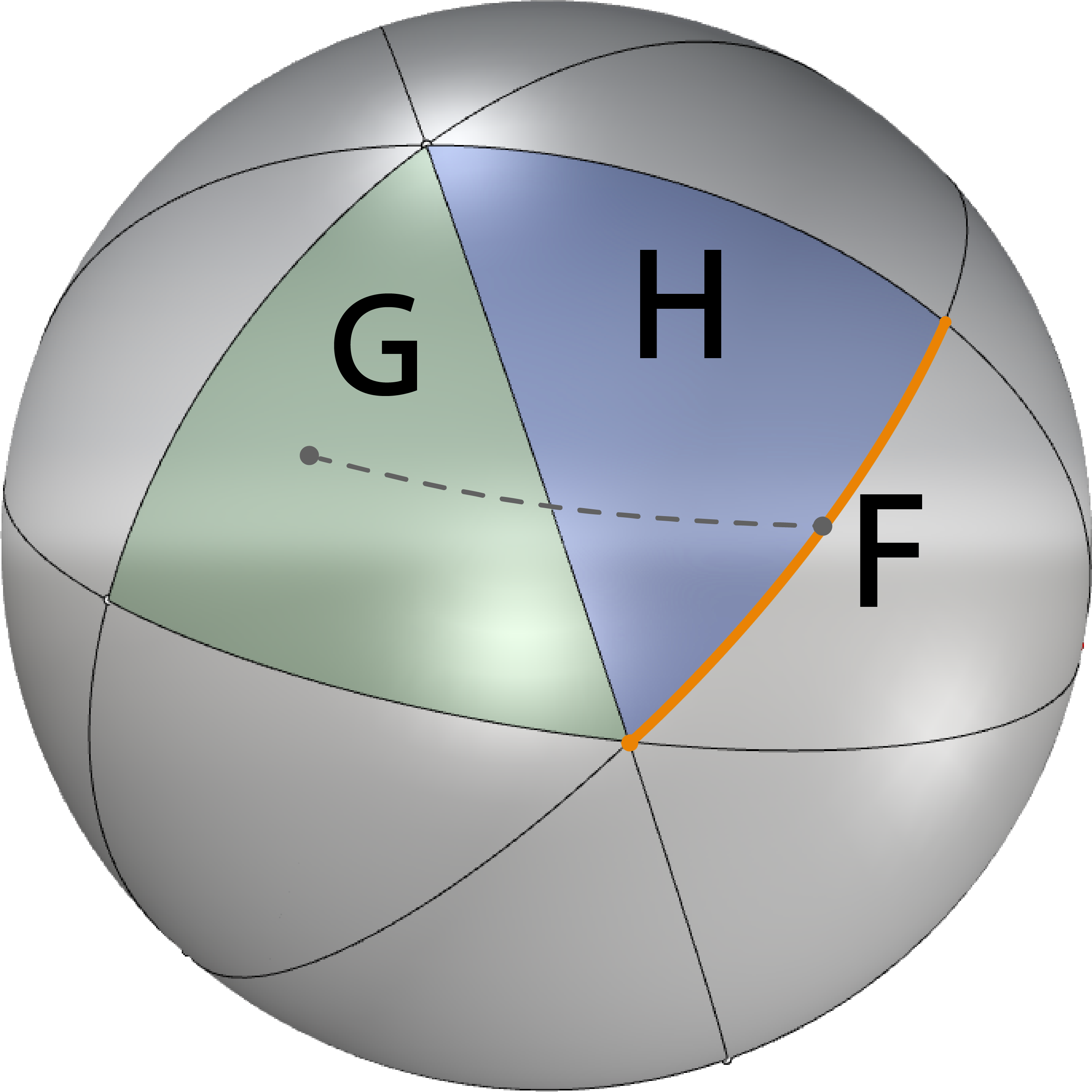}
      \caption{A hyperplane arrangement in $\R^3$, intersected with the sphere.} \label{fig:rank4braid}
\end{figure}
\end{example}
In \cite{tits1974buildings}, Tits famously showed that one can define a product
on the set of faces $\Sigma_{\A}$ of $\A$.
The \defn{Tits product} says that ${\sf F \cdot G} = {\sf H}$,
where ${\sf H}$ is the unique face one encounters starting at a generic point
in ${\sf F}$ and moving a distance of $\epsilon > 0$ on the shortest path
towards a generic point in ${\sf G}$; this is illustrated in
\cref{fig:rank4braid}. With this product, $\Sigma_{\A}$ is a LRB for any
hyperplane arrangement $\A$.

\begin{example}[The Braid arrangement]\label{ex:braidarrangement} \rm
    Consider the special case where the hyperplane arrangement is the
    \defn{Braid arrangement} $\mathcal{A}_{n-1}$
    consisting of hyperplanes $H_{ij}:=
    \{x \in \R^n: x_i = x_j\}$ with $1 \leq i < j \leq n$.
    In this case, both the faces and the Tits product admit a more combinatorial
    description, as follows (see \cite[Chapter 6]{am17} for an in-depth description).

    The faces of $\mathcal{A}_{n-1}$ are in bijection with \defn{ordered} set partitions of
    $[n]:= \{1, \ldots, n\}$: explicitly, an \defn{ordered set partition},
    or a \defn{set composition}, of $[n]$ is an ordered list $(S_1, \ldots, S_\ell)$
    of nonempty subsets of $[n]$ such that $S_i \cap S_j
    = \emptyset$ if $i \neq j$ and $S_1 \cup \cdots \cup S_\ell = [n]$.
    Let $\Sigma_n$ be the set of set compositions of~$[n]$.

    Under the bijection between faces of $\mathcal{A}_{n-1}$ and set compositions in $\Sigma_n$,
    the face of $\mathcal{A}_{n-1}$ corresponding to a set composition $(S_1, \ldots,
    S_\ell)$ with $\ell$ blocks is of dimension $\ell-1$.
    Furthermore, the Tits product of two faces translates to the following product
    on set compositions:
    \[ (S_1, \ldots, S_{\ell}) \cdot (T_1, \ldots T_{m}) = (S_1 \cap T_1, S_1 \cap T_2, \ldots, S_1 \cap T_m, S_2 \cap T_1, S_2 \cap T_2, \ldots, S_{\ell} \cap T_m){^\dagger},\]
    where the superscript $\dagger$ denotes that the empty intersections have been omitted.
    By construction, $(S_1, \ldots, S_{\ell}) \cdot (T_1, \ldots T_{m})$ will always be a refinement of $(S_1, \ldots, S_{\ell})$.
    For instance,
    \[ ( \{ 1,3, 4 \}, \{ 2, 5 \}) \cdot (\{1,2 \}, \{5 \}, \{ 3,4 \}) = (\{ 1 \}, \{3,4 \}, \{ 2 \}, \{ 5 \}). \]
    Note this product is highly non-commutative; multiplying in the other order gives
    \[  (\{1,2 \}, \{5 \}, \{ 3,4 \}) \cdot ( \{ 1,3, 4 \}, \{ 2, 5 \})= (\{ 1 \}, \{ 2 \}, \{ 5 \}, \{3,4 \}). \]
    To simplify notation for set compositions, we concatenate the elements of
    the blocks and replace commas by vertical bars, for example writing $(12|5|34)$ for the
    set composition $(\{1,2 \}, \{5 \}, \{ 3,4 \})$.
\end{example}

\begin{example}[The free LRB]\label{ex:freelrb} \rm
    In \cite{BrownOnLRBs}, Brown considered the \defn{free left regular band}
    $\mathscr{F}_n$, which is a monoid on the set of injective words on
    $[n]$. An \defn{injective word} on $[n]$ is a sequence of elements of $[n]$
    with no repeated entries. Multiplication of words $u,v \in \mathscr{F}_n$
    is given by first concatenating $u$ and $v$ and then removing repeated
    letters from right to left. For example,
    \[ (1, 2, 4) \cdot (3,1,4) = (1,2,4,3)
    \quad\text{and}\quad
    (3,1,4)  \cdot (1, 2, 4) = (3,1,4,2).\]
    To simplify notation, we write $\epsilon$ for the empty sequence and
     write sequences as words. For instance we write $314$ for $(3,1,4)$.
    The free LRB $\mathscr{F}_n$ has connections to card shuffling and derangements; see \cite{brauner2022invariant} for more on this fascinating semigroup.
\end{example}

\subsection{Support lattice and support map}\label{ss:LRBrelations}

Given a LRB $\BB$, we consider two relations on $\BB$.

First, define the relation $\leq$ for $x, y \in \BB$ by
\[ x \leq y \qqiff xy=y. \]
This is a partial order on $\BB$ (i.e. $ \leq$ is reflexive, transitive, and antisymmetric),
so $\BB$ is a poset.

Second, define the relation $\sim$ for $x, y \in \BB$ by
\begin{equation*}
    x \sim y
    \qqiff
    x y = x
    \text{~and~}
    y x = y.
\end{equation*}
This is an equivalence relation and also a semigroup congruence, meaning that
the quotient $\BB/\mathord{\sim}$ inherits a semigroup structure. More
precisely, we have the following.

\begin{prop}[Support lattice and support map]\label{def:supportmap}
    Let $\BB$ be a LRB.
    Let $\lat$ be the quotient of $\BB$ by the equivalence relation $\sim$,
    and let $\supp$ denote the quotient map
    \[ \supp: \BB \longrightarrow  \lat. \]
    Then
    \begin{enumerate}
        \item
            $\lat$ is a join semilattice, and hence a semigroup with multiplication
            $X \cdot Y = X \vee Y$;

        \item\label{item-semigroup-map}
            $\supp: \BB \to \lat$ is a map of semigroups
            (i.e., $\supp(xy) = \supp(x) \vee \supp(y)$);

        \item\label{item-poset-map}
            $\supp: \BB \xrightarrow{} \lat$ is a map of posets
            (i.e., $x \leq y$ in $\BB$ implies $\supp(x) \leq \supp(y)$ in $\lat$).
    \end{enumerate}
\end{prop}
We will be concerned with finite LRBs $\BB$ that contain a unit, in which case $\lat$ is a lattice.
We call $\lat$ the \defn{support lattice} of $\BB$ and $\supp$ the \defn{support map}.
These objects will be essential in studying both LRBs and their generalization, LRBGs.

\begin{example}[Lattice of flats of a Hyperplane Arrangement]\label{ex:latticeflats}
    Continuing \cref{ex:hyperplanearrangement}, if $\BB = \Sigma_{\A}$ is
    the LRB of faces of an arrangement $\A$, then $\lat$ is the so-called
    \defn{lattice of flats} associated with $\A$. Recall that a \defn{flat} of an
    arrangement $\A$ is a linear subspace formed by intersecting a subset of
    hyperplanes in $\A$; this includes the ambient vector space viewed as the
    empty intersection. The set of flats has a natural lattice structure
    induced from ordering subspaces by inclusion (although some authors choose
    to order subspaces by reverse inclusion). Properties of the lattice of
    flats have been extensively studied, see for example \cite{am17}.

    The support map in this case has a simple geometric description: given
    a face ${\sf F} \in \Sigma_{\A}$, its support $\supp({\sf F})$ is the
    smallest subspace containing $F$.
\end{example}

\begin{example}[Partition lattice]\label{ex:partitionlattice} \rm
In the case of the braid arrangement $\mathcal{A}_{n-1}$ both the support lattice and
the support map can be described combinatorially.
Recall from \cref{ex:braidarrangement}, that the faces of $\mathcal{A}_{n-1}$
are in bijection with ordered set partitions $\Sigma_n$.
The support lattice is then the lattice $\Pi_n$ of \defn{set
partitions} of $[n]$ ordered by reverse refinement. To simplify notation for set partitions, we concatenate the elements of
the blocks and replace commas by vertical bars,
for example, writing $\{12|34|5\}$ for the
set partition $\{\{1,2 \}, \{5 \}, \{ 3,4 \}\}$. The support map
\[ \supp: \Sigma_n \longrightarrow \Pi_n \]
simply forgets the ordering of an ordered set partition;
see \cref{fig:posets-of-set-compositions-and-set-partitions}. For instance, 
\[
    \supp\big( 134|25 \big) =
    \supp\big(25|134 \big) =
    \{134|25\} \in \Pi_5.
\]

    \begin{figure}[htpb]
        \centering
        \begin{tikzpicture}[xscale=0.7]
            \node (node_00) at (171.0bp,   6.5bp) {$(123)$};
            \node (node_01) at (146.0bp,  55.5bp) {$(12|3)$};
            \node (node_02) at ( 96.0bp,  55.5bp) {$(13|2)$};
            \node (node_03) at ( 46.0bp,  55.5bp) {$(1|23)$};
            \node (node_06) at (296.0bp,  55.5bp) {$(23|1)$};
            \node (node_07) at (246.0bp,  55.5bp) {$(2|13)$};
            \node (node_10) at (196.0bp,  55.5bp) {$(3|12)$};
            \node (node_04) at ( 81.0bp, 104.5bp) {$(1|2|3)$};
            \node (node_08) at (201.0bp, 104.5bp) {$(2|1|3)$};
            \node (node_05) at ( 21.0bp, 104.5bp) {$(1|3|2)$};
            \node (node_11) at (141.0bp, 104.5bp) {$(3|1|2)$};
            \node (node_09) at (321.0bp, 104.5bp) {$(2|3|1)$};
            \node (node_12) at (261.0bp, 104.5bp) {$(3|2|1)$};
            \draw [black] (node_00) -- (node_01);
            \draw [black] (node_00) -- (node_02);
            \draw [black] (node_00) -- (node_03);
            \draw [black] (node_00) -- (node_06);
            \draw [black] (node_00) -- (node_07);
            \draw [black] (node_00) -- (node_10);
            \draw [black] (node_01) -- (node_04);
            \draw [black] (node_01) -- (node_08);
            \draw [black] (node_02) -- (node_05);
            \draw [black] (node_02) -- (node_11);
            \draw [black] (node_03) -- (node_04);
            \draw [black] (node_03) -- (node_05);
            \draw [black] (node_06) -- (node_09);
            \draw [black] (node_06) -- (node_12);
            \draw [black] (node_07) -- (node_08);
            \draw [black] (node_07) -- (node_09);
            \draw [black] (node_10) -- (node_11);
            \draw [black] (node_10) -- (node_12);
        \end{tikzpicture}
        \qquad
        \begin{tikzpicture}[xscale=0.7]
            \node (node_00) at (171.0bp,   6.5bp) {$\{123\}$};
            \node (node_01) at (100.0bp,  55.5bp) {$\{12|3\}$};
            \node (node_02) at (171.0bp,  55.5bp) {$\{13|2\}$};
            \node (node_03) at (242.0bp,  55.5bp) {$\{1|23\}$};
            \node (node_04) at (171.0bp, 104.5bp) {$\{1|2|3\}$};
            \draw [black] (node_00) -- (node_01);
            \draw [black] (node_00) -- (node_02);
            \draw [black] (node_00) -- (node_03);
            \draw [black] (node_01) -- (node_04);
            \draw [black] (node_02) -- (node_04);
            \draw [black] (node_03) -- (node_04);
        \end{tikzpicture}
        \caption{On the left are the elements of $\Sigma_3$ ordered by the
            relation $\leq$; that is, set compositions of $\{1,2,3\}$ ordered by
            reverse refinement. Its support lattice is the lattice $\Pi_3$ of
            set partitions of $\{1,2,3\}$ ordered by reverse refinement (depicted on
            the right).
            The map $\supp: \Sigma_3 \xrightarrow{} \Pi_3$ forgets the ordering
            of the blocks.}
        \label{fig:posets-of-set-compositions-and-set-partitions}
    \end{figure}
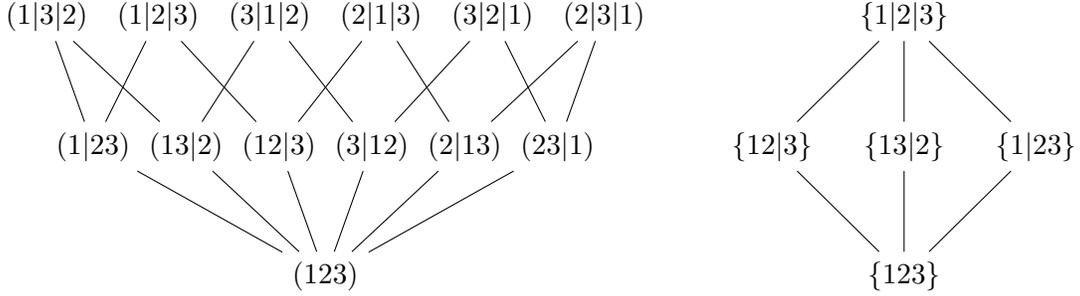
\end{example}

\begin{example}[Boolean lattice]\label{ex:booleanlattice} \rm
    The support lattice of the free LRB $\mathscr{F}_n$ in \cref{ex:freelrb}
    can be identified with the Boolean lattice consisting of
    all subsets of $[n],$ ordered by inclusion;
    see \cref{fig:posets-for-free-lrb-3}. An
    injective word $u \in \mathscr{F}_n$ is sent by the support map to the
    subset $J \subseteq [n]$ where $j \in J$ if and only if $j$ appears in $u$.
    For instance, the word $u = (3,1,4)$ has support $\{ 3,1,4 \}$.
    \begin{figure}[ht]
        \centering
        \begin{tikzpicture}[yscale=0.7, xscale=0.9]
            \node (node_00) at (100bp,  0bp) {$\epsilon$};
            \node (node_01) at (25bp,  50bp) {${{1}}$};
            \node (node_02) at ( 5bp, 100bp) {${{12}}$};
            \node (node_04) at (45bp, 100bp) {${{13}}$};
            \node (node_03) at ( 5bp, 150bp) {${{123}}$};
            \node (node_05) at (45bp, 150bp) {${{132}}$};
            \node (node_06) at (100bp,  50bp) {${{2}}$};
            \node (node_07) at ( 80bp, 100bp) {${{21}}$};
            \node (node_09) at (120bp, 100bp) {${{23}}$};
            \node (node_08) at ( 80bp, 150bp) {${{213}}$};
            \node (node_10) at (120bp, 150bp) {${{231}}$};
            \node (node_11) at (175bp,  50bp) {${{3}}$};
            \node (node_12) at (155bp, 100bp) {${{31}}$};
            \node (node_14) at (195bp, 100bp) {${{32}}$};
            \node (node_13) at (155bp, 150bp) {${{312}}$};
            \node (node_15) at (195bp, 150bp) {${{321}}$};
            \draw [black] (node_00) -- (node_01);
            \draw [black] (node_00) -- (node_06);
            \draw [black] (node_00) -- (node_11);
            \draw [black] (node_01) -- (node_02);
            \draw [black] (node_01) -- (node_04);
            \draw [black] (node_02) -- (node_03);
            \draw [black] (node_04) -- (node_05);
            \draw [black] (node_06) -- (node_07);
            \draw [black] (node_06) -- (node_09);
            \draw [black] (node_07) -- (node_08);
            \draw [black] (node_09) -- (node_10);
            \draw [black] (node_11) -- (node_12);
            \draw [black] (node_11) -- (node_14);
            \draw [black] (node_12) -- (node_13);
            \draw [black] (node_14) -- (node_15);
        \end{tikzpicture}
        \qquad\qquad
        \begin{tikzpicture}[yscale=0.7, xscale=0.65]
            \node (node_0)   at (100bp,  0bp) {$\{\}$};
            \node (node_1)   at (25bp,  50bp) {${\{1\}}$};
            \node (node_2)   at (100bp,  50bp) {${\{2\}}$};
            \node (node_3)   at (175bp,  50bp) {${\{3\}}$};
            \node (node_12)  at ( 25bp, 100bp) {${\{1, 2\}}$};
            \node (node_13)  at (100bp, 100bp) {${\{1, 3\}}$};
            \node (node_23)  at (175bp, 100bp) {${\{2, 3\}}$};
            \node (node_123) at (100bp, 150bp) {$\{1,2,3\}$};
            \draw [black] (node_0) -- (node_1);
            \draw [black] (node_0) -- (node_2);
            \draw [black] (node_0) -- (node_3);
            \draw [black] (node_1) -- (node_12);
            \draw [black] (node_1) -- (node_13);
            \draw [black] (node_2) -- (node_12);
            \draw [black] (node_2) -- (node_23);
            \draw [black] (node_3) -- (node_13);
            \draw [black] (node_3) -- (node_23);
            \draw [black] (node_12) -- (node_123);
            \draw [black] (node_13) -- (node_123);
            \draw [black] (node_23) -- (node_123);
        \end{tikzpicture}
        \caption{On the left are the elements of $\mathscr{F}_3$ ordered by
            the relation $\leq$. Its support lattice is the lattice of subsets
            of $\{1,2,3\}$ ordered by inclusion (depicted on the right). The
            support map sends a word to the set of letters appearing in the
        word.}
        \label{fig:posets-for-free-lrb-3}
    \end{figure}
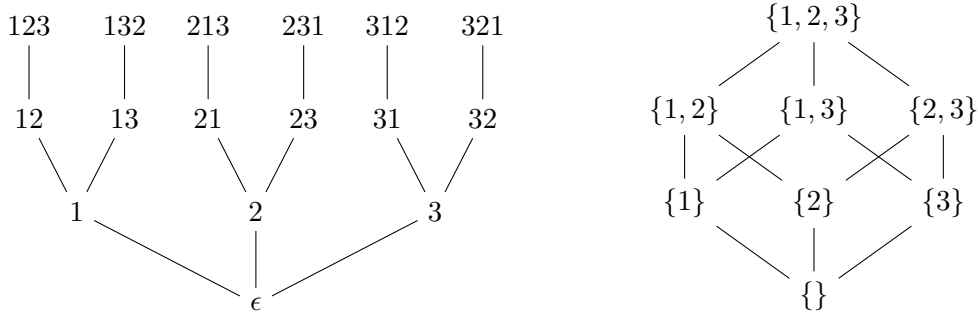
\end{example}

\subsection{Semigroup algebras}
We will be interested in the algebras defined by taking linear combinations of the elements in a semigroup $\SS$. Our notation will follow \cite{am17}.

\begin{definition}[$\B{H}$-basis of a semigroup algebra]
For any semigroup $\SS$, denote by $\kk\SS$ the
$\kk$-algebra with linear basis $\{ \B{H}_x \}_{x \in \SS}$ and multiplication
given by
\[
    \B{H}_x\B{H}_y = \B{H}_{xy}
    \qquad
    \text{for all } x,y \in \SS,
\]
and extended $\kk$-linearly to all of $\kk\SS$.
\end{definition}

For a LRB $\BB$ with support lattice $\lat$,
we have \emph{two} semigroup algebras $\kk \BB$ and $\kk \lat$,
with
\begin{itemize}
    \item
        $\{\B{H}_x : x \in \BB\}$ denoting a basis of $\kk \BB$
        for which $\B{H}_x \B{H}_y = \B{H}_{xy}$ for all $x, y \in \BB$; and

    \item
        $\{\B{H}_X : X \in \lat\}$ denoting a basis of $\kk \lat$
        for which $\B{H}_X \B{H}_Y = \B{H}_{X \vee Y}$ for all $X, Y \in \lat$.
\end{itemize}
These bases are related by the support map $\supp: \BB \to \lat$, which extends
to a surjective morphism of $\kk$-algebras via
\begin{eqnarray*}
    \supp: \kk \BB & \longrightarrow & \kk \lat \\
    \B{H}_{x} & \longmapsto & \B{H}_{\supp(x)}.
\end{eqnarray*}
It turns out that the algebra $\kk \lat$ is split semisimple (i.e.,
isomorphic to a direct product of copies of $\kk$), and that
the kernel of this morphism is the Jacobson radical of $\kk \BB$:
that is, $\ker(\supp) = \rad(\kk\BB)$.

\subsection{Orthogonal idempotents in LRB algebras}\label{ss:CSoPOI-LRB}

Recall that an element $\ee$ of an algebra $A$ is \defn{idempotent} if
$\ee^2 = \ee$, and it is \defn{primitive} if it cannot be written as
$\ee = \ee_1 + \ee_2$ with $\ee_1$ and $\ee_2$ nonzero
idempotents satisfying $\ee_1 \ee_2 = 0 = \ee_2 \ee_1$.
A family of idempotents $\{ \ee_i \}_{i \in I} \subseteq A$
is \defn{mutually orthogonal} if $\ee_i \ee_j = \delta_{ij} \ee_{i}$
for all $i, j \in I$, and it is a \defn{complete system} if $\sum_{i \in I} \ee_i
= 1.$
We are interested in \defn{complete systems of primitive orthogonal
idempotents} (CSoPOI) of semigroup algebras.

There is a well-developed idempotent theory for left regular bands, initiated
by Saliola in \cite{sal2008quiv,saliola2009face,saliola2012eigenvectors} and
further developed by Aguiar and Mahajan in \cite{am17}.
Complete systems of primitive orthogonal idempotents of a LRB algebra $\kk \BB$
are parameterized by \defn{homogeneous sections} of the support map,
by which we mean linear maps
\[ \uu: \kk \lat \to \kk \BB \]
satisfying
\begin{enumerate}
    \item
        the composition
        $\kk \lat \xrightarrow{\uu} \kk \BB  \xrightarrow{\supp} \kk \lat$
        is the identity map on $\kk \lat$ (i.e., $\uu$ is a section
        of $\supp$),

    \item
        for each $X \in \lat$, if we write
        $\uu(\B{H}_X) = \sum_{x \in \BB} c_x \B{H}_x$ with $c_x \in \kk$,
        then $c_y = 0$ if $\supp(y) \neq X$.
\end{enumerate}
To simplify notation, we write $\uu_X:= \uu(\B{H}_X)$ for any $X \in \lat$.
Note that $\sum c_x = 1$ and $\supp(\uu_X) = \B{H}_X$:
\begin{equation*}
    \B{H}_X
    = \supp \left(\uu(\B{H}_X)\right)
    = \supp \Big(\sum_{x \in \BB} c_x \B{H}_x\Big)
    = \sum_{x \in \BB} c_x \supp(\B{H}_x)
    = \Big(\sum_{x} c_x\Big) \B{H}_X.
\end{equation*}

\begin{example}[Uniform section]\label{ex:uniformsection} \rm
Given any LRB, one can define the \defn{uniform section}, which is the
homogeneous section $\uu$ satisfying $c_{x} = c_{y}$ whenever $\supp(x)
= \supp(y)$. For the free LRB $\mathscr{F}_2$, the uniform section is
characterized by
\[
\uu_{\emptyset} = \B{H}_{\epsilon}, \qquad
\uu_{\{ 1 \} } = \B{H}_{1},\qquad
\uu_{\{ 2 \}} = \B{H}_{2}, \qquad
\uu_{\{1,2 \}} = \tfrac{1}{2} \B{H}_{12} + \tfrac{1}{2}\B{H}_{21},
\]
where, as in \cref{ex:freelrb}, we denote the identity element of $\mathscr{F}_2$ by $\epsilon$.
\end{example}

Saliola showed in \cite{saliola07quiverLRB} that any homogeneous section $\uu$
can be used to define a complete system of primitive orthogonal idempotents
$\{\ee^\circ_X \}_{X \in \lat}$ for a LRB algebra, as follows. Aguiar--Mahajan
proved in \cite[Prop. 11.9]{am17} that all such families $\{\ee^\circ_X \}_{X \in
\lat}$ for a LRB algebra arise in this way.
Thus, the data of a homogeneous section $\{\uu_{X}\}_{X \in \lat}$ is equivalent to that of the
family $\{ \ee^\circ_X \}_{X \in \lat}$. More precisely:

\begin{theorem}[{\cite[Theorem~4.2]{saliola07quiverLRB} and \cite[Prop. 11.9]{am17}}] \label{thm:CSoPOIforLRB}
    Let $\BB$ be a finite LRB, $\lat$ its support lattice, and $\kk$ a field.
    Let $\uu: \kk\lat \xrightarrow{} \kk\BB$ be a homogeneous section of the support map
    $\supp: \kk\BB \to \kk\lat$.
    Define elements $\ee^\circ_X$, one for each $X \in \lat$, by
    \begin{equation}\label{eq:saliolarecursion}
    \ee^{\circ}_{X}:= \uu_{X} - \sum_{\substack{Y \in \lat \\ Y > X}} \uu_{X}\cdot \ee^{\circ}_{Y}. \end{equation}
    Then $\{ \ee^\circ_X \}_{X \in \lat}$ is a complete system of primitive orthogonal idempotents for $\kk \BB$.
    Conversely, every complete system of primitive orthogonal idempotents for $\kk\BB$ arises in this way.
\end{theorem}

Moreover, these idempotents satisfy the following very useful property.

\begin{lemma}[{\cite[Lemma 5.1]{saliola2009face}}]
    \label{lem:Saliola's}
    For all $x \in \BB$ and $Y \in \lat$, if $\supp(x) \not\leq Y$, then $\B{H}_x \ee^\circ_Y = 0$.
\end{lemma}

\begin{example}
We again consider the free LRB $\mathscr{F}_2$ and construct idempotents using the uniform section described in \cref{ex:uniformsection}.
As illustrated in \cref{fig:posets-for-free-lrb-3}, the support lattice of $\mathscr{F}_2$ is the Boolean lattice, ordered by inclusion.
Thus the maximum element is $\{ 1,2 \}$, and so
\[ \ee^\circ_{\{ 1,2 \}} = \uu_{\{ 1,2 \}} = \tfrac{1}{2} \B{H}_{12} + \tfrac{1}{2} \B{H}_{21}.\]
Next, we have
\begin{align*}
    \ee^\circ_{\{ 1 \}} =  \uu_{\{ 1 \}} -  \uu_{\{ 1 \}} \cdot \ee^\circ_{\{ 1,2 \}}
    = \B{H}_{1} - \B{H}_1 \left( \tfrac{1}{2} \B{H}_{12} + \tfrac{1}{2} \B{H}_{21} \right)
    = \B{H}_1 - \B{H}_{12}, 
\end{align*}
since $\B{H}_1 \cdot \B{H}_{12}  = \B{H}_1 \cdot \B{H}_{21}  = \B{H}_{12}$ in $\kk \mathscr{F}_2$.
By a similar computation,
\[ \ee^\circ_{\{ 2 \}} = \B{H}_2 - \B{H}_{21}.\]
Finally,
\begin{align*}
    \ee^\circ_{\emptyset} &= \B{H}_{\epsilon} - \uu_{\emptyset} \cdot \ee^\circ_{\{ 1, 2 \}} -  \uu_{\emptyset} \cdot \ee^\circ_{\{ 1 \}} -  \uu_{\emptyset} \cdot \ee^\circ_{\{ 2 \}} \\
    &= \B{H}_{\epsilon} -  \ee^\circ_{\{ 1, 2 \}} -   \ee^\circ_{\{ 1 \}} -  \ee^\circ_{\{ 2 \}} \\
    &= \B{H}_{\epsilon}- \left( \tfrac{1}{2} \B{H}_{12} + \tfrac{1}{2} \B{H}_{21} \right) - \left( \B{H}_1 - \B{H}_{12} \right) - \left( \B{H}_2 - \B{H}_{21} \right)\\
    &= \B{H}_{\epsilon} - \B{H}_{1} - \B{H}_{2} + \tfrac{1}{2} \B{H}_{12} + \tfrac{1}{2} \B{H}_{21}.
\end{align*}
One can check directly that the $\ee^\circ_{J}$ for $J \subseteq \{ 1,2 \}$ do indeed form a complete system of primitive orthogonal idempotents
for $\kk \mathscr{F}_2$.
\end{example}

\subsection{$\B{Q}$-basis and semisimplicity of $\C \lat$}
\label{sec:LRB-Q-basis}

There is a second basis for $\kk\BB$ called the \defn{$\B{Q}$-basis}, defined as follows.
\begin{definition}\label{def:qbasis}
    \label{prop:LRB-Q-Basis}
    Let $\BB$ be a LRB and $\{\ee^\circ_X\}_{X \in \lat}$
    a complete system of primitive orthogonal idempotents for
    $\kk \BB$.
    For each $x \in \BB$, define elements in $\kk \BB$ by
    \begin{equation} \label{eq:LRB-Q-basis}
        \B{Q}^\circ_x = \B{H}_{x} \ee^\circ_{\supp(x)}.
    \end{equation}
By \cite[Proposition 5.4]{saliola2009face},
the set
$\{\B{Q}^\circ_x : x \in \BB\}$ is a basis of $\kk \BB$
consisting of primitive idempotents.
\end{definition}

Since $\lat$ is also a LRB, \cref{prop:LRB-Q-Basis} defines a $\B{Q}$-basis of $\kk\lat$,
which is denoted by $\{\B{Q}^\circ_{X}\}_{X \in \lat}$.
Tracing through the construction, one finds that
\begin{equation*}
    \B{Q}^\circ_X = \sum_{Y \geq X} \mu(X,Y) \, \B{H}_Y \text{ for all } X \in \lat,
\end{equation*}
where $\mu$ is the Möbius function of the support lattice $\lat$,
and that
\[
    \B{Q}^\circ_{X} \B{Q}^\circ_{Y} =
    \begin{cases}
        \B{Q}^\circ_{X}, & \text{if } X = Y, \\
        0,               & \text{otherwise.}
    \end{cases}
\]
The $\B{Q}$-bases of $\kk\BB$ and $\kk\lat$ are related by the support map in
the following way:
\begin{equation} \label{support-of-Q-basis}
    \supp(\B{Q}^\circ_{x}) = \B{Q}^\circ_{\supp(x)}.
\end{equation}

Thus, $\{\B{Q}^\circ_{X}\}_{X \in \lat}$ is both a basis for $\kk\lat$
\emph{and} a complete system of primitive orthogonal idempotents for $\kk\lat$.
Consequently, $\kk\lat$ admits the following decomposition as a direct sum of
algebras:
\begin{equation*}
    \kk\lat = \bigoplus_{X \in \lat} \mathrm{span}(\B{Q}^\circ_{X}) \cong \bigoplus_{X \in \lat} \kk.
\end{equation*}
It thus follows that $\kk\lat$ is semisimple and is in fact the semisimple quotient
of $\kk\BB$.

\section{Left Regular Bands of Groups}\label{sec:LRBG}

Next, we will formally introduce left regular bands of groups. Recall that informally, these semigroups are obtained from a LRB $\BB$ by attaching a finite group $G_x$ to each element $x \in \BB$ according to certain compatibility conditions.
We will then extend several of constructions from LRBs to this context.

\subsection{Main definition}

We start by recalling some standard notions of the theory of finite semigroups.
We refer the reader to \cite{clifford} or \cite[Appendix A]{RhSt09Semigroups} for more information.

Let $\SS$ be a semigroup.
\begin{itemize}
    \item
        The collection of idempotents of $\SS$ is denoted by $E(\SS)$.
        That is, $E(\SS) := \{ x \in \SS : x^2 = x \}$.
    \item
        For each idempotent $x \in E(\SS)$, the set $x \SS x$ forms a submonoid of $\SS$ with identity element~$x$.
        The group of units of this subsemigroup $G_x := (x \SS x)^\times$ is called the \defn{maximal semigroup of $\SS$ at $x$}.
        It is maximal with respect to inclusion among the groups in $\SS$ that contain $x$.
    \item
        If $\SS$ is a finite semigroup, for each element $s \in \SS$ there exists
        % a unique positive power of $s$ that is an idempotent.
        % We let $s^\omega \in E(\SS)$ denote this idempotent.
        a unique idempotent element of $\SS$ that is a positive power of $s$; see for instance \cite[Corollary 1.2]{steinberg16monoids}.
        We let $s^\omega \in E(\SS)$ denote this idempotent.
\end{itemize}

\begin{definition} \label{defn:LRBG}
    A finite semigroup $\SS$ is a \defn{left regular band of groups (LRBG)} if for all $s,t \in \SS$:
    \begin{align}
        \label{eq:lrbg1}\tag{LRBG1}     s^\omega s = s,         \\
        \label{eq:lrbg2}\tag{LRBG2}     s t s^\omega = s t.
    \end{align}
\end{definition}

\begin{remark}
    We follow Margolis and Steinberg \cite{ms11HomologySemigroups} in defining
    a LRBG as a semigroup satisfying \cref{eq:lrbg1,eq:lrbg2}.
    However, there is also a notion of a \emph{band of groups} in the semigroup
    literature \cite{clifford, cp61semigroupsI},
    and it should be noted that \emph{left regular bands of groups} are not
    \emph{bands of groups}.
    More details can be found in \cref{s:appendix}, which characterizes the
    LRBGs that are bands of groups.
\end{remark}

We begin by proving that the set of idempotents $E(\SS)$ of a LRBG $\SS$ is a LRB.

\begin{lemma}
    \label{E(S)-is-a-LRB}
    Let $\SS$ be a LRBG and let $E(\SS)$ denote the set of idempotents of $\SS$.
    Then $E(\SS)$ is a subsemigroup of $\SS$. Moreover, $E(\SS)$ is a LRB.
\end{lemma}

\begin{proof}
    We first prove $E(\SS)$ is a subsemigroup.
    Let $x, y \in E(\SS)$.
    Then $x^\omega = x$ and $y^\omega = y$,
    and so
    \begin{equation*}
        (x y)^2
        = x y x y
        = x y x^\omega y^\omega
        \overset{\cref{eq:lrbg2}}{=}
        x y y^\omega
        \overset{\cref{eq:lrbg1}}{=}
        x y.
    \end{equation*}
    It follows that $x y \in E(\SS)$.
    To see that $E(\SS)$ is a LRB, note that (LRB1) holds because every element
    of $E(\SS)$ is idempotent, and (LRB2) follows by setting $s = x$ and $t
    = y$ in \cref{eq:lrbg2}.
\end{proof}

The next result identifies the maximal subgroups $G_{x}$ of $\SS$ at $x
\in E(\SS)$ with the set of elements in $\SS$ that are mapped to $x$ by the
function $s \mapsto s^\omega$. This implies that $\SS$ is a disjoint union of groups.

\begin{lemma}
    \label{preimages-of-omega-map-are-maximal-subgroups}
    Let $\SS$ be a LRBG and define $\varphi: \SS \xrightarrow{} E(\SS)$ by
    $\varphi(s) = s^\omega$.
    Then for each $x \in E(\SS)$, the preimage
    $\varphi^{-1}(\{x\}) = \{ s \in \SS : s^\omega = x \}$ is the maximal
    subgroup $G_x$ of $\SS$ at $x$.
    Consequently,
    \begin{equation}
        \label{LRBG-disjoint-union-of-groups}
        \SS = \bigsqcup_{x \in E(\SS)} G_x,
        \qquad\text{where~} G_x = \{s \in \SS: s^\omega = x\}.
    \end{equation}
\end{lemma}

\begin{proof}
    We first prove $G_x \subseteq \{s \in \SS : s^\omega = x\}$.
    If $s$ is an element of the finite group $G_x$,
    then some positive power of $s$ coincides with the
    identity element of $G_x$, which is $x$.
    Hence, $s^\omega = x$.

    Conversely, suppose $s^\omega = x$.
    By \cref{eq:lrbg1}, $x s = s^\omega s = s$ and
    $s x = s s^\omega = s$, which implies $s$
    belongs to the submonoid $x \SS x$ with identity element $x$.
    Moreover, $s$ is invertible in $x \SS x$ since if $n$ is a positive integer
    such that $s^n = s^\omega = x$, then the inverse of $s$ is $s^{n-1}$.
    Hence, $s \in G_x$.
\end{proof}

\begin{remark}
    There are LRBGs for which the map $s \xmapsto{} s^\omega$
    is not necessarily a semigroup morphism.
    We call LRBGs for which $s \xmapsto{} s^\omega$
    is a semigroup morphism \defn{strict LRBGs}; see \cref{s:appendix} for
    a detailed discussion and examples.
\end{remark}

\subsection{Main examples: Hsiao's LRBGs}\label{ss:Hsiao's}
The following LRBGs were introduced by Hsiao~\cite{hsiao2007semigroup}.
Note that we use the name \defn{(set) composition} for what Hsiao calls \defn{ordered partition (of a set)}.

\subsubsection{$G$-compositions}

Fix a finite group $G$ and a positive integer $n$.
A \defn{$G$-composition of $[n]$} is a tuple
\begin{equation*}
    s = \big( (S_1,g_1),\dots,(S_k,g_k) \big),
\end{equation*}
where $(S_1,\dots,S_k) \in \Sigma_n$ is a composition of $[n]$ and $g_i \in G$ for all $i$.
We say that $(S_1,\dots,S_k)$ is the \defn{underlying composition} of $s$ and denote it by $\abs{s}$. 
Let $\hsiao$ denote the collection of all $G$-compositions of $[n]$.
The product of two $G$-compositions
\[   s = \big( (S_1,g_1),\dots,(S_k,g_k) \big) \qqand t = \big( (T_1,h_1),\dots,(T_\ell,h_\ell) \big)    \]
is
\[
    st := \big(
        (S_1 \cap T_1 , h_1 g_1),\dots,(S_1 \cap T_\ell, h_\ell g_1),
        (S_2 \cap T_1 , h_1 g_2),
        \dots,
        (S_k \cap T_\ell , h_\ell g_k)
    \big)^\dagger,
\]
where the superscript $\dagger$ denotes that the pairs with an empty intersection have been omitted.
Similar to the case of compositions of $[n]$, we simplify notation by concatenating the elements of the blocks, writing group elements as a superscript, and replacing commas by vertical bars.
For example, $( \{1,3,4\} , g_1) , ( \{2,5\}, g_2) \big) $ will be written as $\big(134^{g_1} \mid 25^{g_2}\big)$, and
\[
    \big(134^{g_1} \mid 25^{g_2} \big)  \big( 12^{h_1} \mid 5^{h_2} \mid 34^{h_3} \big) = \big( 1^{h_1 g_1} \mid 34^{h_3 g_1} \mid 2^{h_1 g_2} \mid 5^{h_2 g_2}  \big).
\]
With this product, $\hsiao$ is a LRBG \cite{hsiao2007semigroup}.

Observe that if we omit the group elements $g_i,h_j$, this product agrees with the semigroup structure on $\Sigma_n$ from \cref{ex:braidarrangement}.
In particular, the map $\hsiao \to \Sigma_n : s \mapsto \abs{s}$ sending a $G$-composition to its underlying composition is a semigroup morphism.

\subsubsection{$G$-partitions}
We define the semigroup of $G$-partitions in an identical way.

    We again fix a finite group $G$ and a positive integer $n$.
    A \defn{$G$-partition of $[n]$} is a set 
    \begin{equation*}
        \{ (S_1,g_1),\dots,(S_k,g_k) \},
    \end{equation*}
    where $\{ S_1,\dots,S_k \} \in \Pi_n$ is a partition of $[n]$ and $g_i \in G$ for all $i$.
    We let $\gsetpartition$ denote the collection of $G$-partitions of $[n]$. Again, we omit set brackets and replacing commas with bars.

    The product of two $G$-partitions and the underlying partition $\abs{S} \in \Pi_n$ of a $G$-partition $S \in \gsetpartition$ are defined in an analogous manner.
    For example,

    \[
        \big \{  134^{g_1} \mid 25^{g_2} \big\}  \big \{  12^{h_1} \mid 5^{h_2} \mid 34^{h_3}  \big \} = \big\{ 1^{h_1 g_1} \mid 2^{h_1 g_2} \mid 34^{h_3 g_1} \mid 5^{h_2 g_2}   \big \} .
    \]

Observe again that omitting the group elements in $\gsetpartition$ recovers the semigroup $\Pi_n$ of partitions of~$[n]$
defined in \cref{ex:partitionlattice}.

\subsubsection{Relationship between $\hsiao$, $\gsetpartition$, $\Sigma_n$ and $\Pi_n$}

The surjection $\hsiao \to \gsetpartition$ that forgets the order of a $G$-composition is a semigroup morphism.
Following \cref{ex:partitionlattice}, we let $\supp$ denote this map.
Then we have the following commutative diagram of semigroup morphisms:
\begin{equation}\label{eq:hsiao-quotients}
    \begin{tikzcd}[column sep=4em]
            \hsiao \arrow[d, "\supp"', two heads] \arrow[r, "\abs{\cdot}", two heads] & \Sigma_n \arrow[d, "\supp", two heads] \\
            \gsetpartition \arrow[r, "\abs{\cdot}"', two heads]                        & \Pi_n
    \end{tikzcd}
\end{equation}
The idempotents of $\hsiao$ are precisely the $G$-compositions of the form $\big( (S_1,1_G),\dots,(S_k,1_G) \big)$, where $1_G$ is the identity of the group $G$;
and a similar statement holds for $\gsetpartition$.
We thus have natural identifications $E(\hsiao) \cong \Sigma_n$ and $E(\gsetpartition) \cong \Pi_n$.
In particular, we identify $s^\omega$ with $\abs{s}$ for all $s \in \hsiao$.

\subsection{Support map}\label{ss:sLRBGrelations}

We extend the relationship depicted in \cref{eq:hsiao-quotients} to all LRBGs.
More precisely, given a LRBG $\SS$, we construct a semigroup quotient $\SS/\mathord{\sim}$
such that the following diagram commutes
\begin{equation}\label{eq:LRBG-quotients-commuative-diagram}
    \begin{tikzcd}
        \SS \arrow[d, two heads, "\supp"'] \arrow[rr, two heads] &[2em] &[-3em] E(\SS) \arrow[d, "\supp", two heads] \\
        \SS/\mathord{\sim} \arrow[r, two heads]    & E(\SS/\mathord{\sim}) = & E(\SS)/\mathord{\sim}
        \end{tikzcd}
\end{equation}
where $E(\SS) \xrightarrow{} E(\SS)/\mathord{\sim}$ is the support map for the LRB $E(\SS)$,
and the horizontal maps are $s \mapsto s^\omega$.

\begin{remark}
    Even though the map $s \xmapsto{} s^\omega$ may fail to be
    a morphism, we always have that $\supp(s)^\omega = \supp(s^\omega)$
    since $\supp$ is a semigroup morphism.
    Also, we will see in \cref{LRBG-supp-and-omega}
    that even though $(st)^\omega$ and $s^\omega t^\omega$
    are not necessarily equal, we always have
    $\supp((st)^\omega) = \supp(s^\omega) \supp(t^\omega)$.
\end{remark}

We begin by extending the relations $\leq$ and $\sim$, originally defined only for LRBs, to LRBGs.

\begin{definition}\label{def:sLRBGrelations}
    For $s,t \in \SS$ define
    \begin{align}
        \label{eq:poset-S} s \leq t         \qqiff &          s t^\omega = t;\\
        \label{eq:equiv-S} s \sim t         \qqiff &          s^\omega t = s \text{ and } t^\omega s = t.
    \end{align}
\end{definition}

Observe that for elements in $E(\SS)$, the relations above coincide with those in \cref{ss:LRBrelations}.
Moreover, if $s \sim x$ for some $x \in E(\SS)$, then $s \in E(\SS)$ since $s^\omega x = s$ and $E(\SS)$ is closed under product.

\begin{prop}
    The relations $\leq$ and $\sim$ in \cref{def:sLRBGrelations} are a partial order and a semigroup congruence, respectively.
\end{prop}

\begin{proof}
    We first prove that $\leq$ is a partial order.
    Reflexivity follows from \cref{eq:lrbg1}.
    Suppose $s \leq t$ and $t \leq s$, then an application of \cref{eq:lrbg2} shows that $\leq$ is antisymmetric:
    \[
        s = t s^\omega = (s t^\omega) s^\omega = s t^\omega = t.
    \]
    Now suppose $s \leq t$ and $t \leq u$, then
    \[
        s u^\omega = s (t u^\omega)^\omega = s t^\omega (t u^\omega)^\omega = t (t u^\omega)^\omega = t u^\omega = u,
    \]
    and so $s \leq u$. Thus, $\leq$ is transitive and therefore a partial order.

    We now show that $\sim$ is an equivalence relation.
    Symmetry is clear from the definition, and reflexivity follows again from \cref{eq:lrbg1}.
    Suppose $s \sim t$ and $t \sim u$, then
    \[
        s^\omega u = (s^\omega t)^\omega u = s^\omega t^\omega u = s^\omega t = s,
    \]
    and similarly $u^\omega s = u$. Thus, $\sim$ is transitive and therefore an equivalence relation.

    Finally, we show that $\sim$ is a semigroup congruence:
    i.e., that $su \sim tu$ and $us \sim ut$ for all $s,t,u \in \SS$ with $s \sim t$.
    Let $s,t,u \in \SS$ and suppose $s \sim t$.
    Then,
    \begin{equation*}
        (s u)^\omega t u                                    \overset{(s \sim t)}{=} 
        ( (s^\omega t) u )^\omega ( t^\omega s ) u          \overset{\cref{eq:lrbg2}}{=}
        ( (s^\omega t) u )^\omega ( s ) u                   \overset{(s \sim t)}{=} 
        (s u)^\omega s u                                    \overset{\cref{eq:lrbg1}}{=}
        s u.
    \end{equation*}
    Exchanging $s$ and $t$ above, we obtain $(t u)^\omega s u = t u$ and deduce that $s u \sim t u$.
    The same trick shows that $u s \sim u t$:
    \[
        (u s)^\omega u t = (u s^\omega t)^\omega u (t^\omega s) = (u s^\omega t)^\omega u (s) = (u s)^\omega u s = u s.
    \]
    Therefore, $\sim$ is a semigroup congruence.
\end{proof}

\begin{example}\label{ex:hsiao-order}
    Let $\hsiao$ be the semigroup of $G$-compositions in \cref{ss:Hsiao's}.
    Then,
    \[
        \big( (S_1,g_1),\dots,(S_k,g_k) \big) \leq \big( (T_1,h_1),\dots,(T_\ell,h_\ell) \big)
    \]
    if and only if 
    \begin{enumerate}
        \item the composition $(S_1,\dots,S_k)$ is refined by $(T_1,\dots,T_\ell)$, and
        \item whenever two blocks $T_i$ and $T_j$ of the second composition are contained in the same block $S_m$ of the first composition, we have $h_i = h_j = g_m$.
    \end{enumerate}

    On the other hand, we have
    \[
        \big( (S_1,g_1),\dots,(S_k,g_k) \big) \sim \big( (T_1,h_1),\dots,(T_\ell,h_\ell) \big)
    \]
    if and only if $k = \ell$ and there is a permutation $\sigma \in \mathfrak{S}_k$ such that $(S_i , g_i) = (T_{\sigma(i)} , h_{\sigma(i)})$ for all $i$.
    Therefore, $\hsiao/\mathord{\sim} \cong \gsetpartition$.
    See \cref{fig:hsiao-C2-relations} for an example with $n=2$ and $G = \ztwo = \{\pm 1\}$, where $\ztwo$ is the cyclic group of order two.

    \begin{figure}[ht]
        \centering
        \begin{tikzpicture}[yscale=0.7, xscale=0.8]
            % \node (node_0) at (-50bp,  0bp) {$\big( (\{1,2\}, 1) \big)$};
            \node (node_0) at (-25bp,  0bp) {$\big( 12^+ \big)$};
            % \node (node_1) at ( 50bp,  0bp) {$\big( (\{1,2\},-1) \big)$};
            \node (node_1) at ( 25bp,  0bp) {$\big( 12^- \big)$};

            % \node (node_00l) at (-200bp,  80bp) {$\big( (\{1\}, 1),(\{2\}, 1) \big)$};
            \node (node_00l) at (-150bp,  80bp) {$\big( 1^+ | 2^+ \big)$};
            % \node (node_11l) at ( -70bp,  80bp) {$\big( (\{1\},-1),(\{2\},-1) \big)$};
            \node (node_11l) at ( -75bp,  80bp) {$\big( 1^- | 2^- \big)$};
            % \node (node_01l) at (-200bp, 130bp) {$\big( (\{1\}, 1),(\{2\},-1) \big)$};
            \node (node_01l) at (-150bp, 130bp) {$\big( 1^+ | 2^- \big)$};
            % \node (node_10l) at ( -70bp, 130bp) {$\big( (\{1\},-1),(\{2\}, 1) \big)$};
            \node (node_10l) at ( -75bp, 130bp) {$\big( 1^- | 2^+ \big)$};

            % \node (node_00r) at (  70bp,  80bp) {$\big( (\{2\}, 1),(\{1\}, 1) \big)$};
            \node (node_00r) at (  75bp,  80bp) {$\big( 2^+ | 1^+ \big)$};
            % \node (node_11r) at ( 200bp,  80bp) {$\big( (\{2\},-1),(\{1\},-1) \big)$};
            \node (node_11r) at ( 150bp,  80bp) {$\big( 2^- | 1^- \big)$};
            % \node (node_01r) at (  70bp, 130bp) {$\big( (\{2\},-1),(\{1\}, 1) \big)$};
            \node (node_01r) at (  75bp, 130bp) {$\big( 2^- | 1^+ \big)$};
            % \node (node_10r) at ( 200bp, 130bp) {$\big( (\{2\}, 1),(\{1\},-1) \big)$};
            \node (node_10r) at ( 150bp, 130bp) {$\big( 2^+ | 1^- \big)$};

            \draw [blue] (-185bp, 60bp) rectangle (-35bp,150bp); \node [blue] at (-180bp, 45bp) {\small$G_{(1|2)} \cong \ztwo \times \ztwo$};
            \draw [blue] ( 185bp, 60bp) rectangle ( 35bp,150bp); \node [blue] at ( 180bp, 45bp) {\small$\ztwo \times \ztwo \cong G_{(2|1)}$};
            \draw [blue] ( -55bp,-20bp) rectangle ( 55bp, 20bp); \node [blue] at (     0,-35bp) {\small$G_{(12)} \cong \ztwo$};

            \draw [black,thick] (node_0) -- (node_00l);
            \draw [black,thick] (node_0) -- (node_00r);
            \draw [black,thick] (node_1) -- (node_11l);
            \draw [black,thick] (node_1) -- (node_11r);
            \draw [red,dashed] (node_00l) to [bend left=20] (node_00r);
            \draw [red,dashed] (node_11l) to [bend left=20] (node_11r);
            \draw [red,dashed] (node_01l) to [bend left=20] (node_01r);
            \draw [red,dashed] (node_10l) to [bend left=20] (node_10r);
        \end{tikzpicture}
        \caption{The ten elements of the semigroup $\Sigma_2[\ztwo]$.
        We write $+$ or $-$ instead of $+1$ or $-1$ to simplify notation.
        The solid black lines correspond to the partial order $\leq$, the four elements at the top are incomparable to any other.
        The dashed red lines correspond to the congruence $\sim$.
        The boxes group elements according to the maximal subgroup they belong to.}
        \label{fig:hsiao-C2-relations}
    \end{figure}
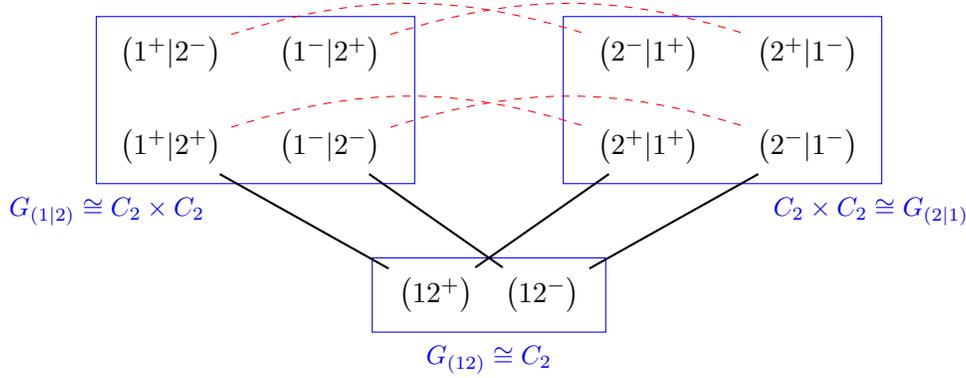
\end{example}

Let $\TT$ be the semigroup quotient $\SS/\mathord{\sim}$ and let $\supp: \SS
\to \TT$ denote the quotient map, which we call the \defn{support map} of
$\SS$. Then $\TT$ is a LRBG with $E(\TT) = E(\SS)/\mathord{\sim}$.
In the next result, we show $\TT$ that it is a \defn{semilattice of groups},
i.e., a regular semigroup with central idempotents \cite{cp61semigroupsI}.
An analysis of the decomposition of $\TT$ into a disjoint union of groups will
be presented in \cref{decomposition-into-maximal-subgroups}.

\begin{prop}[Support map for LRBGs]
    \label{prop:T-cetr-idem}
    Let $\SS$ be a LRBG, let $\TT$ be the semigroup quotient $\SS/\mathord{\sim}$,
    and let $\supp: \SS \longrightarrow  \TT$ denote the quotient map.
    Then
    \begin{enumerate}[start=0]
        \item\label{item-ET=suppES}
            $E(\TT) = \supp(E(\SS))$ is the support lattice of the LRB $E(\SS)$;

        \item\label{item-T-has-central-idempotents}
            $\TT$ is a semilattice of groups; in particular $\TT$ has central idempotents;

        \item\label{item-lrbg-semigroup-map}
            $\supp: \SS \to \TT$ is a map of semigroups
            (i.e., $\supp(st) = \supp(s) \supp(t)$);

        \item\label{item-lrbg-poset-map}
            $\supp: \SS \xrightarrow{} \TT$ is a map of posets
            (i.e., $s \leq t$ in $\SS$ implies $\supp(s) \leq \supp(t)$ in $\TT$).
    \end{enumerate}
\end{prop}
\begin{proof}
    \eqref{item-lrbg-semigroup-map} Follows immediately from the fact that
    $\TT$ is a semigroup quotient of $\SS$.

    \eqref{item-ET=suppES}
    Suppose $X \in E(\TT)$. There exists $s \in \SS$ such that $\supp(s) = X$.
    Then $s \sim s^2$ since $X$ is an idempotent.
    From the definition of $\sim$, we have $s^\omega s^2 = s$, which implies
    $s^2 = s$ because $s^\omega s = s$.

    \eqref{item-T-has-central-idempotents}
    By \cite{cp61semigroupsI}, $\TT$ is a semilattice of groups if and only if
    it is regular and has central idempotents.
    That $\TT$ is regular follows from the fact that it is a LRBG.
    We now show it has central idempotents.

    Let $S \in \TT$ and $X \in E(\TT)$, we want to prove that $SX = XS$.
    Take $s \in \SS$ and $x \in E(\SS)$ with $\supp(s) = S$ and $\supp(x) = X$.
    It is enough to show that $sx \sim xs$.
    We have
    \[
        (sx)^\omega xs = \underbrace{(sx)(sx) \dots (sx)}_{\omega\text{ times}} xs = s x s^\omega = s x,
    \]
    where we first use \cref{eq:lrbg2} to remove all repetitions of $x$ and then we use it again to remove $s^\omega$.
    Similarly,
    \[
        (xs)^\omega sx = x s^\omega s x = x s^\omega s = x s,
    \]
    where the last equality follows from \cref{eq:lrbg1}.

    \eqref{item-lrbg-poset-map}
    Let $s, t \in \SS$ and suppose that $s \leq t$.
    Then $s t^\omega = t$, and so
    $\supp(s) \supp(t)^\omega = \supp(s t^\omega) = \supp(t)$.
    Hence, $\supp(s) \leq \supp(t)$ in $\TT$.
\end{proof}

\begin{cor}
    \label{LRBG-supp-and-omega}
    Let $\SS$ be a LRBG, $\TT = \SS/\mathord{\sim}$,
    and let $\supp$ denote the quotient map $\supp: \SS \to \TT$.
    Then for all $s, t \in \SS$,
    \begin{equation*}
        \supp\left((st)^\omega\right)
        = \supp(s^\omega) \supp(t^\omega)
        = \supp(s^\omega t^\omega).
    \end{equation*}
    In particular, $\TT \to E(\TT) : T \mapsto T^\omega$ is a semigroup
    morphism; i.e., $(st)^\omega \sim s^\omega t^\omega$ for all $s, t \in \SS$.
\end{cor}

\begin{proof}
    Since $\supp$ is a morphism of semigroups,
    it suffices to prove
    $\supp\left((st)^\omega\right)
    = \supp(s^\omega t^\omega)$,
    which is equivalent to
    $(st)^\omega \sim s^\omega t^\omega$.
    First, note that
    \begin{equation*}
        (st)^\omega (s^\omega t^\omega)
        = (st)^{\omega - 1} st s^\omega t^\omega
        \overset{\cref{eq:lrbg2}}{=}
        (st)^{\omega - 1} st s^\omega
        \overset{\cref{eq:lrbg2}}{=}
        (st)^{\omega - 1} st
        = (st)^\omega.
    \end{equation*}
    By \cref{prop:T-cetr-idem}, $\TT$ has central idempotents, which implies
    $sx \sim xs$ for all $s \in \SS$ and $x \in E(\SS)$.
    Applying this several times we have,
    \begin{equation*}
        \begin{aligned}[b]
            (st)^\omega
            & =
            (st)^\omega s^\omega t^\omega
            \sim
            t^\omega (st)^\omega s^\omega
            & (t^\omega \in E(\SS))
            \\
            & =
            t^{\omega-1} t (st)^\omega s^\omega
            \sim
            t^{\omega-1} s^\omega t (st)^\omega
            & (s^\omega \in E(\SS))
            \\
            & =
            t^{\omega-1} s^{\omega-1} (st) (st)^\omega
            \overset{\cref{eq:lrbg1}}{=}
            t^{\omega-1} s^{\omega-1} (st)
            =
            t^{\omega-1} s^{\omega} t
            \sim
            s^{\omega} t^{\omega}
            & (s^\omega \in E(\SS))
        \end{aligned}
            \qedhere
    \end{equation*}
\end{proof}

\subsection{Decomposition into maximal subgroups}
\label{decomposition-into-maximal-subgroups}

Let $\SS$ be a LRBG and $\supp: \SS \to \TT$ its support map.
By \cref{preimages-of-omega-map-are-maximal-subgroups},
both $\SS$ and $\TT$ are disjoint unions of their maximal subgroups:
\begin{equation*}
    \SS = \bigsqcup_{x \in E(\SS)} G_x
    \qqand
    \TT = \bigsqcup_{X \in E(\TT)} G_X,
\end{equation*}
where
\begin{equation*}
    G_{x} = \{ s \in \SS : s^\omega = x \}
    \qqand
    G_{X} = \{ T \in \TT : T^\omega = X \}.
\end{equation*}
We will see that $G_{x} \cong G_{x'}$ whenever $\supp(x) = \supp(x')$,
and that both are isomorphic to $G_{\supp(x)}$.

\begin{prop}
    \label{left-multiplication-induces-group-morphism}
    Let $x, y \in E(\SS)$. Left multiplication by $y$ induces a group morphism
    \begin{eqnarray*}
        \lambda_{y,x} : G_{x} & \to & G_{yx} \\
        \lambda_{y,x}(s) & = & ys,
    \end{eqnarray*}
    If $x, x' \in E(\SS)$ have the same support, then
    $\lambda_{x,x'} : G_{x'} \to G_x$ is an isomorphism with inverse $\lambda_{x,x'}$.
\end{prop}

\begin{proof}
    Let $x, y \in E(\SS)$.
    The fact that $\lambda_{y,x}$ is a group morphism follows from property \cref{eq:lrbg2}:
    indeed, for all $s \in G_x$,
    \[
        (y s)^\omega = y s y s \cdots y s \overset{\cref{eq:lrbg2}}{=} y s^\omega = y x
    \]
    so that $y G_x \subseteq G_{yx}$; and for all $s,t \in G_x$,
    \[
        (y s) (y t) \overset{\cref{eq:lrbg2}}{=} y s t = y (st).
    \]

    If $x,x' \in E(\SS)$ have the same support,
    then $x x' = x$ and $x' x = x'$ and so
    \[
        (\lambda_{x,x'} \circ \lambda_{x',x})(s) =
        x x' s = x s = s \text{ for all } s \in G_x
    \]
    and
    \[
        (\lambda_{x',x} \circ \lambda_{x,x'})(s') =
        x' x s' = x' s' = s' \text{ for all } s' \in G_{x'},
    \]
    which proves that $\lambda_{x,x'}$ is the inverse of $\lambda_{x',x}$.
    Thus, $G_{x} \cong G_{x'}$ whenever $\supp(x) = \supp(x')$.
\end{proof}

\begin{remark}
    In general, \textbf{right} multiplication by $y$ \textbf{does not} give a group morphism $G_x \to G_{xy}$.
    In fact $G_x \, y = \{ sy : s \in G_x \}$ might fail to be fully contained in $G_{xy}$.
    See \cref{ex:not-strict}.
\end{remark}

\begin{prop}
    For each $x \in E(\SS)$, the restriction of $\supp: \SS \to \TT$ to $G_{x}$
    induces an isomorphism:
    \begin{equation}
        \supp\big|_{G_x} : G_x \xrightarrow{~\sim~} G_{\supp(x)}.
    \end{equation}
\end{prop}

\begin{proof}
    Since $\supp$ is a semigroup morphism,
    we need only prove
    $\supp(G_{x})$ is contained in $G_{\supp(x)}$
    and that $\supp$ is injective when restricted to $G_{x}$.

    If $s \in G_{x}$, then $s^\omega = x$ and
    $\supp(s)^\omega = \supp(s^\omega) = \supp(x)$,
    which proves $\supp(s) \in G_{\supp(x)}$.

    Suppose $s, t \in G_{x}$ with $\supp(s) = \supp(t)$.
    Then $t = s^\omega t$ because $s^\omega = x$ is the identity element of $G_{x}$.
    But $s^\omega t = s$ because $s \sim t$. Hence, $\supp|_{G_{x}}$ is injective.
\end{proof}

The following result shows that the morphisms $\lambda_{x,y}: G_{y} \to G_x$
connect the congruence $\sim$ on $\SS$ with its restriction to $E(\SS)$.

\begin{lemma} \label{p:lambda-cong}
    Let $s,t \in \SS$, $x = s^\omega$ and $y = t^\omega$.
    Then,
    \[
        s \sim t
        \qqiff
        x \sim y \text{ and } \lambda_{y,x}(s) = t
        \qqiff
        x \sim y \text{ and } \lambda_{x,y}(t) = s.
    \]
\end{lemma}

\begin{proof}
    The last two conditions are equivalent since $\lambda_{x,y}^{-1}
    = \lambda_{y,x}$ whenever $x \sim y$.
    If $s \sim t$, then $\lambda_{y,x}(s) = ys = t^\omega s = t$ and
    $x \sim y$ because $\supp$ is a semigroup morphism, i.e.,
    \begin{equation*}
        \supp(x) = \supp(s^\omega) = \supp(s)^\omega = \supp(t)^\omega = \supp(t^\omega) = \supp(y).
    \end{equation*}
    If $x \sim y$ and $\lambda_{y,x}(s) = t$, then $t^\omega
    s = y s = \lambda_{y,x}(s) = t$ and $s^\omega t = \lambda_{x,y}(t)
    = \lambda_{y,x}^{-1}(t) = s$, so $s \sim t$.
\end{proof}

\subsection{$\B{R}$-basis of $\kk \SS$}

Recall that a $\B{Q}$-basis of the LRB algebra $\kk E(\SS)$ is defined by the formula
\begin{equation*}
    \B{Q}^\circ_x = \B{H}_x \ee^\circ_{\supp(x)}
    \text{~for all~} x \in E(\SS),
\end{equation*}
where $\{\ee^\circ_{X}\}$ is a complete system of primitive orthogonal idempotents
for the LRB algebra $\kk E(\SS)$.
We will define a basis of the LRBG algebra $\kk \SS$ by the formula
\begin{equation*}
    \B{R}_s := \B{H}_{s} \B{Q}^\circ_{s^\omega}
    = \B{H}_{s} \B{H}_{s^\omega} \ee^\circ_{\supp(s^\omega)}
    \overset{\cref{eq:lrbg1}}{=}
    \B{H}_s \ee^\circ_{\supp(s^\omega)}
    \text{~for all~} s \in \SS.
\end{equation*}
Note that $\B{R}_x = \B{Q}^\circ_x$ for $x \in E(\SS)$ because in this case $x^\omega = x$.
Hence, the $\B{R}$-basis of $\kk \SS$ contains the $\B{Q}$-basis of $\kk E(\SS)$.
The next result proves that this is indeed a basis.

\begin{prop}
    \label{prop:LRBG-R-basis}
    Let $\SS$ be a LRBG and $\supp: \SS \to \TT$ its support map.
    Fix a complete system of primitive orthogonal idempotents
    $\{\ee^\circ_{X}\}_{X \in E(\TT)}$ for the LRB algebra $\kk E(\SS)$.
    Define, for $s \in \SS$,
    \begin{equation}
        \label{eq:LRBG-R-basis}
        \B{R}_s = \B{H}_s \ee^\circ_{\supp(s^\omega)}.
    \end{equation}
    Then $\{\B{R}_s\}_{s \in \SS}$ is a basis of $\kk \SS$.
    More precisely, if the elements
    \begin{equation}
        \label{eq:u-homogeneous-section}
        \uu_{X} = \sum_{\substack{x \in E(\SS) \\ \supp(x) = X}} c_x \B{H}_{x}
    \end{equation}
    define the homogeneous section corresponding to
    $\{\ee^\circ_{X}\}_{X \in E(\TT)}$ as in \cref{thm:CSoPOIforLRB},
    then
    \begin{equation}
        \label{LRBG-H-basis-to-R-basis}
        \B{H}_{s} = \sum_{\substack{x \in E(\SS) \\ \supp(s^\omega) \leq \supp(x)}} c_{x} \B{R}_{sx}
    \end{equation}
    and for all $s, t \in \SS$,
    \begin{equation}
        \label{LRBG-HR-product-formula}
        \B{H}_{s} \B{R}_{t} =
        \begin{cases}
            \B{R}_{st}, & \text{if~} \supp(s^\omega) \leq \supp(t^\omega) \\
            0,          & \text{otherwise.}
        \end{cases}
    \end{equation}
\end{prop}

\begin{proof}
    To prove that we have a basis, it suffices to prove
    \cref{LRBG-H-basis-to-R-basis} since this implies $\{\B{R}_s\}_{ s \in \SS}$
    spans $\kk\SS$.

    Since $\{\ee^\circ_{X}\}_{X \in E(\TT)}$ is a complete system of primitive orthogonal idempotents
    and $\uu_{X} \ee^\circ_{X} = \ee^\circ_{X}$,
    \begin{equation*}
        1_{\kk\SS}
        = \sum_{X} \ee^\circ_{X}
        = \sum_{X} \uu_{X} \ee^\circ_{X}
        \overset{\cref{eq:u-homogeneous-section}}{=}
        \sum_{x \in E(\SS)} c_x \B{H}_{x} \ee^\circ_{\supp(x)}
        =
        \sum_{x \in E(\SS)} c_x \B{R}_{x}.
    \end{equation*}
    Thus,
    \begin{equation*}
        \B{H}_{s}
        = \sum_{x \in E(\SS)} c_x \B{H}_{s} \B{R}_{x}
    \end{equation*}
    so that it suffices to prove
    \begin{equation*}
        \B{H}_{s} \B{R}_{x} =
        \begin{cases}
            \B{R}_{sx}, & \text{if~} \supp(s^\omega) \leq \supp(x), \\
            0, & \text{otherwise}.
        \end{cases}
    \end{equation*}
    But this follows from \cref{LRBG-HR-product-formula} by setting $t = x$
    since $x^\omega = x$, so we turn to proving \cref{LRBG-HR-product-formula}.

    Let $s, t \in \SS$.
    By \cref{eq:lrbg2}, we have $\B{H}_{s t} = \B{H}_{s t s^\omega}$
    so that
    \begin{equation*}
        \B{H}_{s} \B{R}_{t}
        =
        \B{H}_{s} \B{H}_{t} \ee^\circ_{\supp(t^\omega)}
        =
        \B{H}_{st} \B{H}_{s^\omega} \ee^\circ_{\supp(t^\omega)}.
    \end{equation*}
    By \cref{lem:Saliola's},
    this is $0$ if $\supp(s^\omega) \not\leq \supp(t^\omega)$.
    Suppose $\supp(s^\omega) \leq \supp(t^\omega)$.
    By \cref{LRBG-supp-and-omega},
    we have
    $\supp\left((st)^\omega\right) = \supp(s^\omega) \supp(t^\omega)$,
    which equals $\supp(t^\omega)$ because
    $\supp(s^\omega) \leq \supp(t^\omega)$.
    Thus,
    \begin{equation*}
        \B{H}_{s} \B{R}_{t}
        = \B{H}_{s} \B{H}_{t} \ee^\circ_{\supp(t^\omega)}
        = \B{H}_{st} \ee^\circ_{\supp((st)^\omega)}
        = \B{R}_{st}.
        \qedhere
    \end{equation*}
\end{proof}

\subsection{$\B{R}$-basis of $\kk \TT$ and semisimplicity}
\label{sec:R-basis-for-T}

Similarly, one has a $\B{R}$-basis of $\kk \TT$
given by the formula
\begin{equation*}
    \B{R}_{S} = \B{H}_{S} \B{Q}^\circ_{S^\omega}
    \text{~for all~} S \in \TT,
\end{equation*}
where $\{\B{Q}^\circ_{X}\}_{X \in E(\TT)}$ is the $\B{Q}$-basis of $\kk E(\TT)$ defined
in \cref{sec:LRB-Q-basis} as
\[
    \B{Q}^\circ_X = \sum_{Y \geq X} \mu(X,Y) \, \B{H}_Y \text{ for all } X \in E(\TT).
\]
Notice that the $\B{R}$-basis of $\kk \TT$ contains the $\B{Q}$-basis of $\kk E(\TT)$ since
for all $X \in E(\TT)$, we have
\begin{equation}
    \label{R-basis-contains-Q0-basis}
    \B{R}_X = \B{H}_X \B{Q}^\circ_X = \sum_{Y \geq X} \mu(X,Y) \B{H}_{X \vee Y} = \sum_{Y \geq X} \mu(X,Y) \B{H}_{Y} = \B{Q}^\circ_X.
\end{equation}
This basis has the following product formula,
which will allow us to explicitly identify
$\kk \TT$ with a direct product of group algebras.

\begin{prop}
    \label{T-R-basis-product-formula}
    Let $\TT$ be a LRBG with central idempotents.
    The elements $\{ \B{R}_S \}_{S \in \TT}$ form a basis of $\kk\TT$.
    Moreover, they satisfy
    \[
        \B{R}_S \B{R}_T =
        \begin{cases}
            \B{R}_{ST},      & \text{if~} S^\omega = T^\omega, \\
            0,               & \text{otherwise,}
        \end{cases}
        \qqand
        \B{H}_{S} \B{R}_{T} =
        \begin{cases}
            \B{R}_{ST},      & \text{if~} S^\omega \leq T^\omega, \\
            0,               & \text{otherwise.}
        \end{cases}
    \]
\end{prop}

\begin{proof}
    The first statement follows by applying \cref{prop:LRBG-R-basis} with $\SS = \TT$,
    or by showing directly that the $\B{H}$-basis is in their span
    (compare with \cref{LRBG-H-basis-to-R-basis}):
    \[
        \B{H}_S =
        \B{H}_S \B{H}_{S^\omega} =
        \B{H}_S \sum_{Y \in E(\TT) \,:\, Y \geq S^\omega} \B{R}_Y =
        \sum_{Y \in E(\TT) \colon Y \geq S^\omega} \B{R}_{YS},
    \]
    where in the last step we used $(YS)^\omega = Y S^\omega = Y$, which implies
    $\B{H}_S \B{R}_Y = \B{H}_S \B{H}_Y \B{R}_Y = \B{H}_{YS} \B{R}_Y = \B{R}_{YS}$.

    For the first product formula, note that $\B{Q}^\circ_{X}$, for $X \in E(\TT)$, are
    central since $\TT$ has central idempotents and $\B{Q}^\circ_{X}$ is a linear
    combination of idempotents (because $\B{Q}^\circ_{X} \in \kk E(\TT)$). Thus,
    \[
        \B{R}_S \B{R}_T =
        (\B{H}_S \B{Q}^\circ_{S^\omega}) (\B{H}_T \B{Q}^\circ_{T^\omega}) =
        \B{H}_S \B{H}_T \B{Q}^\circ_{S^\omega} \B{Q}^\circ_{T^\omega} =
        \begin{cases}
            \B{H}_{ST} \B{Q}^\circ_{S^\omega} & \text{if } S^\omega = T^\omega,\\
            0                           & \text{if } S^\omega \neq T^\omega.
        \end{cases}
    \]
    If $S^\omega = T^\omega$, then $(ST)^\omega = S^\omega T^\omega = T^\omega$
    by \cref{LRBG-supp-and-omega}, from which it follows that
    \begin{equation*}
        \B{H}_{ST} \B{Q}^\circ_{S^\omega} = \B{H}_{ST} \B{Q}^\circ_{(ST)^\omega} = \B{R}_{ST}.
    \end{equation*}
    The second product formula follows directly from \cref{LRBG-HR-product-formula}.
\end{proof}

\begin{cor}
    \label{cor:semisimplicity-of-kT}
    Let $\TT$ be a LRBG with central idempotents
    and let $\{ \B{R}_{S} \}_{S \in \TT}$ be the $\B{R}$-basis of $\kk \TT$.
    For each $X \in E(\TT)$, the map
    $\kk G_{X} \xrightarrow{} \kk \TT \B{R}_{X}: \B{H}_{S} \mapsto \B{H}_{S} \B{R}_{X}$
    is an isomorphism of algebras, and
    \begin{equation}\label{eq:T-prod-gps}
        \kk \TT
        = \prod_{X \in E(\TT)} \kk \TT \B{R}_{X}
        \cong \prod_{X \in E(\TT)} \kk G_{X}.
    \end{equation}
\end{cor}

\begin{proof}
    By \Cref{R-basis-contains-Q0-basis}, we have that
    $\B{R}_{X} = \B{Q}^\circ_{X}$ for all $X \in E(\TT)$.
    Thus, $\{ \B{R}_{X} \}_{X \in E(\TT)}$ is a complete system of primitive
    orthogonal idempotents for $\kk E(\TT)$ which are central in $\kk \TT$
    since $\TT$ has central idempotents. Thus, as algebras, we have
    $\kk \TT = \prod_{X \in E(\TT)} \kk \TT \B{R}_{X}$.

    Let $\varphi: \kk G_{X} \xrightarrow{} \kk \TT \B{R}_{X}$
    be the map defined on the $\B{H}$-basis by $\varphi(\B{H}_{S}) = \B{H}_{S} \B{R}_{X}$.
    By \cref{T-R-basis-product-formula},
    if $S \in G_{X}$, then $\B{H}_{S} \B{R}_{X} = \B{R}_{SX} = \B{R}_{S}$
    so that $\varphi(\B{H}_{S}) = \B{R}_{S}$.
    It follows that $\varphi$ is an isomorphism of vector spaces.
    To see that $\varphi$ is a map of algebras, note that
    for $S, T \in G_{X}$ we have, again by \cref{T-R-basis-product-formula},
    \begin{equation*}
        \varphi(\B{H}_{S} \B{H}_{T})
        = \varphi(\B{H}_{ST})
        = \B{R}_{ST}
        = \B{R}_{S} \B{R}_{T}
        = \varphi(\B{H}_{S}) \varphi(\B{H}_{T}).
        \qedhere
    \end{equation*}
\end{proof}

Note that \cref{cor:semisimplicity-of-kT} implies  that $\supp : \kk\SS \to \kk\TT$ is the semisimple quotient of $\kk
\SS$ described by Margolis and Steinberg in \cite[\S~5.2.1]{ms11HomologySemigroups}.

\section{Idempotents for LRBG algebras}\label{sec:idemforLRBGs}

Let $\SS$ be a LRBG.
In this section we construct complete systems of primitive orthogonal idempotents for the semigroup algebra $\kk \SS$.
Recall that $\kk \SS$ is the $\kk$-algebra with linear basis $\{ \B{H}_s \}_{s \in \SS}$ and multiplication defined on this basis by
\[
    \B{H}_s\B{H}_t = \B{H}_{st}
    \qquad
    \text{for all } s,t \in \SS.
\]
Producing a CSoPOI for $\kk \SS$ is already difficult even when $E(\SS)$ consists of a single element, i.e. when $\SS$ is a group.
We review this case first.

\subsection{Idempotents for group algebras}\label{ss:isotypic}

We restrict to group algebras $\kk G$ where $G$ is finite. The idempotent theory of $\kk G$ is integral to understanding $\kk G$ as an algebra, and thus the representation theory of $G$ itself.

There are no closed formulas describing a complete system of primitive orthogonal idempotents for an arbitrary group algebra $\kk G$
although some case-by-case constructions can be found in the literature.
For example, the idempotent elements inside the Young seminormal basis for the symmetric group algebra $\kk \mathfrak{S}_n$ form a CSoPOI.
See \cite[Theorem~1.5]{Murphy1983}.

The general theory of characters allows us to construct a complete family of orthogonal idempotents that lie in the center $Z(\kk G)$ of the group algebra.
This collection of idempotents is a CSoPOI if and only if $G$ is abelian.
Let $\chi$ be a character of an irreducible complex representation of $G$.
The \defn{isotypic projector} $\B{E}_\chi \in \kk G$ is the element
\[
    \B{E}_\chi := \frac{\chi(1_G)}{|G|} \sum_{g \in G} \ol{\chi(g)} \B{H}_g,
\]
where $\ol{\chi(g)}$ denotes the complex conjugate of $\chi(g) \in \C^{\times}$.
The name isotypic projector comes from the fact that for any $\C G$-module $V$, the action of $\B{E}_{\chi}$ is the projection onto the isotypic component of the irreducible representation with character $\chi$.
In particular,
\[
    \B{E}_\chi \B{E}_\phi = \begin{cases}
        \B{E}_\chi      & \text{if } \chi = \phi,\\
        0               & \text{otherwise.}
    \end{cases} 
\]
When the group $G$ is abelian, there are as many irreducible representations as group elements.
It follows that in this case, the collection of isotypic projectors is the unique complete system of primitive orthogonal idempotents of $\kk G$.

\begin{example}
    Suppose $G = C_n = \{ 1, a, a^2, \cdots, a^{n-1} \}$ is the cyclic group of order $n$, and let $\omega = \omega_n \in \C$ be a primitive $n$-th root of unity.
    Then, the characters of the different irreducible representations of $G$ are $\{ \chi_0,\dots, \chi_{n-1} \}$, where $\chi_j$ is completely determined by $\chi_j(a) = \omega^j$ for all $0 \leq j \leq n-1$.
    Thus,
    \[
        \B{E}_{\chi_j}
        = \frac{1}{n} \Big( \B{H}_1 + \omega^{-j} \B{H}_{a} + \omega^{-2j} \B{H}_{a^2} + \dots +  + \omega^{-(n-1)j} \B{H}_{a^{n-1}} \Big)
        \in  \C G .
    \]
\end{example}

\subsection{Idempotents for semilattice of groups algebras}\label{ss:idemp-lat-gps}

Let $\TT$ be a LRBG with central idempotents.
Recall from \cref{cor:semisimplicity-of-kT} that
$\kk E(\TT)$ admits a basis $\{ \B{Q}^\circ_{X} \}_{X \in E(\TT)}$
such that the map
\begin{eqnarray*}
    \prod_{X \in E(\TT)} \kk G_{X}
    & \xrightarrow{} &
    \kk \TT
\end{eqnarray*}
defined on each $\kk G_{X}$ by $\B{H}_{S} \mapsto \B{H}_{S} \B{Q}^\circ_{X}$ is an
isomorphism of algebras.
By applying this map to a CSoPOI for each of the algebras $\kk G_{X}$,
we obtain a CSoPOI for $\kk \TT$.

\begin{prop}\label{prop:CSoPOI-SemLatGps}
    Let $\TT$ be a LRBG with central idempotents.
    For each $X \in E(\TT)$, fix a complete system of primitive orthogonal idempotents
    $\{ \B{E}_X^{(i)} \}_{i \in I_{X}}$ of the group algebra $\kk G_X$ for some index set $I_{X}$.
    Then, $\{ \B{E}_X^{(i)} \B{Q}^\circ_X \}_{X \in E(\TT), i \in I_{X}}$ is a CSoPOI of $\kk \TT$.
    Moreover, all CSoPOIs arise in this manner.
\end{prop}

\subsection{Idempotents for LRBG algebras}

In this section we intertwine the construction of CSoPOI for LRB algebras in
\cref{ss:CSoPOI-LRB} and the ideas of \cref{ss:idemp-lat-gps}. Our main result is \cref{thm:CSoPOI-LRBG} below.

\begin{theorem}\label{thm:CSoPOI-LRBG}
    Let $\SS$ be a LRBG with support map $\supp: \SS \to \TT$
    and let $\{ \ee^\circ_X \}_{X \in E(\TT)}$ be a complete system of primitive orthogonal idempotents
    of the LRB algebra $\kk E(\SS)$.
    For each $X \in E(\TT)$,
    fix an element $x \in E(\SS)$ with $\supp(x) = X$ and
    a CSoPOI $\{ \B{E}^{(i)}_X \}_{i}$ of $\kk G_x$.
    Define elements $\ee_{(X,i)} \in \kk \SS$ by
    \[
        \ee_{(X,i)} = \ee^\circ_X \B{E}^{(i)}_X \ee^\circ_X.
    \]
    Then the collection $\{ \ee_{(X,i)} \}_{X,i}$ forms a complete system of primitive orthogonal idempotents
    of $\kk \SS$.
\end{theorem}

Observe that these elements are independent of the choice of $x \in E(\SS)$ with $\supp(x) = X$,
since if $x' \in E(\SS)$ is any other element with the same support, then
$\{ \B{H}_{x'} \B{E}^{(i)}_X \}$ is a CSoPOI of $\kk G_{x'}$ by
\cref{left-multiplication-induces-group-morphism},
and $\ee^\circ_X (\B{H}_{x'} \B{E}^{(i)}_X) \ee^\circ_X = (\ee^\circ_X \B{H}_{x})
\B{H}_{x'} \B{E}^{(i)}_X \ee^\circ_X = \ee^\circ_X \B{E}^{(i)}_X \ee^\circ_X
= \ee_{(X,i)}$ for all $i$.

To prove \cref{thm:CSoPOI-LRBG}, we fix the following notation
to be used throughout this section:
\begin{itemize}
    \item
        $\SS$ denotes a LRBG with support map $\supp: \SS \to \TT$; 
    \item
        $\{ \ee^\circ_X \}_{X \in E(\TT)}$ denotes a CSoPOI of the LRB algebra $\kk E(\SS)$;
    \item
        $\{ \uu_X \}_{X \in E(\TT)}$ is the homogeneous section associated with
        $\{ \ee^\circ_X \}_{X \in E(\TT)}$ (see \cref{thm:CSoPOIforLRB}).
\end{itemize}

\begin{prop}\label{p:alt-eXi}
With the above notation,
for all $X,Y \in E(\TT)$, $y \in E(\SS)$ with $\supp(y) = Y$, and $v \in \kk G_y$,
\[
    \ee^\circ_X \uu_Y v \ee^\circ_Y =
    \begin{cases}
        \uu_Y v \ee^\circ_Y & \text{if } X = Y,\\
        0                       & \text{if } X \neq Y.
    \end{cases}
\]
In particular, for all $X$ and $i$,
\[\ee_{(X,i)} = \uu_X \B{E}^{(i)}_X \ee^\circ_X.\]
\end{prop}

\begin{proof}
Note that by linearity, we can assume that $v = \B{H}_g$ for some $g \in G_y$.
The proof is by induction on $X$.

\medskip

\noindent\textbf{Base case.}
Let $X = \top$, the top element of the lattice $E(\TT)$.
If $Y = \top$, then $\ee^\circ_X = \ee^\circ_Y = \uu_\top$ and the result follows because $\uu_\top$ is idempotent.
If $Y \neq \top$, then for any $x \in E(\SS)$ with $\supp(x) = X$ we have
\begin{equation}\label{eq:aux-eHe-uHe}
    \ee^\circ_X \uu_Y \B{H}_g \ee^\circ_Y                       =
    \ee^\circ_X \B{H}_x \uu_Y \B{H}_g \ee^\circ_Y               \overset{\cref{eq:lrbg2}}{=}
    \ee^\circ_X \B{H}_x \uu_Y \B{H}_g \B{H}_x \ee^\circ_Y       \overset{\rm (\cref{lem:Saliola's})}{=}
    0 .
\end{equation}
Observe that this computation holds as long as the hypotheses of \cref{lem:Saliola's} are satisfied; that is, whenever $X \not\leq Y$. We will use this fact below.

\medskip

\noindent\textbf{Inductive step.}
Assume the result is true for all $Z > X$.
If $X = Y$, we use the recursive definition of the $\ee^\circ_Y$ in \Cref{eq:saliolarecursion}:
\[
      \ee^\circ_Y \uu_Y \B{H}_g \ee^\circ_Y
    = \Big( \uu_Y - \sum_{Z > Y} \uu_Y \ee^\circ_Z \Big) \uu_Y \B{H}_g \ee^\circ_Y
    = \uu_Y \uu_Y \B{H}_g \ee^\circ_Y - \sum_{Z > Y} \uu_Y (\ee^\circ_Z \uu_Y \B{H}_g \ee^\circ_Y)
    = \uu_Y \B{H}_g \ee^\circ_Y.
\]
The elements $Z$ in the sum satisfy $Z > X$, so the last equality follows from the induction hypothesis.
If $X \not\leq Y$, then the computation in \cref{eq:aux-eHe-uHe} holds.
Finally, if $X < Y$, we use the induction hypothesis again to deduce:
\begin{align*}
          \ee^\circ_X \uu_Y \B{H}_g \ee^\circ_Y 
    & = \Big( \uu_X - \sum_{Z > X} \uu_X \ee^\circ_Z \Big) \uu_Y \B{H}_g \ee^\circ_Y
    = \uu_X \uu_Y \B{H}_g \ee^\circ_Y - \sum_{Z > X} \uu_X (\ee^\circ_Z \uu_Y \B{H}_g \ee^\circ_Y)\\
    &
    = \uu_X \uu_Y \B{H}_g \ee^\circ_Y - \uu_X (\ee^\circ_Y \uu_Y \B{H}_g \ee^\circ_Y)
    = \uu_X \uu_Y \B{H}_g \ee^\circ_Y - \uu_X \uu_Y \B{H}_g \ee^\circ_Y
    = 0.
\end{align*}
This completes the proof of the first identity.
Finally, for all $X$ and $i$ we have
\[
    \ee_{(X,i)} = 
    \ee^\circ_X \B{E}^{(i)}_X \ee^\circ_X = 
    \ee^\circ_X \uu_X \B{E}^{(i)}_X \ee^\circ_X = 
    \uu_X \B{E}^{(i)}_X \ee^\circ_X. \qedhere
\]
\end{proof}

We can prove \cref{thm:CSoPOI-LRBG} using \cref{p:alt-eXi}.
\begin{proof}[Proof of \cref{thm:CSoPOI-LRBG}]
    \noindent\textbf{Idempotent and orthogonal.}
        For $X \neq Y$, the orthogonality of $\ee^\circ_X$ and $\ee^\circ_Y$ implies the orthogonality of $\ee_{(X,i)}$
        and $\ee_{(Y,j)}$.
        For $X = Y$, we have
        \[
            \ee_{(X,i)} \ee_{(X,j)} =
            (\ee^\circ_X \B{E}^{(i)}_X \ee^\circ_X) (\ee^\circ_X \B{E}^{(j)}_X \ee^\circ_X)
            = \uu_{X} \B{E}^{(i)}_X \uu_X \B{E}^{(j)}_X \ee^\circ_X
            = \uu_{X} \B{E}^{(i)}_X \B{E}^{(j)}_X \ee^\circ_X
            = \begin{cases}
                \ee^\circ_X \B{E}^{(i)}_X \ee^\circ_X   & \text{if } i = j,\\
                0                                   & \text{if } i \neq j.
            \end{cases}
        \]
        In the second step we used $\B{E}^{(i)}_X \uu_X = \B{E}^{(i)}_X \B{H}_x \uu_X = \B{E}^{(i)}_X \B{H}_x = \B{E}^{(i)}_X$.

    \medskip
    \noindent\textbf{Complete.}
        This follows from completeness of $\{ \ee^\circ_X \}_X$ and each of the $\{ \B{E}^{(i)}_X \}_i$:
        \[
            \sum_{X,i} \ee^\circ_X \B{E}^{(i)}_X \ee^\circ_X = \sum_X \ee^\circ_X \Big( \sum_i \B{E}^{(i)}_X \Big) \ee^\circ_X = \sum_X \ee^\circ_X \B{H}_x \ee^\circ_X = \sum_X \ee^\circ_X = 1.
        \]
        Above, we use that $\B{H}_x$ is the unit of $\kk G_x$ and that $\ee^\circ_X \B{H}_x = \ee^\circ_X$, since $\supp(x) = X$.

    \medskip
    \noindent\textbf{Primitive.}
        Recall that if $e$ is an idempotent in a finite dimensional algebra $A$ and
        $e + \rad(A)$ is a primitive idempotent in $A/\rad(A)$,
        then $e$ is also primitive
        (see, for instance, \cite[Lemma~D.28]{am17} or \cite[Corollary~1.7.4]{Benson1998:I}).
        Since $\rad(\kk \SS) = \ker(\supp)$,
        it suffices to prove $\supp(\ee^\circ_X \B{E}^{(i)}_X \ee^\circ_X)$
        is primitive in $\kk \TT$.

        To this end, observe that
        \[
            \supp(\ee^\circ_X \B{E}^{(i)}_X \ee^\circ_X) = 
            \B{Q}^\circ_X \supp(\B{E}^{(i)}_X) \B{Q}^\circ_X =
            \supp(\B{E}^{(i)}_X) \B{Q}^\circ_X,
        \]
        and so it
        is one of the primitive idempotents of $\kk \TT$ in \cref{prop:CSoPOI-SemLatGps}.
        Indeed, recall that the restriction of the support map to $G_x \subseteq \SS$ is an isomorphism to $G_X \subseteq \TT$, and therefore $\supp(\B{E}^{(i)}_X)$ is a primitive idempotent of $\kk G_X$ and $\supp(\B{E}^{(i)}_X) \B{Q}^\circ_X$ is a primitive idempotent of $\kk \TT$.
\end{proof}

This construction requires having a CSoPOI for each of the group algebras $\kk G_x$ which, as discussed in \cref{ss:isotypic}, is already a very difficult problem.
In the following section we restrict to the case in which each of the maximal subgroups $G_x$ is abelian.

\subsection{Left regular bands of abelian groups}\label{sec:LRBaG}

Let us assume that $\SS$ is a \defn{left regular band of abelian groups (LRBaG)}: a LRBG in which each maximal subgroup $G_x$ is abelian. We will give a more explicit construction for a CSoPOI in this case in \cref{t:CSoPOI-LRBaG}. To do so, we will use the dual group $\wh{G}$ of an abelian group $G$.

\begin{definition}\label{def:S-hat}
    Let $\SS$ be a LRBaG. Define
    \[
        \wh{\SS} := \bigsqcup_{x \in E(\SS)} \wh{G_x},
    \]
    where $\wh{G} = {\rm Hom}(G,\kk^\times)$ denotes the dual group of $G$
    consisting of group morphisms $\phi: G \to \kk^\times$ with multiplication given by
    $(\phi \phi')(g) = \phi(g) \phi'(g)$ for all $\phi, \phi': G \to \kk^\times$ and $g \in G$.

    Given $\phi \in \wh{\SS}$, we let $\abs{\phi}$ denote the idempotent $x \in
    E(\SS)$ such that $\phi \in \wh{G_{x}}$.
    Thus, $\phi \in \wh{G_{\abs{\phi}}}$.
\end{definition}

Although each $\phi \in \wh{\SS}$ is a function
defined on the subgroup $G_{|\phi|}$, we extend it to a function on all of $\SS$ as follows.

\begin{definition}\label{def:evaluation}
    Given $s \in \SS$ and $\phi \in \wh{\SS}$ the \defn{evaluation} of $\phi$ on $s$ is 
    \[
        \phi(s) := \begin{cases}
            \phi(s) & \text{if } s \in G_{\abs{\phi}},\\
            0 & \text{otherwise.}
        \end{cases}
    \]
    Note that $s \in G_{\abs{\phi}}$ if and only if $s^\omega = \abs{\phi}$.
\end{definition}

We use the partial order $\leq$ (resp. equivalence relation $\sim$) on $E(\SS)$ to define a partial order $\chleq$ (resp. equivalence relation $\sim$) on $\wh{\SS}$.
Recall that $yx = y$ whenever $x \leq y$ or $x \sim y$.
In that case, we have group morphisms $\lambda_{y,x} : G_x \to G_y$ with dual maps $\lambda_{y,x}^* : \wh{G_y} \to \wh{G_x}$
(see \cref{decomposition-into-maximal-subgroups}).
Explicitly, if $\psi \in \wh{G_y}$, then $\lambda_{y,x}^*(\psi) \in \wh{G_x}$ is defined by
\[
    \lambda_{y,x}^*(\psi)(s) = \psi(\lambda_{y,x}(s)) = \psi(y s) \quad\text{for all } s \in G_x.
\]

\begin{definition}\label{def:chleq}
    For $\phi,\psi \in \wh{\SS}$, with $x = \abs{\phi}$ and $y = \abs{\psi}$, define
    \begin{align}
        \label{eq:poset-S-hat} \phi \chleq \psi         \qqiff &          x \leq y \text{ and } \phi = \lambda_{y,x}^*(\psi);\\
        \label{eq:equiv-S-hat} \phi \sim \psi         \qqiff &          x \sim y \text{ and } \phi = \lambda_{y,x}^*(\psi).
    \end{align}
\end{definition}
Since $\lambda_{x,x}^* = {\rm id}_{\wh{G_x}}$ for all $x$, and
$\lambda_{z,y}^* \circ \lambda_{y,x}^* = \lambda_{z,x}^*$ whenever $x \leq y \leq z$ or $x \sim y \sim z$ (because $zyx = zx = z$), these relations are reflexive and transitive. Moreover, the antisymmetry of $\leq$ implies that of $\chleq$.

We let $\wh{\TT} := \wh{\SS}/\mathord{\sim}$ denote the collection of equivalence classes and let $\supp : \wh{\SS} \to \wh{\TT}$ denote the map that sends an element $\phi \in \wh{\SS}$ to its equivalence class $\supp(\phi) \in \wh{\TT}$.

\begin{example}
    Fix a finite abelian group $G$ and let $\SS = \hsiao$ be the semigroup of $G$-compositions of $[n]$ in \cref{ss:Hsiao's}.
    We naturally identify $\wh{\SS}$ with the collection $\Sigma_n[\wh{G}]$ of $\wh{G}$-compositions of~$[n]$:
    tuples $\phi = \big( (S_1,\phi_1),\dots,(S_k,\phi_k) \big)$ with $(S_1,\dots,S_k) \in \Sigma_n$ and $\phi_i \in \wh{G}$ for all $i$.
    Under the identification $E(\hsiao) \cong \Sigma_n$, we have that $\abs{\phi}$ is the underlying set composition $(S_1,\dots,S_k)$.
    With $\phi$ as above, the evaluation of $\phi$ on $s = \big( (S_1,g_1),\dots,(S_k,g_k) \big) \in G_{\abs{\phi}}$ is simply
    \[
        \phi(s) = \phi_1(g_1) \phi_2(g_2) \dots \phi_k(g_k).
    \]
    Observe that $\phi(s)$ is only defined if $\abs{s} = \abs{\phi}$, i.e. if $s$ and $\phi$ have the same underlying set composition.

    The partial order $\chleq$ on $\wh{G}$-compositions is determined by
    \[
        \big( (S_1,\phi_1),\dots,(S_k,\phi_k) \big) \chleq \big( (T_1,\psi_1),\dots,(T_\ell,\psi_\ell) \big)
    \]
    if and only if
    \begin{enumerate}
        \item
            $(S_1,\dots,S_k)$ is refined by $(T_1,\dots,T_\ell)$, and
        \item
            whenever $S_m = T_{i_m} \sqcup \dots  \sqcup T_{j_m}$, we have $\phi_m = \psi_{i_m}  \dots \psi_{j_m}$.
    \end{enumerate}
    In other words, the cover relations of the partial order $\chleq$ on $\wh{G}$-compositions are of the form
    \begin{align*}
        \big( (S_1,\phi_1) , \dots, (S_i \sqcup S_{i+1} ,\phi_i\phi_{i+1}) ,\dots , (S_k,\phi_k) \big)
        \chleq 
        \big( (S_1,\phi_1) , \dots, (S_i ,\phi_i) , (S_{i+1} ,\phi_{i+1}) ,\dots , (S_k,\phi_k) \big).
    \end{align*}

    On the other hand, the relation $\sim$ on $\wh{G}$-compositions satisfies
    \[
        \big( (S_1,\phi_1),\dots,(S_k,\phi_k) \big) \sim \big( (T_1,\psi_1),\dots,(T_\ell,\psi_\ell) \big)
    \]
    if and only if $k = \ell$ and there is a permutation $\sigma \in \mathfrak{S}_k$ such that $(S_i , \phi_i) = (T_{\sigma(i)} , \psi_{\sigma(i)})$ for all $i$.
    Thus, $\wh{\TT}$ is identified with $\Pi_n[\wh{G}]$, the collection of $\wh{G}$-partitions of $[n]$, and
    \[
        \supp\big( ((S_1,\phi_1),\dots,(S_k,\phi_k)) \big) = \big\{ (S_1,\phi_1),\dots,(S_k,\phi_k) \big\}.
    \]
    See \cref{fig:hsiao-C2-hat-relations} for an example with $n=2$ and $G = \ztwo$.
    \begin{figure}[ht]
        \begin{tikzpicture}[yscale=0.7, xscale=0.7]
            \node (node_0) at (-50bp,  0bp) {$\big( 12^{\trv} \big)$};
            \node (node_1) at ( 50bp,  0bp) {$\big( 12^{\sgn} \big)$};
            
            \node (node_00l) at (-290bp,  80bp) {$\big( 1^{\trv} | 2^{\trv} \big)$};
            \node (node_11l) at (-210bp,  80bp) {$\big( 1^{\sgn} | 2^{\sgn} \big)$};
            \node (node_01l) at (-130bp,  80bp) {$\big( 1^{\trv} | 2^{\sgn} \big)$};
            \node (node_10l) at ( -50bp,  80bp) {$\big( 1^{\sgn} | 2^{\trv} \big)$};
            
            \node (node_00r) at (  50bp,  80bp) {$\big( 2^{\trv} | 1^{\trv} \big)$};
            \node (node_11r) at ( 130bp,  80bp) {$\big( 2^{\sgn} | 1^{\sgn} \big)$};
            \node (node_01r) at ( 210bp,  80bp) {$\big( 2^{\sgn} | 1^{\trv} \big)$};
            \node (node_10r) at ( 290bp,  80bp) {$\big( 2^{\trv} | 1^{\sgn} \big)$};

            \draw [blue] (-330bp, 60bp) rectangle (-10bp,100bp); \node [blue] at (-260bp, 40bp) {$\wh{G_{(1|2)}}$};
            \draw [blue] ( 330bp, 60bp) rectangle ( 10bp,100bp); \node [blue] at ( 260bp, 40bp) {$\wh{G_{(2|1)}}$};
            \draw [blue] ( -90bp,-20bp) rectangle ( 90bp, 20bp); \node [blue] at (     0,-40bp) {$\wh{G_{(12)}}$};
            
            \draw [black,thick] (node_0) -- (node_00l);
            \draw [black,thick] (node_0) -- (node_11l);
            \draw [black,thick] (node_0) -- (node_00r);
            \draw [black,thick] (node_0) -- (node_11r);
            \draw [black,thick] (node_1) -- (node_10l);
            \draw [black,thick] (node_1) -- (node_01l);
            \draw [black,thick] (node_1) -- (node_10r);
            \draw [black,thick] (node_1) -- (node_01r);
            \draw [red,dashed] (node_00l) to [bend left=30] (node_00r);
            \draw [red,dashed] (node_11l) to [bend left=30] (node_11r);
            \draw [red,dashed] (node_01l) to [bend left=30] (node_01r);
            \draw [red,dashed] (node_10l) to [bend left=30] (node_10r);
        \end{tikzpicture}
        \caption{
        The ten elements in $\Sigma_2[\wh{C_2}]$, where ${\trv}$ denotes the trivial character of $\ztwo$ (${\trv}(+1) = {\trv}(-1) = 1$) and ${\sgn}$ denotes the sign representation of $\ztwo$ (${\sgn}(+1) = 1$ and ${\sgn}(-1) = -1$).
        To shorten notation, we write $\big( 12^{\trv} \big)$ for $\big( (\{1,2\} , {\trv} ) \big)$ and 
        $\big( 1^{\sgn} | 2^{\trv} \big)$ for $\big( (\{1\} , {\sgn} ), ( \{2\} , {\trv} ) \big)$. 
        The solid black lines correspond to the partial order $\chleq$.
        The dashed red lines correspond to the equivalence relation $\sim$.
        The boxes group elements according to what the maximal subgroup they are characters of.}
        \label{fig:hsiao-C2-hat-relations}
    \end{figure}
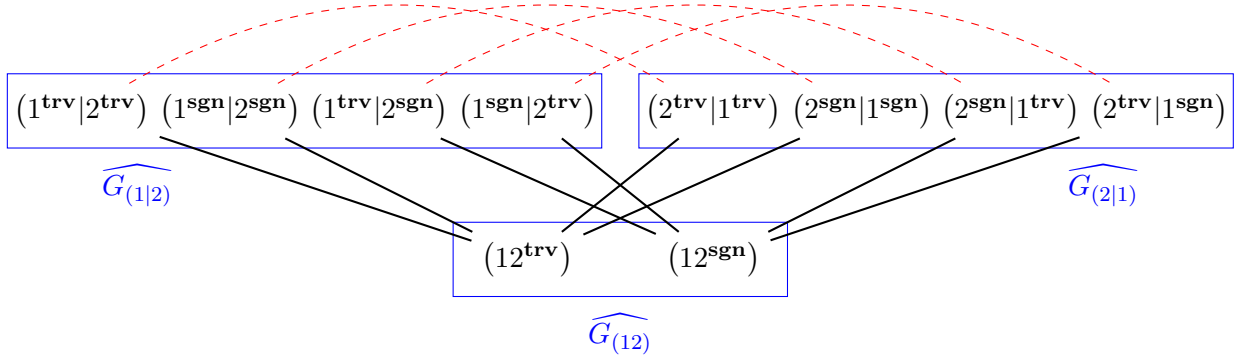
\end{example}

For $\phi \in \wh{\SS}$, let $\B{E}_\phi \in \kk G_{\abs{\phi}} = \kk \{ \B{H}_s \}_{s^\omega = \abs{\phi}}$ be the isotypic projector of $\phi$, as reviewed in \cref{ss:isotypic}.
Explicitly:
\begin{equation}
    \B{E}_\phi := \dfrac{1}{|G_{\abs{\phi}}|} \sum_{s \in G_{\abs{\phi}}} \ol{\phi(s)} \B{H}_s.
\end{equation}
Since the groups $G_x$ are abelian, $\{ \B{E}_\phi \}_{\abs{\phi} = x}$ is a linear basis of $\kk G_x$ for all $x \in E(\SS)$, and $\{ \B{E}_\phi \}_{\phi \in \wh{\SS}}$ is a linear basis of $\kk \SS$.
We call this the \defn{basis of locally orthogonal idempotents} since for all $\phi,\psi \in \wh{G_x}$ we have
\begin{equation}
    \B{E}_\phi \B{E}_\psi =
    \begin{cases}
        \B{E}_\phi 	& 	\text{if } \phi = \psi,\\
        0			&	\text{otherwise.}
    \end{cases}
\end{equation}

\begin{prop}\label{prop:multiplyEandH}
    Let $\SS$ be a LRBaG. For all $x,y \in E(\SS)$ and $\phi \in \wh{\SS}$ with $\abs{\phi} = x$,
    \[
        \B{H}_y \B{E}_\phi = \sum_{\substack{\psi \in \wh{G_{yx}}\\ \lambda_{y,x}^*(\psi) = \phi}} \B{E}_\psi.
    \]
    In particular, if $x \leq y$ 
    \begin{equation} \label{eq:EandHmultiplication}
        \B{H}_y \B{E}_\phi =
        \sum_{\substack{\psi \in \wh{G_y} \\ \phi \chleq \psi}} \B{E}_\psi,
    \end{equation}
    and if $x \sim x'$,
    \begin{equation} \label{eq:EandHmultiplication2}
        \B{H}_{x'} \B{E}_\phi = \B{E}_{\phi'},
    \end{equation}
    where $\phi' \in \wh{G_{x'}}$ is the unique character with $\supp(\phi') = \supp(\phi)$.
\end{prop}

\begin{proof}
    \Cref{eq:EandHmultiplication,eq:EandHmultiplication2} follow from the first claim by the definition of the partial order $\chleq$ and the support relation $\supp$ in \cref{def:chleq}.
    The first claim is an instance of the following result about abelian groups:
    if $\lambda : H \to G$ is a morphism of abelian groups and $\phi \in \wh{H}$, then
    \[
        \lambda(\B{E}_\phi) = \sum_{\psi \circ \lambda = \phi} \B{E}_\psi \in \kk G.
    \]
    Indeed, an algebra morphism sends idempotents to idempotents, so $\lambda(\B{E}_\phi)$ is necessarily a sum of isotypic projectors of $G$.
    To see which isotypic projectors appear in the sum, we study the action of $\lambda(\B{E}_\phi)$ on the simple $G$-modules. The element $\B{E}_\psi$ appears in the sum if and only if $\lambda(\B{E}_\phi)$ acts trivially on the simple module with character $\psi$.
    Take a simple $G$-module $M$ with character $\psi$, and view it as an $H$-module via $h \cdot m := \lambda(h) \cdot m$ for all $h \in H$ and $m \in M$.
    With this definition, the action of $\lambda(\B{E}_\phi)  \in \kk G$ on $M$ is the same as the action of $\B{E}_\phi \in \kk H$ on $M$.
    Since $M$ is one-dimensional, $M$ is a simple $H$-module, with character
    \[
        {\rm trace}(M \xrightarrow{{h} \cdot } M)
        = {\rm trace}(M \xrightarrow{{\lambda(h)} \cdot } M)
        = \psi(\lambda(h)).
    \]
    Thus, $\lambda(\B{E}_\phi)$ acts trivially on $M$ if and only if $\B{E}_\phi$ acts trivially on $M$ if and only if $\phi = \psi \circ \lambda$.
\end{proof}

\begin{theorem}\label{t:CSoPOI-LRBaG}
    Let $\SS$ be a LRBaG and let
    $\{ \B{Q}^\circ_x \}_{x \in E(\SS)}$ be the $\B{Q}$-basis of $\kk E(\SS)$
    with corresponding homogeneous section $\{ \uu_X \}_{X \in E(\TT)}$.
    For~$\phi \in \wh{\SS}$, and $\Phi \in \wh{\TT}$ with $\supp(\phi) = \Phi$, define
    \begin{equation}\label{eq:CSoPOI-LRBaG}
        \B{Q}_\phi := \B{E}_\phi \B{Q}^\circ_{\abs{\phi}}
        \qqand
        \ee_\Phi := \uu_{\abs{\Phi}} \B{Q}_\phi.
    \end{equation}
    Then $\{\B{Q}_{\phi}\}_{\phi \in \wh{\SS}}$ is a basis of primitive idempotents of $\kk \SS$
    and $\{ \ee_\Phi \}_{\Phi \in \wh{\TT}}$ is a complete system of primitive orthogonal
    idempotents of $\kk \SS$.
\end{theorem}

It follows that
\begin{equation}\label{eq:LRBaG-R-to-Q}
    \B{Q}_\phi =
    \dfrac{1}{|G_{\abs{\phi}}|} \Big( \sum_{s^\omega = \abs{\phi}} \ol{\phi(s)} \B{H}_s  \Big) \B{Q}^\circ_{\abs{\phi}} =
    \dfrac{1}{|G_{\abs{\phi}}|} \sum_{s^\omega = \abs{\phi}} \ol{\phi(s)} \B{R}_s.
\end{equation}
Note that $\ee_\Phi$ is independent of the chosen $\phi$ since, by \Cref{eq:EandHmultiplication2},
$\B{Q}_{\phi'} = \B{H}_{\abs{\phi'}} \B{Q}_{\phi}$
whenever $\supp(\phi') = \supp(\phi)$.

\begin{proof}[Proof of \cref{t:CSoPOI-LRBaG}]
    Since $\B{E}_\phi = \B{E}_\phi \B{H}_{\abs{\phi}}$ and $\B{Q}^\circ_{\abs{\phi}} = \B{H}_{\abs{\phi}} \ee^\circ_{\supp(\abs{\phi})}$, we have $\B{Q}_\phi = \B{E}_\phi \ee^\circ_{\supp(\abs{\phi})}$.
    In view of \cref{p:alt-eXi}, the elements in \cref{eq:CSoPOI-LRBaG} are precisely the ones in the CSoPOI of \cref{thm:CSoPOI-LRBG}.

    Moreover, since
    \[
        \supp(\B{Q}_\phi) = \supp(\B{E}_\phi \B{Q}^\circ_{\abs{\phi}}) = \B{E}_{\supp(\phi)} \B{Q}^\circ_{\abs{\supp(\phi)}}
    \]
    is one of the primitive idempotents in \cref{prop:CSoPOI-SemLatGps}, $\B{Q}_\phi$ is primitive \cite[Lemma~D.28]{am17}.
\end{proof}

\begin{example}
    Let $\SS = \Sigma_2[C_2]$ and consider the following homogeneous section:
    \[
        \uu_{\{1|2\}} = \B{H}_{(1|2)}
        \hspace{5em}
        \uu_{\{12\}} = \B{H}_{(12)}.
    \]
    The corresponding CSoPOI of $\kk E(\SS)$ is
    \[
        \ee^\circ_{\{1|2\}} = \B{H}_{(1|2)}
        \hspace{5em}
        \ee^\circ_{\{12\}} = \B{H}_{(12)} - \B{H}_{(1|2)},
    \]
    and the CSoPOI of $\kk \SS$ is
    \begin{align*}
        \ee_{\{ 1^{\trv} \mid 2^{\trv} \}} =
        \B{E}_{( 1^{\trv} \mid 2^{\trv} )} & =
                    \dfrac{1}{4} \Big( \B{H}_{(1^+ \mid 2^+)} + \B{H}_{(1^- \mid 2^+)} + \B{H}_{(1^+ \mid 2^-)} + \B{H}_{(1^- \mid 2^-)} \Big) \\
        \ee_{\{ 1^{\trv} \mid 2^{\sgn} \}} =
        \B{E}_{( 1^{\trv} \mid 2^{\sgn} )} & =
                    \dfrac{1}{4} \Big( \B{H}_{(1^+ \mid 2^+)} + \B{H}_{(1^- \mid 2^+)} - \B{H}_{(1^+ \mid 2^-)} - \B{H}_{(1^- \mid 2^-)} \Big) \\
        \ee_{\{ 1^{\sgn} \mid 2^{\trv} \}} =
        \B{E}_{( 1^{\sgn} \mid 2^{\trv} )} & =
                    \dfrac{1}{4} \Big( \B{H}_{(1^+ \mid 2^+)} - \B{H}_{(1^- \mid 2^+)} + \B{H}_{(1^+ \mid 2^-)} - \B{H}_{(1^- \mid 2^-)} \Big) \\
        \ee_{\{ 1^{\sgn} \mid 2^{\sgn} \}} =
        \B{E}_{( 1^{\sgn} \mid 2^{\sgn} )} & =
                    \dfrac{1}{4} \Big( \B{H}_{(1^+ \mid 2^+)} - \B{H}_{(1^- \mid 2^+)} - \B{H}_{(1^+ \mid 2^-)} + \B{H}_{(1^- \mid 2^-)} \Big)  \\
        \ee_{\{ 12^{\trv} \}} =
        \B{E}_{( 12^{\trv} )} \ee^\circ_{\{12\}} & =
                    \B{E}_{( 12^{\trv} )} - \B{E}_{( 1^{\trv} \mid 2^{\trv} )} - \B{E}_{( 1^{\sgn} \mid 2^{\sgn} )} \\ & =
                    \dfrac{1}{2} \Big( \B{H}_{(12^+)} + \B{H}_{(12^-)} - \B{H}_{(1^+ \mid 2^+)} - \B{H}_{(1^- \mid 2^-)} \Big)  \\   
        \ee_{\{ 12^{\sgn} \}} =
        \B{E}_{( 12^{\sgn} )} \ee^\circ_{\{12\}} & =
                    \B{E}_{( 12^{\sgn} )} - \B{E}_{( 1^{\sgn} \mid 2^{\trv} )} - \B{E}_{( 1^{\sgn} \mid 2^{\trv} )} \\ & =
                    \dfrac{1}{2} \Big( \B{H}_{(12^+)} - \B{H}_{(12^-)} - \B{H}_{(1^+ \mid 2^+)} + \B{H}_{(1^- \mid 2^-)} \Big).
    \end{align*}
\end{example}

\begin{remark} 
Unlike the case of LRBs, not all CSoPOI of $\kk \SS$ are of the form in \cref{t:CSoPOI-LRBaG}. For example, consider the following LRBaG $\SS$ with 5 elements.
\[
    \begin{gathered}
        \begin{tabular}{c|ccccc}
              & $1$    & $l^+$ & $l^-$ & $r^+$ & $r^-$ \\
        \hline                                                                                 
        $1$   & $1$    & $l^+$ & $l^-$ & $l^+$ & $r^-$ \\
        $l^+$ & $l^+$  & $l^+$ & $l^-$ & $l^+$ & $l^-$ \\
        $l^-$ & $l^-$  & $l^-$ & $l^+$ & $l^-$ & $l^+$ \\
        $r^+$ & $r^+$  & $r^+$ & $r^-$ & $r^+$ & $r^-$ \\
        $r^-$ & $r^-$  & $r^-$ & $r^+$ & $r^-$ & $r^+$ 
        \end{tabular}
    \end{gathered}
    \qquad\qquad\qquad
    \begin{gathered}
        \begin{tikzpicture}
            \draw[<->] (-60bp,0) -- (60bp,0);
            \draw [fill=black] (0,0) circle (.1);
            
            \node [] at (  0bp, -15bp) {$1$};
            \node [] at (-40bp,  15bp) {$l^+ \quad l^-$};
            \node [] at ( 40bp,  15bp) {$r^+ \quad r^-$};
    
            \draw [blue] (-60bp,  5bp) rectangle (-20bp, 25bp); \node [blue] at (-40bp,  35bp) {$G_{l^+} \cong \ztwo$};
            \draw [blue] ( 60bp,  5bp) rectangle ( 20bp, 25bp); \node [blue] at ( 40bp,  35bp) {$\ztwo \cong G_{r^+}$};
            \draw [blue] (-10bp, -5bp) rectangle ( 10bp,-25bp); \node [blue] at (-20bp, -15bp) {$G_{1}$};
        \end{tikzpicture}
    \end{gathered}
\]
Write $\wh{G_{l^+}} = \{ l^{\trv} , l^{\sgn} \}$ and $\wh{G_{r^+}} = \{ r^{\trv} , r^{\sgn} \}$.
Then, the following is a CSoPOI that is not of the form in \cref{t:CSoPOI-LRBaG}:
\[
    \B{E}_{l^{\trv}} = \dfrac{1}{2} \big( \B{H}_{l^+} + \B{H}_{l^-} \big)
    \qquad\qquad
    \B{E}_{r^{\sgn}} = \dfrac{1}{2} \big( \B{H}_{r^+} - \B{H}_{r^-} \big)
    \qquad\qquad
    \B{H}_1 - \B{E}_{l^{\trv}} - \B{E}_{r^{\sgn}}.
\]
\end{remark}

\section{LRBGs with symmetry}\label{sec:LRBGsymmetry}
We will now consider the case of LRBGs where there is a natural group action. 

\subsection{Motivation}
Studying LRBGs with a natural group action is motivated by a beautiful result of Bidigare \cite{bidigare} relating the face algebra of a hyperplane arrangement to Solomon's Descent algebra. Let $\A_W$ be the \defn{reflection arrangement} associated to a finite Coxeter group $W$, meaning that the hyperplanes $H \in \A_W$ are the fixed space of the reflections in $W$. Then $W$ acts on $\A_W$, and this action extends to the faces $\Sigma_{\A_W}$.
Moreover, this action satisfies
\[ w \cdot ({\sf F \cdot G}) = (w \cdot {\sf F}) \cdot (w \cdot {\sf G}). \]
Thus the algebra $\kk \Sigma_{\A_W}$ is a $W$-module, where 
\[ w \cdot \B{H}_x = \B{H}_{w\cdot x}.\]
On the other hand, Solomon proved in \cite{solomon} that any finite Coxeter group defines a subalgebra $\sol(W)$ of the group algebra $\kk W$ spanned by sums of elements in $W$ with the same \defn{Coxeter descent set}. 

In his thesis, Bidigare studied the subalgebra of $W$-invariants $\left( \kk \Sigma_{\A_W} \right)^W$, proving the following:
\begin{theorem}[Bidigare]\label{thm:bidigare}
    For any finite Coxeter group $W$, the algebra $\left( \kk \Sigma_{\A_W} \right)^W$ is anti-isomorphic to Solomon's Descent algebra $\sol(W)$.
\end{theorem}
In fact, we may take Bidigare's theorem as definition; through this lens, $\sol(W)$ is the precisely the opposite algebra of $\left( \kk \Sigma_{\A_W} \right)^W$.

Hsiao showed that studying the $\sym_n$-invariants of 
the algebra $\kk\hsiao$ yields a parallel relationship to an algebra called the \defn{Mantaci--Reutenauer algebra} $\MR_n[G]$, to be defined and discussed in \cref{sec:MRalg}. Before doing so, we will define notation and conventions for studying the invariant algebras of any LRBG algebra $\kk \SS$ with symmetry. 

\subsection{Orbits and the type map}\label{sec:orbitstypemap}

Let $\SS$ be a LRBG with a left action of a finite group $W$ satisfying
\begin{equation}\label{eq:action-endo}
    w \cdot (s t) = (w \cdot s)(w \cdot t)
    \qquad\qquad
    \text{for all $s,t \in \SS$ and $w \in W$.}
\end{equation}
As we will see shortly, this property induces actions of $W$ on the semigroups $E(\SS)$, $\TT$, and $E(\TT)$,
as well as on the sets $\wh{\SS}$ and $\wh{\TT}$ when $\SS$ is a LRBaG.
These actions are compatible with the corresponding support maps and yield commutative diagrams
\begin{equation}\label{eq:typesuppDiagrams}
    \begin{tikzcd}
        {{\SS}\,\,\!_{\red{s,t}}} \arrow[d, "\supp"'] \arrow[r, "\type"] & {{\SS}^W\,\,\!_{\red{p,q}}} \arrow[d, "\supp"] \\
        {{\TT}\,\,\!_{\red{S,T}}} \arrow[r, "\type"']                    & {{\TT}^W\,\,\!_{\,}}                          
    \end{tikzcd}
    \hspace{2em}
    \begin{tikzcd}
    E(\SS)\,\,\!_{\red{x,y}} \arrow[r, "\type"] \arrow[d, "\supp"']         &       E(\SS)^W \arrow[d, "\supp"] \\
    E(\TT)\,\,\!_{\red{X,Y}} \arrow[r, "\type"']                              &       E(\TT)^W
    \end{tikzcd}
    \hspace{2em}
    \begin{tikzcd}
        {\wh{\SS}\,\,\!_{\red{\phi,\psi}}} \arrow[d, "\supp"'] \arrow[r, "\type"] & {\wh{\SS}^W\,\,\!_{\red{\alpha,\beta}}} \arrow[d, "\supp"] \\
        {\wh{\TT}\,\,\!_{\red{\Phi,\Psi}}} \arrow[r, "\type"']                    & {\wh{\TT}^W\,\,\!_{\red{\lambda}}}
    \end{tikzcd}
\end{equation}
where the superscript $W$ denotes the set of $W$-orbits and the maps labeled
$\type$ send an element to its $W$-orbit. We have indicated in red the symbols
that we use for a typical element of each set.

We discuss some consequences of the property \cref{eq:action-endo}.
First, we have $(w \cdot s)^2 = w \cdot s^2$, so the action of $W$ on $\SS$ restricts to an action on the LRB $E(\SS)$.
Moreover, since the support relation $\sim$ is defined in terms of the semigroup product (\cref{def:sLRBGrelations}),
the action of $W$ descends to the quotients $\TT = \SS/\mathord{\sim}$ and $E(\TT) = E(\SS)/\mathord{\sim}$, and it is compatible with the support maps.
Also, each $w \in W$ yields group isomorphisms $G_x \to G_{w \cdot x} : s \mapsto w \cdot s$ for all $x \in E(\SS)$.

When $\SS$ is a LRBaG,
the action of $W$ on $\wh{\SS}$ is defined as follows: given $x \in E(\SS)$, $\phi \in \wh{G_x}$, and $w \in W$, the element $w \cdot \phi \in \wh{G_{w \cdot x}}$ is the image of $\phi$ under the map dual to the isomorphism $G_{w \cdot x} \to G_x : s \mapsto w^{-1} \cdot s$.
That is,
\[
    (w \cdot \phi)(s) = \phi(w^{-1} \cdot s) \qquad\text{for all } s \in G_{w \cdot x}.
\]
That this action is compatible with $\supp: \wh{\SS} \to \wh{\TT}$ follows from the identity
\[
    w \cdot (\phi \circ \lambda_{y,x}) = (w \cdot \phi) \circ \lambda_{w \cdot y , w \cdot x},
\]
where $\lambda_{y, x} : G_{x} \to G_{yx}$ is the group morphism
induced by left-multiplication by $y$ (\cref{left-multiplication-induces-group-morphism}).
Finally, the partial order relations $\leq$ of $\SS$ and $\chleq$ of $\wh{\SS}$ are compatible with the $W$-action.
Thus, there are induced partial orders on the orbit spaces $\SS^W$ and $\wh{\SS}^W$ defined by
\begin{align*}
    p \leq q &\qqiff \text{there are $s \leq t$ in $\SS$ with $\type(s) = p$ and $\type(t) = q$},\\
    \alpha \chleq \beta &\qqiff \text{there are $\phi \chleq \psi$ in $\wh{\SS}$ with $\type(\phi) = \alpha$ and $\type(\psi) = \beta$}.
\end{align*}

\begin{example}\label{ex:braidtype}
    Perhaps the most illustrative example of the $\type$ map comes from considering the LRB of set compositions $\Sigma_n$ and its support lattice $\Pi_n$, the lattice of set partitions. The symmetric group $\sym_n$ naturally acts on the subsets of $[n]$, and this action extends to both $\Sigma_n$ and $\Pi_n$:
    \[
        w \cdot ( S_1 , \dots , S_k ) = ( w \cdot S_1 , \dots , w \cdot S_k )
        \qqand
        w \cdot \{ S_1 , \dots , S_k \} = \{ w \cdot S_1 , \dots , w \cdot S_k \}.
    \]
    Two compositions $( S_1 , \dots , S_k )$ and $( T_1 , \dots , T_\ell )$ are in the same orbit if and only if $k = \ell$ and $|S_i| = |T_i|$ for all $i$.
    Thus, we identify the orbits $(\Sigma_n)^{\sym_n}$ with the set $\Gamma_n$ of \defn{(integer) compositions of $n$},
    and define $\type : \Sigma_n \to \Gamma_n$ by
    \begin{equation*}
        \type\big(( S_1 , \dots , S_k )\big) = ( |S_1| , \dots , |S_k| );
    \end{equation*}
    for example,
    \[ \type\big( (678 | 12 | 345) \big) = (3,2,3). \]
    We similarly identify the orbits $(\Pi_n)^{\sym_n}$ with the set $\Lambda_n$ of \defn{(integer) partitions of $n$}, and define the corresponding type map $\type : \Pi_n \to \Lambda_n$ in an analogous way.
    Continuing with the example above, we have
    \begin{align*}
        \supp\big( (678 | 12 | 345)\big) = \{ 678 | 345 | 12 \}
        \qqand
        \type\big(\{ 678 | 345 | 12 \}\big) = \supp\big((3,2,3)\big) = \{3, 3, 2\}.
    \end{align*}
\end{example}

\begin{example} \label{ex:invariantshsiao}
    We now consider the semigroup $\hsiao$ of $G$-compositions of $[n]$ from \cref{ss:Hsiao's}.
    The symmetric group $\sym_n$ acts on $\hsiao$ by acting on the underlying set composition:
    \[
        w \cdot \big( (S_1,g_1), \cdots, (S_k, g_k) \big) := \big( (w \cdot S_1,g_1), \cdots, (w \cdot S_k, g_k) \big).
    \]
    Observe that the group elements are not modified.
    
    The symmetric group $\sym_n$ acts on $\gsetpartition$, $\Sigma_n[\wh{G}]$, and $\Pi_n[\wh{G}]$ in the same way:
    by acting on the underlying set composition/partition and leaving the group elements/characters intact.
    These actions are precisely the ones arising from the semigroup structure of $\hsiao$.
    
    Extending \cref{ex:braidtype}, we identify the orbits
    $(\hsiao)^{\sym_n}$ with the set $\gcomposition$ of $G$-compositions of~$n$;
    $(\gsetpartition)^{\sym_n}$ with the set $\gpartition$ of $G$-partitions of $n$;
    $(\Sigma_n[\wh{G}])^{\sym_n}$ with the set $\Gamma_n[\wh{G}]$ of $\wh{G}$-compositions of $n$;
    $(\Pi_n[\wh{G}])^{\sym_n}$ with the set $\Lambda_n[\wh{G}]$ of $\wh{G}$-partitions of $n$.
    \cref{fig:hsiaotypediagram} illustrates these sets and the maps between them.
    We will return to this example in \cref{sec:hsiaosalgebra}.

    \begin{figure}[!ht]
    \begin{center}
        \begin{tikzcd}[column sep=-2ex]
            \begin{array}{@{}c@{}} \gsetcomposition\!:\!{\footnotesize  \text{$G$-compositions of $[n]$}} \\ {\scriptsize\red\text{$ \big( 246^{g_1} \mid 3^{g_2} \mid 15^{g_3} \big)$}} \end{array}
                \arrow[rr, "\type"]
                \arrow[dd, "\supp"']
                \arrow[rd, "\abs{\cdot}"'] &&
            \begin{array}{@{}c@{}} \gcomposition\!:\!{\footnotesize  \text{$G$-compositions of $n$}} \\ {\scriptsize\red\text{$ \big( 3^{g_1} , 1^{g_2} , 2^{g_3} \big)$}} \end{array}
                \arrow[dd]
                \arrow[rd] \\
            &
            \begin{array}{@{}c@{}} \Sigma_n\!:\!{\footnotesize  \text{compositions of $[n]$}} \\ {\scriptsize\red\text{$ \big( 246 \mid 3 \mid 15 \big)$}} \end{array}
                \arrow[rr, crossing over] &&
            \begin{array}{@{}c@{}} \Gamma_n\!:\!{\footnotesize  \text{compositions of $n$}} \\ {\scriptsize\red\text{$ (3,1,2)$}} \end{array}
                \arrow[dd] \\
            \begin{array}{@{}c@{}} \gsetpartition\!:\!{\footnotesize  \text{$G$-partitions of $[n]$}} \\ {\scriptsize\red\text{$ \big\{ 15^{g_3} \mid 246^{g_1} \mid 3^{g_2} \big\}$}} \end{array}
                \arrow[rr]
                \arrow[rd] &&
            \begin{array}{@{}c@{}} \gpartition\!:\!{\footnotesize  \text{$G$-partitions of $n$}} \\ {\scriptsize\red\text{$ \big\{ 3^{g_1} , 2^{g_3} , 1^{g_2} \big\}$}} \end{array}
                \arrow[rd] \\
            &
            \begin{array}{@{}c@{}} \Pi_n\!:\!{\footnotesize  \text{partitions of $[n]$}} \\ {\scriptsize\red\text{$ \big\{ 15 \mid 246 \mid 3 \big\}$}} \end{array}
                \arrow[rr]
                \arrow[from=2-2, crossing over] &&
            \begin{array}{@{}c@{}} \Lambda_n\!:\!{\footnotesize  \text{partitions of $n$}} \\ {\scriptsize\red\text{$ \{3,2,1\}$}} \end{array}
        \end{tikzcd}
     \end{center}
     \caption{  The diagram shows the different semigroups arising from $\hsiao$ (on the left)
                and the corresponding sets of $\sym_n$-orbits (on the right).
                In red we show the image of the element $\big( 246^{g_1} \mid 3^{g_2} \mid 15^{g_3} \big) \in \Sigma_6[G]$ across the different combinations of support, type, and $\abs{\cdot}$ maps. We can construct a similar diagram where we replace $G$ by $\wh{G}$ (and $g_1,g_2,g_3 \in G$ by $\phi_1,\phi_2,\phi_3 \in \wh{G}$).
            }
     \label{fig:hsiaotypediagram}
    \end{figure}
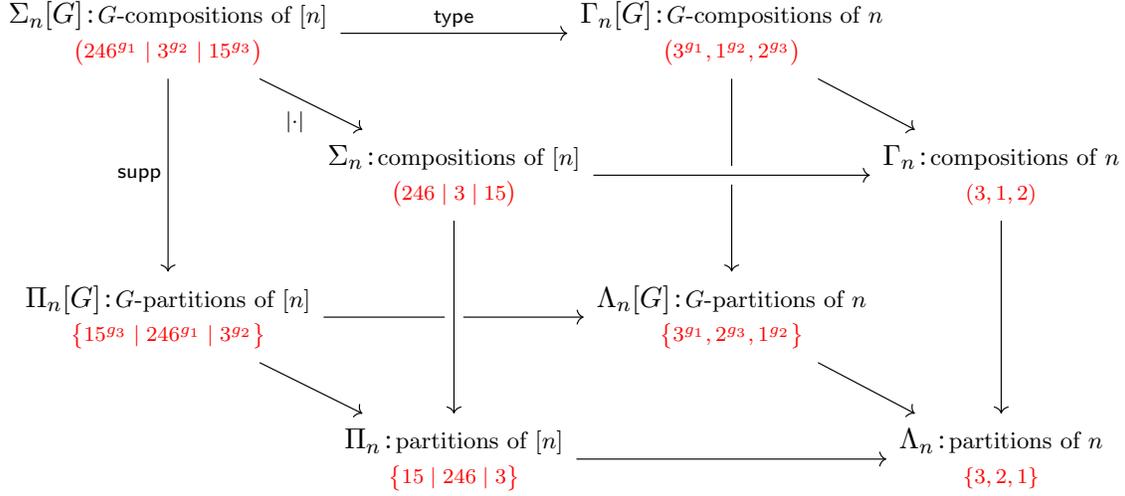
\end{example}

\subsection{The invariant subalgebra of a LRBaG}\label{sec:invarLRBaG}

Let $W$ be a finite group acting on a LRBaG $\SS$.
We extend this action to the semigroup algebra $\kk \SS$ by permuting the $\B{H}$-basis:
\[
    w \cdot \B{H}_s := \B{H}_{w \cdot s} \qquad\text{ for all } s \in \SS.
\]
The basis of local idempotents $\{\B{E}_\phi\}_{\phi \in \wh{\SS}}$ is also permuted by the $W$-action:
\[
    w \cdot \B{E}_\phi
    = w \cdot \Big( \dfrac{1}{|G_{\abs{\phi}}|} \sum_{s \in G_{\abs{\phi}}} \ol{\phi(s)} \B{H}_s \Big)
    = \dfrac{1}{|G_{\abs{\phi}}|} \sum_{s \in G_{\abs{\phi}}} \ol{\phi(s)} \B{H}_{w \cdot s}
    = \B{E}_{w \cdot \phi},
\]
where $w \cdot \phi$ corresponds to the $W$ action on $\wh{\SS}$,
and we used that $\phi(s) = (w \cdot \phi)(w \cdot s)$ for all $\phi \in \wh{\SS}$ and $s \in G_{\abs{\phi}}$.
Therefore, summing over $W$-orbits of the $\B{H}$-basis and the $\B{E}$-basis,
we obtain two different bases for the invariant subalgebra $(\kk \SS)^W$, one indexed by $\SS^W$ and one indexed by $\wh{\SS}^W$.
Explicitly, define $\{\SB{H}_p\}_{p \in \SS^W}$ and $\{ \SB{E}_\alpha \}_{\alpha \in \wh{\SS}^W}$ by:
\begin{equation}\label{eq:def-symH-basis}
        \SB{H}_p := \sum_{\type(s) = p} \B{H}_s
    \qqand
    \SB{E}_\alpha := \sum_{\type(\phi) = \alpha} \B{E}_\phi.
\end{equation}
\begin{example}
    In $\Sigma_3[G]$, we have
    \[
        \SB{H}_{( 2^g , 1^h )} =
        \B{H}_ {( 12^g \mid 3^h )} + 
        \B{H}_ {( 13^g \mid 2^h )} + 
        \B{H}_ {( 23^g \mid 1^h )},
    \]
    and
    \[
        \SB{E}_{( 2^\phi , 1^\psi )} =
        \B{E}_ {( 12^\phi \mid 3^\psi )} + 
        \B{E}_ {( 13^\phi \mid 2^\psi )} + 
        \B{E}_ {( 23^\phi \mid 1^\psi )}.
    \]
\end{example}

In order to make the construction of idempotents in \cref{thm:CSoPOI-LRBG} compatible with the $W$-action, we need to carefully choose the homogeneous section $\uu$.
We say that $\uu$ is a \defn{$W$-invariant homogeneous section of the support map} if it is a homogeneous section
$\{\uu_X\}_{X \in E(\TT)} \subseteq \kk E(\SS)$ such that
\[
    w \cdot \uu_X = \uu_{w \cdot X}
    \qquad
    \text{for all }w \in W \text{ and } X \in E(\TT).
\]
This is equivalent to the map $\uu : \kk E(\TT) \to \kk E(\SS)$ being $W$-equivariant.

Let $\uu$ be a $W$-invariant homogeneous section.
It then follows from the definitions that the corresponding elements $\ee^\circ$, $\B{Q}^\circ$, $\B{R}$, and $\B{Q}$ satisfy
\begin{equation}
    w \cdot \ee^\circ_X = \ee^\circ_{w \cdot X},
    \qquad
    w \cdot \B{Q}^\circ_x = \B{Q}^\circ_{w \cdot x},
    \qquad
    w \cdot \B{R}_s = \B{R}_{w \cdot s},
    \qqand
    w \cdot \B{Q}_\phi = \B{Q}_{w \cdot \phi},
\end{equation}
for all $w \in W$, $X \in E(\TT)$, $x \in E(\SS)$, $s \in \SS$, and $\phi \in \wh{\SS}$.
Thus we may define two more bases $\{ \SB{R}_p \}_{p \in \SS^W}$ and $\{ \SB{Q}_\alpha \}_{\alpha \in \wh{\SS}^W}$ of $(\kk \SS)^W$ by:
\[
    \SB{R}_p := \sum_{\type(s) = p} \B{R}_s
    \qqand
    \SB{Q}_\alpha := \sum_{\type(\phi) = \alpha} \B{Q}_\phi.
\]
The action of $W$ also permutes the CSoPOI $\{ \ee_\Psi \}_{\Psi \in \wh{\TT}}$ from \cref{eq:CSoPOI-LRBaG}:
\[
    w \cdot \ee_\Phi =
    w \cdot (\uu_{\abs{\Phi}} \B{Q}_\phi) = 
    (w \cdot  \uu_{\abs{\Phi}}) (w \cdot \B{Q}_\phi) = 
    \uu_{w \cdot\abs{\Phi}} \B{Q}_{w \cdot \phi} =
    \ee_{w \cdot \Phi}.
\]

Note that the sum of orthogonal idempotents is itself an idempotent. The following result is a consequence of \cref{t:CSoPOI-LRBaG}.

\begin{theorem}\label{thm:CSoPOI-invar-subalg}
    Let $\SS$ be a LRBaG with a group action by $W$. Given a $W$-invariant homogeneous section $\uu$, the corresponding collection 
    $\{ \ee_\lambda \}_{\lambda \in \wh{\TT}^W}$, where
    \begin{equation}\label{eq:CSoPOI-invariant}
        \ee_\lambda := \sum_{\type(\Phi) = \lambda} \ee_\Phi,
    \end{equation}
    is a complete system of primitive orthogonal idempotents of the invariant subalgebra $(\kk \SS)^W$.
\end{theorem}

We will see in \cref{sec:MRalg} that \cref{thm:CSoPOI-invar-subalg} has interesting consequences for the Mantaci--Reutenauer algebra. 

\begin{example}[Uniform section]\label{ex:uniform-general}
    Recall the definition of the uniform section of a LRB from \cref{ex:uniformsection}.
    Explicitly, the uniform section is determined by
    \[
        \uu_X = \dfrac{1}{c^X} \sum_{\supp(x) = X} \B{H}_x,
    \]
    where $c^X := |\{ x \in E(\SS) : \supp(x) = X \}|$.
    Thus, for all $\Phi \in \wh{\TT}$, \Cref{eq:CSoPOI-LRBaG} gives
    \[
        \ee_\Phi = \dfrac{1}{c^{\abs{\Phi}}} \sum_{\supp(\phi) = \Phi} \B{Q}_\phi.
    \]
   Therefore
    \[
        \ee_\lambda = \dfrac{1}{c^\lambda} \sum_{\type(\supp(\phi)) = \lambda} \B{Q}_\phi = \dfrac{1}{c^\lambda} \sum_{\supp(\alpha) = \lambda} \SB{Q}_\alpha
    \]
    for all $\lambda \in \wh{\TT}^W$, where $c^\lambda = c^{\abs{\Phi}}$ for any (all) $\Phi$ of type $\lambda$.
\end{example}

\section{Hsiao's algebra}\label{sec:hsiaosalgebra}

We specialize the previous results to Hsiao's LRBG $\SS = \hsiao$ with its natural $\sym_n$-action.
While we will work with an arbitrary finite group $G$, we will give special attention to the context where:
\begin{itemize}[wide]
    \item $G$ is abelian, in which case $\SS$ is a LRBaG and
        we have a natural identification $\wh{\SS} = \Sigma_n[\wh{G}]$; and
    \item the homogeneous section $\uu$ is the uniform section.
\end{itemize}
For clarity, \Cref{fig:notationHsiao} illustrates the diagrams in \cref{eq:typesuppDiagrams} for Hsiao's semigroup, and recalls the various bases and CSoPOIs that have appeared throughout the document.

We begin by making explicit the partial order relations $\leq$ and $\chleq$ from \cref{sec:orbitstypemap} specialized to Hsiao's algebra.

\begin{definition}\label{def:orderforMR}
    For $p, q \in \Gamma_n[G]$, 
    \[
        p \leq q
        \qqiff
        \text{$q$ is a $G$-preserving refinement of $p$.}
    \]
    That is, the blocks of $q$ refining a given block of $p$ must be labeled by the same group element.
    For example,
    \[
        \big( 5^g , 1^h , 1^g \big) \leq \big( 2^g , 3^g , 1^h , 1^g \big),
    \]
    and, if $g \neq h$, the first element is minimal with respect to this order.
    
    For $\alpha, \beta \in \Gamma_n[\wh{G}]$, 
    \[
        \alpha \chleq \beta
        \qqiff
        \text{$\beta$ is a $\wh{G}$-multiplicative refinement of $\alpha$.}
    \]
    That is, the blocks of $\beta$ refining a given block of $\alpha$ must be labeled by characters whose product is the character labeling the given block of $\alpha$.
    For example,
    \[
        \big( 7^{\phi\psi\chi\eta} \big) \chleq \big( 5^{\phi\psi} , 1^\chi , 1^\eta \big) \chleq \big( 2^\phi , 3^\psi , 1^\chi , 1^\eta \big).
    \]
    In contrast to $\Gamma_n[G]$, the minimal elements of $\Gamma_n[\wh{G}]$ always have length one.

    Recall that for $p \in \Gamma_n[G]$ (resp. $\alpha \in \Gamma_n[\wh{G}]$), $\abs{p}$ (resp. $\abs{\alpha}$) denotes its underlying composition.
    A simple but important observation is that
    \[
        p \leq q \textrm{ implies } \abs{p} \leq \abs{q}
        \qquad\qquad
        \alpha \chleq \beta \textrm{ implies } \abs{\alpha} \leq \abs{\beta}.
    \]
\end{definition}

\begin{figure}[!ht]

\begin{centering}
    \begin{tikzcd}
        {{\hsiao}\,\,\!_{\red{s,t}}} \arrow[d, "\supp"'] \arrow[r, "\type"]   & {{\gcomposition}\,\,\!_{\red{p,q}}} \arrow[d, "\supp"] \\
        {{\gsetpartition}\,\,\!_{\red{S,T}}} \arrow[r, "\type"']              & {{\gpartition}\,\,\!_{\,}}                          
    \end{tikzcd}
    \hspace{2em}
    \begin{tikzcd}
        {{\Sigma_n}\,\,\!_{\red{x,y}}} \arrow[d, "\supp"'] \arrow[r, "\type"]   & {{\Gamma_n}} \arrow[d, "\supp"] \\
        {{\Pi_n}\,\,\!_{\red{X,Y}}} \arrow[r, "\type"']              & {{\Lambda_n}}                          
    \end{tikzcd}
    \hspace{2em}
    \begin{tikzcd}
        {\Sigma_n[\wh{G}]\,\,\!_{\red{\phi,\psi}}} \arrow[d, "\supp"'] \arrow[r, "\type"] & {\Gamma_n[\wh{G}]\,\,\!_{\red{\alpha,\beta}}} \arrow[d, "\supp"] \\
        {\Pi_n[\wh{G}]\,\,\!_{\red{\Phi,\Psi}}} \arrow[r, "\type"']                    & {\Lambda_n[\wh{G}]\,\,\!_{\red{\lambda}}}                        
    \end{tikzcd}

    \bigskip
    \bigskip

    \begin{tabular}{cccccc}
        \toprule
        Algebra                  & $\B{H}$-basis  & $\B{R}$-basis & $\B{E}$-basis & $\B{Q}$-basis & CSoPOI \\ [-0.5ex]
                                 & \footnotesize (Canonical) & & \footnotesize (Isotypic)               & \footnotesize (Primitive)              & \\ \midrule
        $\kk \Pi_n$              & $\B{H}_X$       &                                                  &                 & $\B{Q}^\circ_X := \sum_{Y \geq X} \B{H}_Y$                      & $\B{Q}^\circ_X$     \\
        $\kk \Sigma_n$           & $\B{H}_x$       &                                                  & $\B{H}_x$       & $\B{Q}^\circ_x : = \B{H}_x \ee^\circ_{\supp(x)}$          & $\ee^\circ_X$ \\
        $\kk \Pi_n[G]$           & $\B{H}_S$       & $\B{R}_S := \B{H}_S \B{Q}^\circ_{S^\omega}$            & $\B{E}_\Phi$    & $\B{Q}_\Phi := \B{E}_\Phi \B{Q}^\circ_{\abs{\Phi}}$             & $\B{Q}_\Phi$  \\
        $\kk\hsiao$              & $\B{H}_s$       & $\B{R}_s := \B{H}_s \ee^\circ_{\supp(s^\omega)}$ & $\B{E}_\phi$    & $\B{Q}_\phi := \B{E}_\phi \ee^\circ_{\supp(\abs{\phi})}$  & $\ee_\Phi$    \\
        $(\kk\hsiao)^{\sym_n}$   & $\SB{H}_p$      & $\SB{R}_p$                                       & $\SB{E}_\alpha$ & $\SB{Q}_\alpha$                                           & $\ee_\lambda$ \\ \bottomrule
    \end{tabular} 
\end{centering}
\caption{
        Typical elements of various semigroups/sets related to $\hsiao$, they index bases and CSoPOI of the different algebras listed in the table.
        Several of these elements depend on the choice of a homogeneous section $\{\uu_X\}_{X \in \Pi_n}$ of $\supp : \Sigma_n \to \Pi_n$.}
\label{fig:notationHsiao}
\end{figure}

In \cref{s:change-of-basis-hsiao}, we will give formulas for the elements in the table shown in \cref{fig:notationHsiao} in the case that $\{\uu_X\}_{X \in \Pi_n}$ is the uniform section.
We need the following definition first, which will  follow the notation in \cite{am17}.

\begin{definition}\label{def:deg}
   
    Given two integer compositions $p,q \in \Gamma_n$, with $q$ refining $p$, let $(q/p)_i$ denote the composition consisting of the contiguous blocks of $q$ refining the $i$-th block of $p$.
    Define  
    \[
        \deg(q/p) := \prod_i \ell( (q/p)_i) \qquad\text{and}\qquad \deg!(q/p) := \prod_i \ell( (q/p)_i)!
    \]
    We extend this definition to all $G$/$\wh{G}$/unlabelled set/integer compositions by applying it to their underlying integer compositions (using the $\type$ and $\abs{\cdot}$ maps if necessary).
\end{definition}

For instance, if $p = (2,6)$, and $q = (1,1,2,1,3)$, then $(q/p)_1 = (1,1)$, $(q/p)_2 = (2,1,3)$,
\[
    \deg(q/p) = 2 \cdot 3 = 6
    \qqand
    \deg!(q/p) = 2! \cdot 3! = 12.
\]
Similarly, if $\phi = \big( 14^{\trv} \mid 235678^{\trv} \big)$ and
$\psi = \big( 1^{\sgn} \mid 4^{\sgn} \mid 25^{\sgn} \mid 8^{\trv} \mid 367^{\sgn} \big)$,
then
\begin{equation*}
    (\phi/\psi)_{1} = \big( 1^{\sgn} \mid 4^{\sgn} \big)
    \qqand
    (\phi/\psi)_{2} = \big( 25^{\sgn} \mid 8^{\trv} \mid 367^{\sgn} \big)
\end{equation*}
so that
\[
    \deg(\psi/\phi) = 2 \cdot 3 = 6
    \qqand
    \deg!(\psi/\phi) = 2! \cdot 3! = 12.
\]
Observe that,
since $\abs{\phi} \leq \abs{\psi}$, ${\sgn} \cdot {\sgn} = {\trv}$ and ${\sgn} \cdot {\trv} \cdot {\sgn} = {\trv}$,
we have $\phi \chleq \psi$. Moreover, $\type(\abs{\phi}) = (2,6)$ and $\type(\abs{\psi}) = (1,1,2,1,3)$.

\subsection{Change of basis and idempotents in Hsiao's algebra} \label{s:change-of-basis-hsiao}

Let $\{\uu_X\}_{X \in \Pi_n}$ be the uniform section, as in \cref{ex:uniform-general}.
For Hsiao's algebra, this can be written explicitly as 
\[
    \uu_X = \dfrac{1}{\ell(X)!} \sum_{\supp(x) = X} \B{H}_x.
\]
Let $\{\B{Q}^\circ_x\}_{x \in \Sigma_n}$ be the $\B{Q}$-basis of the LRB algebra $\kk \Sigma_n$ associated to the uniform section.
Aguiar and Mahajan \cite[Lemma 12.74]{am17} give the following explicit change of basis formulas between the $\B{H}$-basis and $\B{Q}$-basis of $\kk \Sigma_n$.
For all $x \in \Sigma_n$,
\begin{equation}\label{eq:Q0-basis}
    \B{H}_x = \sum_{y \geq x} \dfrac{1}{\deg!(y/x)} \B{Q}^\circ_y
    \qqand
    \B{Q}^\circ_x = \sum_{y \geq x} \dfrac{(-1)^{\ell(y) - \ell(x)}}{\deg(y/x)} \B{H}_y.
\end{equation}

We will show that similar formulas give the change of basis between the $\B{H}$ and $\B{R}$ bases, and between the $\B{E}$ and $\B{Q}$ bases of $\kk \hsiao$.

\begin{prop}
    Let $G$ be a finite group. Let $\{\B{R}_s\}_{s \in \hsiao}$ be the $\B{R}$-basis of $\kk \hsiao$ associated to the uniform section.
    Then, for all $s \in \hsiao$,
    \begin{equation}\label{eq:Hsiao-H-to-R}
        \B{H}_s =
            \sum_{t \geq s} \dfrac{1}{\deg!(t/s)} \B{R}_t
        \qqand
        \B{R}_s =
            \sum_{t \geq s} \dfrac{(-1)^{\ell(t) - \ell(s)}}{\deg(t/s)} \B{H}_t.
    \end{equation}
\end{prop}

\begin{proof}
    We derive the formulas by multiplying the equations in
    \cref{eq:Q0-basis} by $\B{H}_s$ on the left, where $x = s^\omega$.

    Fix $s \in \hsiao$ and let $x = s^\omega$.
    One can verify directly that for all $y \in E(\SS)$ refining $x$, $sy = ys$ (alternatively, see \cref{lem:xf=fx}).
    Thus, for all such $y$ we have
    \[
        \B{H}_s \B{Q}^\circ_y = \B{H}_s \B{H}_y \B{Q}^\circ_y = \B{H}_{ys} \B{Q}^\circ_y = \B{R}_{ys},
    \]
    the last equality since $(ys)^\omega = y s^\omega = y$.
    Therefore,
    \[
        \B{H}_s = \B{H}_s \B{H}_x = \sum_{y \geq x} \dfrac{1}{\deg!(y/x)} \B{H}_s \B{Q}^\circ_y = \sum_{y \geq x} \dfrac{1}{\deg!(y/x)} \B{R}_{ys}.
    \]
    The characterization of the partial order $\leq$ on $\hsiao$ in \cref{ex:hsiao-order} shows that the elements $ys$ with $x \leq y$ are exactly those $t$ with $s \leq t$ (alternatively, see \cref{lem:sLRBG-leq}).
    This completes the proof of the first formula in \cref{eq:Hsiao-H-to-R}, the second follows by the same reasoning.
\end{proof}

\begin{prop}
    Let $G$ be a finite abelian group. Let $\{\B{Q}_\phi\}_{\phi \in \Sigma_n[\wh{G}]}$ be the $\B{Q}$-basis of $\kk \hsiao$ associated to the uniform section.
    Then, for all $\phi \in \Sigma_n[\wh{G}]$,
    \begin{equation}\label{eq:Hsiao-E-to-Q}
        \B{E}_\phi =
            \sum_{\psi \chgeq \phi} \dfrac{1}{\deg!(\psi/\phi)} \B{Q}_\psi
        \qqand
        \B{Q}_\phi =
            \sum_{\psi \chgeq \phi} \dfrac{(-1)^{\ell(\psi) - \ell(\phi)}}{\deg(\psi/\phi)} \B{E}_\psi.
    \end{equation}
\end{prop}

\begin{proof}
    The formulas are obtained from \cref{eq:Q0-basis} by multiplying both sides of the equations on the left by $\B{E}_\phi$, where $x = \abs{\phi}$,
    and following the same ideas of the previous proof. An important detail here is that $\B{E}_\phi$ is a linear combination of elements $\B{H}_s$ with $s^\omega = x$.
    Thus, for all $y \geq x$, 
    \[
        \B{E}_\phi \B{H}_y          =
        \B{H}_y \B{E}_\phi          \overset{\cref{eq:EandHmultiplication}}{=}
        \sum_{ \substack{ \psi \chgeq \phi \\ \abs{\psi} = y } } \B{E}_\psi
        \qqand
        \B{E}_\phi \B{Q}^\circ_y          =
        \B{E}_\phi \B{H}_y\B{Q}^\circ_y          =
        \sum_{ \substack{ \psi \chgeq \phi \\ \abs{\psi} = y } } \B{E}_\psi \B{Q}^\circ_y =
        \sum_{ \substack{ \psi \chgeq \phi \\ \abs{\psi} = y } } \B{Q}_\psi. \qedhere
    \]
\end{proof}

Recall that for any LRBaG and any homogeneous section, the $\B{Q}$-basis can be written in the $\B{R}$-basis as follows (see \Cref{eq:LRBaG-R-to-Q})
\[
    \B{Q}_\phi = 
            \dfrac{1}{|G_{\abs{\phi}}|} \sum_{\abs{s} = \abs{\phi}} \ol{\phi(s)} \B{R}_s,
\]
where $\phi(s)$ is the evaluation of the character $\phi \in \wh{G_{\abs{\phi}}}$ on the group element $s \in G_{\abs{\phi}}$ (see \cref{def:evaluation}), and the bar denotes complex conjugation.

Using \cref{eq:Hsiao-H-to-R} to expand $\B{R}_s$ in the $\B{H}$-basis, we deduce the following.

\begin{prop}
    Let $G$ be a finite abelian group. Let $\{\B{Q}_\phi\}_{\phi \in \Sigma_n[\wh{G}]}$ be the $\B{Q}$-basis of $\kk \hsiao$ associated to the uniform section.
    Then, for all $\phi \in \Sigma_n[\wh{G}]$,
    \begin{equation}\label{eq:Hsiao-H-to-Q}
        \B{Q}_\phi =
            \dfrac{1}{|G|^{\ell(\phi)}} \sum_{\abs{s} = \abs{\phi}} \ol{\phi(s)} \sum_{t \geq s} \dfrac{(-1)^{\ell(t)-\ell(s)}}{\deg(t/s)} \B{H}_t.
    \end{equation}
\end{prop}

\begin{remark}
    If instead we use the second equation in \cref{eq:Hsiao-E-to-Q} and expand $\B{E}_\psi$ in $\B{H}$-basis, we obtain a different expression for~$\B{Q}_\phi$ in the $\B{H}$-basis:
    \[
        \B{Q}_\phi = \sum_{\psi \chgeq \phi} \dfrac{(-1)^{\ell(\psi) - \ell(\phi)}}{\deg(\psi/\phi)} \dfrac{1}{|G|^{\ell(\psi)}} \sum_{\abs{t} = \abs{\psi}} \ol{\psi(t)} \B{H}_t.
    \]
    A direct proof that these two different expressions for $\B{Q}_\phi$ are equal can be deduced using the following property of the characters of finite abelian groups.
    If $G$ is an abelian group, $H \leq G$ a subgroup, and $\phi \in \wh{H}$, then
    \[
        \sum_{\psi \in \wh{G} : \psi|_H = \phi} \psi(g) =
        \begin{cases}
            [G:H] \phi(g)      & \text{if } g \in H,\\
            0                           & \text{otherwise.}
        \end{cases}
    \]
\end{remark}

Finally, we obtain formulas for $\ee_\Phi$ in terms of the $\B{R}$ and $\B{E}$ bases. 
\begin{cor}
    Let $G$ be a finite abelian group. Let $\{\B{R}_s\}_{s \in \hsiao}$ be the $\B{R}$-basis of $\kk \hsiao$ associated to the uniform section.
    Then, for all and $\Phi \in \Pi_n[\wh{G}]$,
    \[
        \ee_\Phi = 
            \dfrac{1}{\ell(\Phi)! \, |G|^{\ell(\Phi)}} \sum_{\abs{\supp(s)} = \abs{\Phi}} \ol{\Phi(\supp(s))} \B{R}_s
        \qqand
        \ee_\Phi = 
            \dfrac{1}{\ell(\Phi)!} \sum_{\supp(\phi) = \Phi} \sum_{\psi \chgeq \phi} \dfrac{(-1)^{\ell(\psi) - \ell(\phi)}}{\deg(\psi/\phi)} \B{E}_\psi.
    \]
\end{cor}

Note that the evaluation $\Phi(\supp(s))$ is well-defined since $\Phi$ and $\supp(s)$ are $G$-partitions and $\wh{G}$-partitions of $[n]$, and thus their blocks (sets) are unambiguously paired. The same would not be true for $G$-partitions and $\wh{G}$-partitions of $n$.

We conclude this section with some formulas for the $\SB{R}$ and $\SB{Q}$ bases of the invariant subalgebra $(\kk\hsiao)^{\sym_n}$.
These will play an important role in \cref{sec:MRalg}.

\begin{prop}\label{p:SH-to-SR}
    Let $G$ a finite group. Let $\{\SB{R}_p\}_{p \in \Gamma_n[G]}$ be the $\SB{R}$-basis of $(\kk\hsiao)^{\sym_n}$ associated to the uniform section.
    Then, for all $p \in \Gamma_n[G]$,
    \begin{equation}\label{SH-to-SR}
        \SB{R}_p = \sum_{q \geq p} \dfrac{(-1)^{\ell(q)-\ell(p)}}{\deg(p/q)} \SB{H}_q.
    \end{equation}
\end{prop}

\begin{proof}
    Substituting \cref{eq:Hsiao-H-to-R} in the definition of $\SB{R}_p$ one obtains
    \[
        \SB{R}_p = \sum_{\type(s) = p} \sum_{t \geq s} \dfrac{(-1)^{\ell(t)-\ell(s)}}{\deg(t/s)} \B{H}_t.
    \]
    Observe that for each $q \in \Gamma_n[G]$ with $q \geq p$ and each $t \in \hsiao$ with $\type(t) = q$, there exists exactly one $s \in \hsiao$ with $s \leq t$ and $\type(s) = p$: if
    two adjacent blocks of $q$ refine the same block of $p$, then merge the corresponding blocks of $t$. Thus, for each $q \in \Gamma_n[G]$ with $q \geq p$, the coefficient of $\B{H}_t$ in the expression for $\SB{R}_p$ above is equal for all $t$ with $\type(t) = q$.
    Grouping terms and using that $\ell(s) = \ell(\type(s))$ and $\deg(t/s) = \deg(\type(t)/\type(s))$ yields the result.
\end{proof}

A similar argument proves the following.

\begin{prop}\label{p:SR-to-SQ}
    Let $G$ be a finite abelian group. Let $\{\SB{R}_p\}_{p \in \Gamma_n[G]}$ and $\{\SB{Q}_\phi\}_{\phi \in \Gamma_n[\wh{G}]}$ be the $\SB{R}$ and $\SB{Q}$ bases of $(\kk\hsiao)^{\sym_n}$ associated to the uniform section.
    Then, for all $\alpha \in \Gamma_n[\wh{G}]$,
    \begin{equation}\label{eq:HsiaoSymQbasis}
        \SB{Q}_\alpha =
    \dfrac{1}{|G|^{\ell(\alpha)}} \sum_{\abs{p} = \abs{\alpha}} \ol{\alpha(p)} \SB{R}_p
    \qqand
        \SB{Q}_\alpha = \sum_{\beta \chgeq \alpha} \dfrac{(-1)^{\ell(\beta) - \ell(\alpha)}}{\deg(\beta/\alpha)} \SB{E}_\beta.
    \end{equation}
\end{prop}

Finally, we can use the previous results to write the CSoPOI $\{ \ee_\lambda \}_{\lambda \in \Lambda[\wh{G}]} $ in the $\SB{H}$-basis.

\begin{cor}\label{cor:uniform-CSoPOI-Hsiao}
    Let $G$ be a finite abelian group. Let $\{\ee_\lambda\}_{\lambda \in \Lambda[\wh{G}]}$ be the complete system of primitive orthogonal idempotents of $(\kk \hsiao)^{\sym_n}$ associated to the uniform section.
    Then, for all $\lambda \in \Lambda[\wh{G}]$,
    \begin{align*}
        \ee_\lambda &= \frac{1}{\ell(\lambda)!} \sum_{\substack{ \alpha \in \Gamma_n[\wh{G}] \\ \supp(\alpha) = \lambda}} \sum_{\substack{\beta \in \Gamma_n[\wh{G}] \\ \alpha \ \chleq \ \beta}} \dfrac{(-1)^{\ell(\beta)-\ell(\alpha)}}{\deg(\beta/\alpha)} \cE_\beta\\
        &=\dfrac{1}{|G|^{\ell(\lambda)} \ell(\lambda)!} \sum_{\substack{ \alpha \in \Gamma_n[\wh{G}] \\ \supp(\alpha) = \lambda}} \sum_{\substack{p \in \Gamma_n[G] \\ \abs{p} = \abs{\alpha} }} \ol{\alpha(p)} \sum_{\substack{q \in \Gamma_n[G] \\ p \leq q }} \dfrac{(-1)^{\ell(q)-\ell(p)}}{\deg(q/p)} \SB{H}_q.
    \end{align*}
\end{cor}

\subsection{External product on Hsiao's algebra}\label{sec:externalproductonHsiao}

For any finite group $G$, we define an external product on the space $\kk\Sigma[G]$ of all $G$-compositions:
\[
    \kk\Sigma[G] := \bigoplus_{n \geq 0} \kk\Sigma_n[G].
\]

Given a set composition $x$ of a nonempty subset $S \subseteq \mathbb{N}$ of size $n$, its \defn{standardization} ${\rm std}(x) \in \Sigma_n$ is the composition of $[n]$ induced by the only order preserving bijection $S \to [n]$. For example,
\[
    {\rm std}\Big( \big( 368 \mid 15 \mid 9 \big) \Big) = \big( 245 \mid 13 \mid 6\big).
\]
We extend this definition to $G$-compositions (resp. $\wh{G}$-compositions) by leaving the group elements (resp. characters) unchanged.
For example,
\[
    {\rm std}\Big( \big( 368^{g_1} \mid 15^{g_2} \mid 9^{g_3} \big) \Big) = \big( 245^{g_1} \mid 13^{g_2} \mid 6^{g_3} \big).
\]
Given two $G$-compositions $s,t$ (resp. $\wh{G}$-compositions $\phi,\psi$) of disjoint subsets $S,T \subseteq \mathbb{N}$, we let $(s,t)$ (resp. $(\phi,\psi)$) denote their concatenation.
For example, if
\[
    \phi = \big( 368^{\phi_1} \mid 15^{\phi_2} \big)
    \text{~and~}
    \psi = \big( 29^{\psi_1} \mid 7^{\psi_2} \big),
    \text{~then~}
    (\phi,\psi) = \big( 368^{\phi_1} \mid 15^{\phi_2} \mid 29^{\psi_1} \mid 7^{\psi_2} \big).
\]

\begin{definition}\label{def:staronhsiao}
    Given $s \in \Sigma_n[G]$ and $t \in \Sigma_m[G]$, define
    \[
	\B{H}_s \star \B{H}_t := \sum_{ \substack{ (u,v) \in \Gamma_{n+m}[G] \\ {\rm std}(u)=s,\, {\rm std}(v)=t } } \B{H}_{(u,v)}.
    \]
\end{definition}

\begin{example}
    With $G = \ztwo$, we have
    \[		
        \B{H}_{({13}^-,2^+)} \star \B{H}_{(1^-)} = \B{H}_{(13^-,2^+,4^-)} + \B{H}_{(14^-,2^+,3^-)} + \B{H}_{(14^-,3^+,2^-)} + \B{H}_{(24^-,3^+,1^-)}.
    \]
\end{example}

\begin{prop}
    Extending the $\star$ product linearly endows $\kk \Sigma[G]$ with the structure of a graded associative unital algebra.
\end{prop}

\begin{proof}
    The unit of $\kk \Sigma[G]$ under the $\star$ product is $\B{H}_{()}$, where $() \in \Sigma_0[G]$ denotes the only $G$-composition of the empty set.
    The associativity is easily verified in the $\B{H}$-basis since (i.) concatenation of $G$-compositions is associative and (ii.) if $(s,t) = {\rm std}(u,v)$, where $\ell(s) = \ell(u)$ and $\ell(t) = \ell(v)$, and $r = {\rm std}(s)$, then $r = {\rm std}(u)$.
\end{proof}

\begin{remark}
    When $G$ is the trivial group, $\kk\Sigma[G]$ is the dual of the coalgebra of set compositions introduced by Chapoton \cite{Chapoton00} and which has been further studied in relation with Coxeter groups and Species in \cite{am06,am10}.
\end{remark}

When $G$ is abelian, the $\star$ product also behaves well with respect to the $\B{E}$-basis.

\begin{prop}
    Let $G$ be a finite abelian group. Then, for all $\phi \in \Sigma_n[\wh{G}]$ and $\psi \in \Sigma_m[\wh{G}]$,
    \begin{equation}\label{eq:star-E-basis}
        \B{E}_{\phi} \star \B{E}_{\psi} = \sum_{ \substack{ (\varphi,\chi) \in \Gamma_{n+m}[\wh{G}] \\ {\rm std}(\varphi)=\phi,\, {\rm std}(\chi)=\phi } } 
	\B{E}_{(\varphi,\chi)}.
    \end{equation}
\end{prop}

\begin{proof}
    Note that $G_{\abs{\phi}} = \{ s \in \SS : \abs{s} = \abs{\phi} \} \cong G^{\ell(\phi)}$. Thus,
    \[
        \B{E}_\phi \star \B{E}_\psi =
        \dfrac{1}{|G|^{\ell(\phi)+\ell(\psi)}} \sum_{ \substack{ s,t: \\ \abs{s} = \abs{\phi} \\ \abs{t} = \abs{\psi} } } \ol{\phi(s)}  \, \ol{\psi(t)} \B{H}_s \star \B{H}_t =
        \dfrac{1}{|G|^{\ell(\phi)+\ell(\psi)}} \sum_{ \substack{ u,v: \\ \abs{ {\rm std}(u) } = \abs{\phi} \\ \abs{{\rm std}(v)} = \abs{\psi} } } \ol{\phi({\rm std}(u))}  \, \ol{\psi({\rm std}(v))} \B{H}_{(u,v)}.
    \]
    On the other hand,
    \[
        \sum_{ \substack{ \varphi,\chi : \\ {\rm std}(\varphi) = \phi \\ {\rm std}(\chi) = \psi } }  \B{E}_{(\varphi,\chi)} = 
        \dfrac{1}{|G|^{\ell(\phi)+\ell(\psi)}} \sum_{\substack{ \varphi,\chi,u,v : \\ {\rm std}(\varphi) = \phi \\ {\rm std}(\chi) = \psi \\ \abs{u} = \abs{\varphi} \\ \abs{v} = \abs{\chi} }} \ol{\varphi(u)}  \, \ol{\chi(v)} \B{H}_{(u,v)}.
    \]
    Finally, observe that the sums in the two previous equations are the same since (i.) $u$ and $\phi$ (resp. $v$ and $\psi$) completely determine $\varphi$ (resp. $\chi$) and (ii.) $\varphi(u) = \phi({\rm std}(u))$ (resp. $\chi(v) = \psi({\rm std}(v))$).
\end{proof}

\begin{prop}\label{prop:invar-star-subalg}
    Let $G$ be a finite group.
    The subspace of invariants 
    \[
        (\kk \Sigma[G])^{\sym} := \displaystyle\bigoplus_{n \geq 0} (\kk\Sigma_n[G])^{\sym_n}
    \]
    is a subalgebra of $\kk\Sigma[G]$ under the $\star$ product.
    In fact, for all $p \in \Gamma_n[G]$ and $q \in \Gamma_m[G]$,
    \[
        \SB{H}_{p} \star \SB{H}_{q} = \SB{H}_{pq},
    \]
    where $pq$ denotes the concatenation of (integer) $G$-compositions $p$ and $q$.
\end{prop}

\begin{proof}
    The type of a $G$-composition of any subset $S \subseteq \mathbb{N}$ of size $n$ is
    \[
        \type\big( (S_1,g_1),\dots,(S_k,g_k) \big) = \big( (|S_1|,g_1),\dots,(|S_k|,g_k) \big) \in \gcomposition.
    \]
    In this manner, the type of a $G$-composition is the same as the type of its standardization,
    and the type of a concatenation of $G$-compositions is the concatenation of their types.
    That is,
    \[ 
        \type(s) = \type({\rm std}(s))
        \qqand
        \type((s,t)) = \type(s)\type(t).
    \]
    Thus, using the definition of the $\SB{H}$-basis in \Cref{eq:def-symH-basis}, we have
    \[
        \SB{H}_p \star \SB{H}_q = 
        \sum_{ \substack{ \type(s) = p \\ \type(t) = q } } \B{H}_s \star \B{H}_t =
        \sum_{ \substack{ \type({\rm std}(s)) = p \\ \type({\rm std}(t)) = q } } \B{H}_{(s,t)} =
        \sum_{ \type((s,t)) = p q } \B{H}_{(s,t)} = 
        \SB{H}_{p q}. \qedhere
    \]
\end{proof}

When $G$ is abelian, the $\star$ product also has a nice formula in the $\SB{E}$-basis of $(\kk \Sigma[G])^{\sym}$.
We deduce the formula below following the argument in the last proof and using \cref{eq:star-E-basis}.

\begin{prop}
   Let $G$ be a finite abelian group.
    For all for all $\alpha \in \Gamma_n[\wh{G}]$ and $\beta \in \Gamma_m[\wh{G}]$,
    \[
        \SB{E}_\alpha \star \SB{E}_\beta = \SB{E}_{\alpha\beta},
    \]
    where $\alpha\beta$ denotes the concatenation of the (integer) $\wh{G}$-compositions $\alpha$ and $\beta$.
\end{prop}

We now consider the collection of uniform sections, one for each value of $n$, to define the $\SB{R}$-basis and, when $G$ is abelian, the $\SB{Q}$-basis of $(\kk \Sigma[G])^{\sym}$.

\begin{prop}
    Let $G$ be a finite group.
    For each $n$, let $\{\SB{R}_p\}_{p \in \Gamma_n[{G}]}$ be the $\SB{R}$-basis of $(\kk\hsiao)^{\sym_n}$ associated to the uniform section.
    Then, for all $p \in \Gamma_n[{G}]$ and $q \in \Gamma_m[{G}]$,
    \[
        \SB{R}_p \star \SB{R}_q = \SB{R}_{pq}.
    \]
\end{prop}

\begin{proof}
    Suppose we have four integer ${G}$-compositions satisfying $p \leq q$ and $p' \leq q'$.
    It follows that $p p' \leq q q'$ and
    \[
        \deg(q q' / p p') = \deg(q / p) \deg(q' / p')
        \qquad
        \ell(q q') - \ell(p p') = 
        (\ell(q) - \ell(p)) + (\ell(q') - \ell(p')).
    \]
    Moreover, each integer ${G}$-composition refining $p p'$ is of the form $q q'$ above.
    Applying this and $\SB{H}_p \star \SB{H}_q = \SB{H}_{pq}$ to \cref{SH-to-SR} yields the result.
\end{proof}

When $G$ is abelian, the same argument above, but using that $\SB{E}_\alpha \star \SB{E}_\beta = \SB{E}_{\alpha\beta}$ and the second formula in \cref{eq:HsiaoSymQbasis} in the last step, yields the following result.

\begin{prop}\label{prop:SymQ-basis-star}
    Let $G$ be a finite abelian group.
    For each $n$, let $\{\SB{Q}_\alpha\}_{\alpha \in \Gamma_n[\wh{G}]}$ be the $\SB{Q}$-basis of $(\kk\hsiao)^{\sym_n}$ associated to the uniform section.
    Then, for all $\alpha \in \Gamma_n[\wh{G}]$ and $\beta \in \Gamma_m[\wh{G}]$,
    \[
        \SB{Q}_\alpha \star \SB{Q}_\beta = \SB{Q}_{\alpha \beta}.
    \]
\end{prop}

\section{The Mantaci--Reutenauer algebra and Vazirani's idempotents}\label{sec:MRalg}
In this section, we will study the Mantaci--Reutenauer algebra, which Hsiao related to the invariant algebra $(\kk \hsiao)^{\sym_n}$. We will first define the Mantaci--Reutenauer algebra in \cref{s:MRalgebra} for an arbitrary finite group $G$. We will then specialize to the case that $G = \ztwo$, which is of particular interest, and appears in the work of \cite{vazirani, DouTom2018decomposition, brauner2023type,BauHoh2008solomon}.

\subsection{The Mantaci--Reutenauer algebra}\label{s:MRalgebra}

In \cite{hsiao2007semigroup}, Hsiao established a connection between his algebra $\kk \hsiao$ and a combinatorial algebra called the \defn{Mantaci--Reutenauer algebra} $\MR_n[G]$, which was introduced by Mantaci and Reutenauer in \cite{MR95} and is a subalgebra of $\kk \sym_n[G]$.
Intuitively, $\MR_n[G]$ can be thought of as a ``colored'' analogue of Solomon's Descent algebra $\sol(\sym_n)$, consisting of elements of $\sol(\sym_n)$ ``colored'' by the elements of $G$.
Though originally defined only when $G$ is abelian, more recently, Baumann--Hohlweg \cite{BauHoh2008solomon} and Hsiao \cite{hsiao2007semigroup} gave generalizations for any finite group.

Recall that the wreath product $\sym_n[G]$ is a group with underlying set $\sym_n \times G^n$. Write a generic element in $\sym_n[G]$ as $(\sigma,g) = \big( (\sigma_1,g_1), (\sigma_2, g_2), \cdots, (\sigma_n, g_n) \big)$, so $\sigma = (\sigma_1, \cdots, \sigma_n) \in \sym_n$ is in one-line notation. Multiplication is then given by
\[ \big( (\sigma_1,g_1),  \cdots, (\sigma_n, g_n) \big) \cdot \big( (\tau_1,h_1), \cdots, (\tau_n, h_n) \big) = ((\sigma_{\tau_1},g_{\tau_1}h_1), \cdots, (\sigma_{\tau_n},g_{\tau_n}h_n)).\]

For $(\sigma, g) \in \sym_n[G]$, let $\co(\sigma,g) = \big((a_1,h_1), \cdots, (a_k,h_k)\big) \in \gcomposition$ be the coarsest $G$-composition of $n$ such that
\begin{align*}
    \sigma_1 < \cdots < \sigma_{a_1} & \hspace{2em} g_1 = g_2 = \cdots g_{a_1} = h_1, \\
    \sigma_{a_{1}+1}< \cdots < \sigma_{a_{1}+a_{2}} & \hspace{2em} g_{a_{1}+1} = \cdots = g_{a_{1}+a_{2}} = h_2, \\
    \vdots \\
    \sigma_{a_{1}+\cdots a_{k-1}+1} < \cdots < \sigma_n & \hspace{2em} g_{a_{1}+\cdots a_{k-1}+1} = \cdots = g_n =h_k.
\end{align*}
That $\co(\sigma,g)$ is the coarsest composition with this property means that: $h_1 \neq h_2$ or $\sigma_{a_1} > \sigma_{a_1+1}$; $h_2 \neq h_3$ or $\sigma_{a_2} > \sigma_{a_2+1}$; and so on.
Thus if $G$ is trivial, $\co(\sigma,g)$ encodes the descents of $\sigma$, which are the positions $i \in [n-1]$ such that $\sigma_i > \sigma_{i+1}$.  

\begin{example}
    Take $G = \ztwo$.
    We identify the wreath product $\sym_n[\ztwo]$ with group of signed permutations $\mathfrak{B}_n$ (or type B Coxeter group). We will write signed permutations $\sigma \in \mathfrak{B}_n$ in one-line notation as permutations of the set 
    \[ [n]^{\pm}:= \{ 1, \ol{1}, \cdots, n, \ol{n} \},\]
subject to the conditions that $\ol{\ol{i}} = i$ and if $\sigma(i) = j$, then $\sigma(\ol{i}) = \ol{j}$. For example, the permutation $\ol{1}\ol{3}2$ sends $1\mapsto \ol{1}$, $2 \mapsto \ol{3}$ and $3\mapsto 2$.
In this language, the longest element of $\mathfrak{B}_n$ is written as $w_0 = \ol{1}\ol{2}\cdots \ol{n}$.
  
Then for instance, if 
\[      w = 3 5 \ol{7} \ol{8} \ol{2} 4 1 6 \in \mathfrak{B}_8
    \]
    we have
    \[
        \co(w) = \big( 2^{+}, 2^{-} , 1^{-} , 1^{+} , 2^{+} \big) \in \Gamma_8[\ztwo].
    \]
\end{example}

Recall the order relation $\leq$ on elements of $\Gamma_n[G]$ from \cref{def:orderforMR}. Given a $G$-composition $p \in \Gamma_n[G]$, define elements $Y_p$ and $X_p$ of the group algebra $\kk \sym_n[G]$ by
\[
    Y_{p} := \sum_{\substack{(\sigma,g) \in \sym_n[G]\\ \co(\sigma,g) = p}} (\sigma,g) 
    \qqand
    X_{p} := \sum_{\substack{q \in \Gamma_n[G] \\ q \leq p}} Y_{q}.
\]
Möbius inversion shows that both $\{ Y_{p} \}_{p \in \Gamma_n[G]}$ and $\{ X_{p} \}_{p \in \Gamma_n[G]}$ linearly span the same subspace of $\kk \sym_n[G]$.
It turns out that this subspace is a subalgebra.

\begin{theorem}[Mantaci--Reutenauer, \cite{MR95}]\label{thm/def:mralg}
    Let  $\MR_n[G]$ be the $\kk$-subspace of linear combinations of the $Y_{p}$ (equiv. $X_{p}$) for $p \in \Gamma_n[G]$.
    Then $\MR_n[G]$ is a subalgebra of $\kk \sym_n[G]$.
\end{theorem}

From our perspective, a more natural characterization of $\MR_n[G]$ comes from work of Hsiao relating it to the LRBG $\hsiao$.
Recall from \cref{ex:invariantshsiao} that $\sym_n$ acts on $\kk \hsiao$ and that $\{ \SB{H}_p \}_{p \in \Gamma_n[G]}$ is a basis of the invariant subalgebra $(\kk \hsiao)^{\sym_n}$.

\begin{theorem}[Hsiao, \cite{hsiao2007semigroup}]\label{thm:hsiao} 
    Given any finite group $G$, the linear map
    \begin{align*}
        f : (\kk \hsiao)^{\sym_n}   & \longrightarrow   \MR_n[G] \\
                \SB{H}_{p}     &  \longmapsto     X_{p}
    \end{align*}
    is an anti-isomorphism of algebras.
\end{theorem}

\begin{remark} \rm \label{rmk:descentscompositions}
    
    When $G$ is the trivial group, $\MR_n[\{1\}] = \sol(\sym_n)$ and \cref{thm:hsiao} recovers the case $W = \sym_n$ of \cref{thm:bidigare}.
    Recall that the descent set of a permutation $\sigma = (\sigma_1, \cdots, \sigma_n) \in \sym_n$ is 
    \[
        \Des(\sigma) := \{ i \in [n-1]: \sigma_i > \sigma_{i+1} \}.
    \]
    The definition of the elements $\{ Y_{p} \}_{p \in \Gamma_n}$ and $\{ X_{p} \}_{p \in \Gamma_n}$ can be rephrased
    using the usual bijection between subsets of $[n-1]$ (written as $2^{[n-1]}$) and integer compositions of $n$:
    \begin{align*}
        \varphi: \Gamma_n &\longrightarrow 2^{[n-1]}\\
        (p_1, \ldots, p_k) & \longmapsto \{ p_1 \,,\, p_1 + p_2 \,,\, \ldots \,,\, p_1 + \cdots + p_{k-1} \}.
    \end{align*}
    For example, $\varphi((n)) = \emptyset$ and $\varphi((1,\dots,1)) = [n-1]$.
    Note that $\varphi$ is order-preserving, in the sense that $p \leq q$ if and only if $\varphi(p) \subseteq \varphi(q)$.
    Thus we can write $X_{p}$ and $Y_{p}$ as
    \[
         Y_{p} = \sum_{\substack{w \in \sym_n \\ \Des(w) = \varphi(p) }} w
         \qqand
         X_{p} = \sum_{\substack{w \in \sym_n \\ \Des(w) \subseteq \varphi(p)}} w.
    \]
    
    It follows that the Mantaci--Reutenauer algebra $\MR_n[G]$ always contains a copy of Solomon's Descent algebra: $\sol(\sym_n) = \MR_n[\{1_G\}] \subseteq \MR_n[G]$,
    where for $p \in \Gamma_n[G]$, the elements $X_{p}, Y_{p}$ are in $\sol(\sym_n)$ if and only if the labels of $p$ are all $1_G$.
\end{remark}

Importantly, \cref{thm:hsiao} implies that the image of any complete system of primitive orthogonal idempotents of $(\kk \hsiao)^{\sym_n}$ under the map $f$ is also a complete system of primitive orthogonal idempotents of the Mantaci--Reutenauer algebra $\MR_n[G]$.
When $G$ is abelian, we apply the $f$ map to the idempotents in \cref{cor:uniform-CSoPOI-Hsiao} to obtain the following closed form expressions for idempotents in $\MR_n[G]$.

\begin{theorem}\label{cor:CSoPOIforMR}
    Let $G$ be a finite abelian group.
    The collection $\{ \mathfrak{f}_{\lambda} \}_{\lambda \in \Lambda_n[\wh{G}]} \subseteq \MR_n[G]$ determined by
    \[
        \mathfrak{f}_{\lambda} = \frac{1}{|G|^{\ell(\lambda)} \ell(\lambda)!} \sum_{\substack{ \alpha \in \Gamma_n[\wh{G}] \\ \supp(\alpha) = \lambda}} \sum_{\substack{p \in \Gamma_n[G] \\ \abs{p} = \abs{\alpha} }} \ol{\alpha(p)} \sum_{\substack{q \in \Gamma_n[G] \\ p \leq q }} \dfrac{(-1)^{\ell(q)-\ell(p)}}{\deg(q/p)} X_q,
    \]
    is a complete system of primitive orthogonal idempotents of the Mantaci--Reutenauer algebra.
\end{theorem}

\begin{example}
    Consider the case $n=2$ and $G =\ztwo$. Then, the CSoPOI of $\MR_2[\ztwo]$ consists of five elements
    \[
        \mathfrak{f}_{\{2^{\trv}\}}\qquad\mathfrak{f}_{\{2^{\sgn}\}}\qquad\mathfrak{f}_{\{1^{\trv},1^{\trv}\}}\qquad\mathfrak{f}_{\{1^{\trv},1^{\sgn}\}}\qquad\mathfrak{f}_{\{1^{\sgn},1^{\sgn}\}}.
    \]
    We expand the first of these elements in the $X$-basis of $\MR_2[\ztwo]$.
    For $\lambda = \{2^{\trv}\}$, the only $\alpha \in \Gamma_2[\wh{\ztwo}]$ of support $\lambda$ is $\alpha = (2^{\trv})$, and the $p \in \Gamma_2[\ztwo]$ having the same underlying composition are $(2^+)$ and $(2^-)$. Since $\alpha$ is labeled by the trivial character, in both cases we have $\ol{\alpha(p)} = 1$, thus
    \[
        \mathfrak{f}_{\{2^{\trv}\}} = 
        \dfrac{1}{2} \Big( X_{(2^+)} + \dfrac{-1}{2} X_{(1^+,1^+)} + X_{(2^-)} + \dfrac{-1}{2} X_{(1^-,1^-)} \Big) =
        \dfrac{1}{4} \Big( 2 X_{(2^+)} + 2 X_{(2^-)} - X_{(1^+,1^+)} - X_{(1^-,1^-)} \Big).
    \]
\end{example}

We are not aware of other work constructing idempotents for $\MR_n[G]$ for
arbitrary finite abelian groups with the unique exception being that of Vazirani
\cite{vazirani} constructing idempotents for $\MR_n[C_2]$
(also known as the \defn{Type $B$ Mantaci--Reutenauer algebra}, to be discussed
in \cref{sec:definingvazidem}).
We will show in the subsequent section that her idempotents are precisely the ones from \cref{cor:CSoPOIforMR}.

\subsection{Specializing to Type B }\label{sec:definingvazidem} \label{s:typeBMR}
In this section, we will turn to the Type $B$ Mantaci--Reutenauer algebra, $\MR_n[\ztwo]$, which has been studied in a variety of contexts; see for instance \cite{abn04peak, brauner2023type, DouTom2018decomposition, pang2021eigenvalues, vazirani}. 

Note that in the case that $G = \ztwo$, there is a canonical isomorphism
between $G$ and $\wh{G}$ identifying $+1$ with $\trv$ and $-1$ with $\sgn$.
Thus in this case, all indexing sets for $\Sigma_n[\ztwo]$ coincide. To
simplify notation and to be compatible with the notation in \cite{vazirani},
we will identify the elements of $\ztwosetcomposition, \ztwosetcompositiondual$ and
$\ztwosetpartition, \ztwosetpartitiondual$ with set compositions and set partitions of
\[ [n]^{\pm}= \{ 1, \ol{1}, \cdots, n, \ol{n} \},\]
by writing the blocks $(S, g)$ in a labelled composition as $\ol{S}$ if $g
\in \{-1, \sgn\}$ and as $S$ if $g \in \{1, \trv\}$. For example
\[
    \big((\{1,2,3 \}, -1), (\{4,5 \}, 1)\big)
    \longleftrightarrow
    \big({\{\ol{1},\ol{2},\ol{3} \}}, \{ 4, 5 \}\big)
    \longleftrightarrow
    \big(\ol{\{1,2,3 \}}, \{ 4, 5 \}\big).
\]
If a block $S \subseteq [\ol{n}]:= \{ \ol{1}, \cdots, \ol{n} \}$, we say it is negative; if $S \subseteq [n]$, we say it is positive. 
We will write the elements of the sets $\ztwocomposition, \ztwocompositiondual$ and $\ztwopartition, \ztwopartitiondual$ in an analogous way. For instance, the element $((\{1,2,3 \}, -1), (\{4,5 \}, 1))$ above has type $(\ol{3},2)$. Recall that $\ol{\ol{i}} = i$. 

\begin{remark}
Note that $\MR_n[\ztwo]$ is a subalgebra of the group algebra of $\sym_n[\ztwo] \cong \mathfrak{B}_n$, the hyperoctahedral group. As an element of $\C \mathfrak{B}_n$, we have that $X_{(n)} = 12\cdots n$, the identity permutation, while $X_{(\ol{n})} = \ol{12\cdots n}$, the longest permutation.
\end{remark}

Our first goal is to define Vazirani's construction of idempotents for $\MR_n[\ztwo]$. These can be defined using an algebra structure on $\MR[\ztwo] := \bigoplus_n \MR_n[\ztwo]$, which is the restriction of a well-studied Hopf algebra structure on signed permutations. To describe this algebra structure, we will follow the notation of \cite{abn04peak}.

\begin{definition}\label{def:staronmr}
  
Let $\kk \mathfrak{B} = \bigoplus_n \kk \mathfrak{B}_n$. For $u \in \mathfrak{B}_n$ and $v \in \mathfrak{B}_m$, define
\[
    u \star v = \sum_{w \in {\sf Sh}(n,m)} w \cdot (u \times v)
\]
where $u \times v$ is seen as a permutation in $\mathfrak{B}_{n+m}$ via the usual inclusion $\mathfrak{B}_n \times \mathfrak{B}_m \hookrightarrow \mathfrak{B}_{n+m}$, and
\[
    {\sf Sh}(n,m) = \{ w \in \sym_{n+m} \,:\, w_1 < \dots < w_n \,,\, w_{n+1} < \dots < w_{n+m} \}.
\]  
\end{definition}
The subspace $\MR[\ztwo] := \bigoplus_n \MR_n[\ztwo]$ is closed under the $\star$ product, see for instance \cite[\S 8.2]{abn04peak}.
Moreover,
\begin{equation}\label{eq:starMR-X-basis}
    X_p \star X_q = X_{pq}.
\end{equation}
Compare this with \cref{prop:invar-star-subalg}.
We deduce the following.

\begin{cor}\label{c:isoexternalprod}
    The direct sum of the (anti)-isomorphisms $\kk \Sigma_n[\ztwo] \to \MR_n[\ztwo] : \SB{H}_p \to X_p$ in \cref{thm:hsiao} is an isomorphism of algebras $\kk \Sigma[\ztwo] \to \MR[\ztwo]$ under the $\star$ product.
\end{cor}

There will be three key objects of interest for us. 
First, is the 
 \defn{Reutenauer idempotent}, introduced by Reutenauer in \cite[\S 8.4]{reutenauer2001free}, and which appears in both the study of $\sol(\sym_n)$ and $\MR[\ztwo]$:
\[ r_{n} := \sum_{p \in \Gamma_n}  \frac{(-1)^{\ell(p)-1}}{\ell(p)} X_{p}.\]
Note that because $p \in \Gamma_n$, it follows that $r_{n}\in \sol(\sym_n)$. Analogously, define 
\[ \ol{r_n}:=  \sum_{p \in \Gamma_n}  \frac{(-1)^{\ell(p)-1}}{\ell(p)} X_{\ol{p}}.\]

We will take 
\[ r_{p}:= r_{p_1} \star r_{p_2} \dots \star r_{p_k}\]
for $p = (p_1, p_2, \cdots, p_k) \in \Gamma_n[\ztwo]$. 

The second important object will be the $\I$-basis of $\MR[\ztwo]$, defined as follows. Let 
\[ \I_n:= \I_{(n)}= \frac{1}{2} (r_n + \ol{r_n}), \qquad \text{  } \qquad \I_{\ol{n}}:= \I_{(\ol{n})}= \frac{1}{2}(r_n - \ol{r_n}),\]
and for $\alpha = (\alpha_1, \cdots, \alpha_k) \in  \ztwocompositiondual$ define
\[ \I_\alpha := \I_{\alpha_1} \star \I_{\alpha_2} \star \dots \star \I_{\alpha_k}\]
with $\star$ as in \cref{def:staronmr}.

Finally, we may define our third object of interest:  Vazirani's idempotents.
\begin{definition} \label{def:vazidem}
For $\lambda \in \ztwopartitiondual$ the \defn{Vazirani idempotents} are given by 
\begin{align*}
     \frakg_{\lambda} &:= \frac{1}{\ell( \lambda) !} \sum_{\substack{\alpha \in \ztwocompositiondual\\ \supp(\alpha) = \lambda}} \I_\alpha.
    \end{align*}
\end{definition}
\begin{remark}
\cref{def:vazidem} is a slight reformulation of the original definition given by Vazirani in \cite{vazirani}. In particular, her paper does not explicitly use the Hopf algebra structure, but is equivalent. Douglass and Tomlin give an alternate definition in \cite{DouTom2018decomposition} that does not use the Hopf structure and instead realizes the $\frakg_{\lambda}$ directly in $\Q \mathfrak{B}_n $.
\end{remark}
\begin{theorem}[Vazirani]
    The collection $\{\frakg_{\lambda}\}_{\lambda \in \ztwopartitiondual}$ is a complete system of primitive orthogonal idempotents for $\MR_n[\ztwo]$.
\end{theorem}
\begin{example} 
When $n=2$, the idempotents $\frakg_{\lambda}$ in $\MR_2[\ztwo]$ are
\begin{align*}
    \frakg_{(2)} & 
        = \frac{1}{4} \left( -Y_{(1,1)}  - Y_{(\ol{1},\ol{1})} + Y_{(2)} + Y_{(\ol{2})}\right) 
        = \frac{1}{4} \left( -X_{(1,1)}  - X_{(\ol{1},\ol{1})} + 2X_{(2)} + 2X_{(\ol{2})}\right)\\
    \frakg_{(1,1)} &
        = \frac{1}{8} \left( Y_{(1,1)} + Y_{(1, \ol{1})} + Y_{(\ol{1}, 1)} + Y_{(\ol{1},\ol{1})} + Y_{(2)} + Y_{(\ol{2})} \right) 
        =\frac{1}{8} \left( X_{(1,1)} + X_{(1, \ol{1})} + X_{(\ol{1}, 1)} + X_{(\ol{1},\ol{1})}\right)  \\
    \frakg_{(1, \ol{1})} &
        = \frac{1}{4} \left( Y_{(1,1)} - Y_{(\ol{1},\ol{1})} + Y_{(2)} - Y_{(\ol{2})} \right) 
        = \frac{1}{4} \left( X_{(1,1)} - X_{(\ol{1},\ol{1})}\right)  \\
    \frakg_{(\ol{1},\ol{1})} &
        = \frac{1}{8} \left( Y_{(1,1)} - Y_{(1, \ol{1})} - Y_{(\ol{1}, 1)} + Y_{(\ol{1},\ol{1})} + Y_{(2)} + Y_{(\ol{2})} \right) 
        = \frac{1}{8} \left( X_{(1,1)} - X_{(1, \ol{1})} - X_{(\ol{1}, 1)} + X_{(\ol{1},\ol{1})}\right) \\
    \frakg_{(\ol{2})} &
        =  \frac{1}{4} \left( -Y_{(1,1)}  + Y_{(\ol{1},\ol{1})} + Y_{(2)} - Y_{(\ol{2})} \right) 
        = \frac{1}{4} \left( -X_{(1,1)}  + X_{(\ol{1},\ol{1})} + 2X_{(2)}  -2X_{(\ol{2})}\right)
\end{align*}
\end{example}
Our goal going forward is to identify via Hsiao's map $f$ the following objects:
\begin{itemize}
    \item the Reutenauer idempotents with the $\SB{R}$-basis,
    \item the $\SB{I}$-basis with the $\SB{Q}$-basis, and  
    \item the $\frakg_\lambda$ with $\ee_{\lambda}$.
\end{itemize}

It turns out that the results developped in \cref{s:change-of-basis-hsiao} will make short work of this. 

\begin{theorem}
Let $\SB{Q}_{\alpha}$ and $\ee_{\lambda}$ be determined by the uniform section. Then we have the following evaluations under Hsiao's map: 
\begin{align*}
f: (\kk\Sigma_n[\ztwo])^{\sym_n} &\longrightarrow \MR_n[\ztwo] \\
  \SB{R}_p &\longmapsto  r_p  \\ 
   \SB{Q}_{\alpha} &\longmapsto \SB{I}_{\alpha}\\
   \ee_{\lambda} &\longmapsto \frakg_{\lambda}.
\end{align*}
for all $p \in \Gamma_n[\ztwo] $, $\alpha \in \Gamma_n[\wh{\ztwo}] \in $ and $\lambda \in \Lambda_n[\wh{\ztwo}] $.
\end{theorem}
\begin{proof}
Observe that for $p \geq (n)$, $\deg(p/(n)) = \ell(p)$.
Then, by \cref{p:SH-to-SR} we have
\[
    \SB{R}_{(n)} = \sum_{p \geq (n)} \dfrac{(-1)^{\ell(p)-1}}{\ell(p)} \SB{H}_p
    \qqand
    \SB{R}_{(\ol{n})} = \sum_{p \geq (\ol{n})} \dfrac{(-1)^{\ell(p)-1}}{\ell(p)} \SB{H}_p.
\]

Hence,
\[
    f(\SB{R}_{(n)}) = r_n
    \qqand
    f(\SB{R}_{(\ol{n})}) = r_{\ol{n}}
\]
and by the compatibility of the $\star$ product from \cref{c:isoexternalprod}, it follows that $f(\SB{R}_p) = r_p$ for any $p \in \Gamma_n[\ztwo]$.

Moreover, from \cref{p:SR-to-SQ}, 
\[
    \SB{Q}_{(n)} =
    \dfrac{1}{2} \Big( \ol{{\trv}(+1)} \SB{R}_{(n)} + \ol{{\trv}(-1)} \SB{R}_{(\ol{n})} \Big) =
    \dfrac{1}{2} \Big( \SB{R}_{(n)} + \SB{R}_{(\ol{n})} \Big)
\]

and similarly,
\[
    \SB{Q}_{(\ol{n})} =
    \dfrac{1}{2} \Big( \ol{{\sgn}(+1)} \SB{R}_{(n)} + \ol{{\sgn}(-1)} \SB{R}_{(\ol{n})} \Big) =
    \dfrac{1}{2} \Big( \SB{R}_{(n)} - \SB{R}_{(\ol{n})} \Big).
\]

Therefore,
\[  
    f(\SB{Q}_{(n)}) = \SB{I}_{(n)}
    \qqand
    f(\SB{Q}_{(\ol{n})}) = \SB{I}_{(\ol{n})}.
\]

It then follows by \cref{prop:SymQ-basis-star} and the definition of $\SB{I}_\alpha$ that for all $\alpha \in \Gamma[\wh{\ztwo}]$,
\[
    f(\SB{Q}_\alpha) = \SB{I}_\alpha.
\]
Finally, since
\[
    \ee_\lambda = \dfrac{1}{\ell(\lambda)!} \sum_{\supp(\alpha) = \lambda} \SB{Q}_\alpha,
\]
see \cref{ex:uniform-general}, we have that for all $\lambda \in \Lambda[\wh{\ztwo}]$
\[
    f(\ee_\lambda) = \mathfrak{g}_\lambda. \qedhere
\]
\end{proof}

Using \cref{cor:uniform-CSoPOI-Hsiao}, we obtain the following formula for Vazirani's idempotents.

\begin{cor}\label{cor:vazclosedformexpression}
For any $\lambda \in \Lambda_n[\wh{\ztwo}]$, Vazirani's idempotents $\frakg_{\lambda}$ have the expression 
  \[  \frakg_{\lambda} = \dfrac{1}{2^{\ell(\lambda)} \cdot \ell(\lambda)!} \sum_{\substack{ \alpha \in \Gamma_n[\wh{\ztwo}] \\ \supp(\alpha) = \lambda}} \sum_{\substack{p \in \Gamma_n[\ztwo] \\ \abs{p} = \abs{\alpha} }} \ol{\alpha(p)} \sum_{\substack{q \in \Gamma_n[\ztwo] \\ p \leq q }} \dfrac{(-1)^{\ell(q)-\ell(p)}}{\deg(q/p)} X_q. \]
\end{cor}

\bibliographystyle{plain}
\bibliography{bib}

\newpage

\appendix

\section{LRBGs and BGs}\label{s:appendix}

In the semigroup theory literature, one finds the notion of a \defn{band of
groups} \cite{clifford, cp61semigroupsI},
but it should be noted that \defn{left regular bands of groups},
as defined in \cite{ms11HomologySemigroups} and in this paper,
are not bands of groups as one might expect from the name.
This appendix serves to outline the differences and to characterize the left
regular bands of groups that are bands of groups
(which we call \defn{strict left regular bands of groups}).

\subsection{Bands of semigroups}
\label{bands-of-semigroups}

Let $\SS$ be a finite semigroup.
Let $\BB = \{\SS_\alpha\}_{\alpha \in I}$
be a family of disjoint subsemigroups of $\SS$
such that for every pair $\SS_\alpha, \SS_\beta \in \BB$,
there exists $\SS_\gamma \in \BB$
such that $\SS_\alpha \SS_\beta \subseteq \SS_{\gamma}$.
Then $\BB$ is a semigroup for the operation:
\begin{equation*}
    \SS_\alpha \ast \SS_\beta = \SS_\gamma
    \qqiff
    \SS_\alpha \SS_\beta \subseteq \SS_\gamma.
\end{equation*}
Since $\SS_\alpha \SS_\alpha \subseteq \SS_\alpha$, it follows that $\BB$
is a \defn{band} (i.e., a semigroup all of whose elements are idempotents),
and we call $\BB$ a \defn{band of semigroups}.
If in addition $\BB$ is commutative, then we call $\BB$
a \defn{semilattice of semigroups}, since a \defn{semilattice} is a semigroup
that is a commutative band.

A semigroup $\SS$ is a \defn{union of a band of semigroups} if there exists
a band of semigroups $\BB = \{\SS_\alpha\}_{\alpha \in I}$ in $\SS$ such that
\begin{equation*}
    \SS = \bigsqcup_{\SS_\alpha \in I} \SS_\alpha.
\end{equation*}
Similarly, one defines \defn{semilattices of semigroups}.
By an abuse of terminology, one sometimes encounters \emph{``$\SS$ is a band of
semigroups''} in place of \emph{``$\SS$ is a union of a band of semigroups''}.

If $\SS$ is the union of the band of semigroups $\BB
= \{\SS_\alpha\}_{\alpha \in I}$, then the function
$\varphi: \SS \xrightarrow{} \BB$ defined by
$\varphi(s) = \SS_\alpha$ for $s \in \SS_\alpha$ is a semigroup
morphism as $s \in \SS_\alpha$ and $t \in \SS_\beta$ imply that
$s t \in \SS_\alpha \SS_\beta \subseteq \SS_\alpha \ast \SS_\beta$.
Conversely, every union of a band of semigroups arises in this way: if
$\varphi: \SS \xrightarrow{} \BB$ is a surjective semigroup morphism
onto a band $\BB$, then $\SS$ is the union of the band of semigroups
$\{\varphi^{-1}(b)\}_{b \in \BB}$.

\subsection{Bands of groups}
\label{bands-of-groups}

We are mostly interested in the case where each of the semigroups in a band of
semigroups is a group, so we state this definition explicitly.

\begin{definition}
    A semigroup $\SS$ is a \defn{union of a band of groups} if there is a decomposition
    \begin{equation*}
        \SS = \bigsqcup_{\alpha \in I} G_\alpha
    \end{equation*}
    into disjoint subsemigroups of $\SS$ such that
    \begin{enumerate*}[label=(\roman*)]
        \item each $G_\alpha$ is a group, and
        \item for each pair $\alpha, \beta \in I$, there exists
              a unique $\gamma \in I$ such that $G_\alpha G_\beta \subseteq G_\gamma$.
    \end{enumerate*}
\end{definition}

\begin{theorem}[{\cite[Theorem 1]{clifford}}]
    \label{theorem-clifford-band-of-groups}
    A semigroup $\SS$ is a union of a band of groups if and only if there exists a band $\BB$ and a
    semigroup morphism $\SS \to \BB$ such that the preimage of each $b \in \BB$ is a group.
\end{theorem}

\begin{example}
    Every band $\BB$ is the union of a band of groups: consider the identity
    morphism $\BB \xrightarrow{} \BB$. In this case, the preimages are trivial
    groups.
\end{example}

\begin{example}
    Fix a group $G$ and a band $\BB$. Endow $\SS := \BB \times G$ with the
    coordinate-wise product. Then, $\SS$ is a the union of the band of groups
    $\{\SS_b\}_{b \in \BB}$, where $\SS_b = \{b\} \times G$.
    Note that $\SS_b \cong G$.
\end{example}

\begin{remark}
    Note that if $\SS$ is a union of a band of groups
    $\BB = \{G_\alpha\}_{\alpha \in I}$,
    then $E(\SS)$ is in bijection with $\BB$
    since each of the subgroups $G_\beta$ contains exactly one idempotent of
    $\SS$. However, $E(\SS)$ is not necessarily isomorphic to $\BB$;
    in fact, $E(\SS)$ is not necessarily a subsemigroup of $\SS$.
\end{remark}

\subsection{Left Regular Bands of Groups}
\label{appendix:LRBGs}

Recall that for each element $s$ in a finite semigroup $\SS$, there exists a unique idempotent that is a positive power of $s$ and it is denoted by $s^\omega$.
Moreover, we can pick $\omega \in \Z_{> 0}$ uniformly for all $s \in \SS$, for instance by setting $\omega = |\SS|!$ (see \cite[Remark 1.3]{steinberg16monoids}). To ease exposition, we will henceforth assume that $\omega = |\SS|!$, so that it does not depend on a particular $s \in \SS$. This assumption does not change any of the claims made.
Following \cite{ms11HomologySemigroups}, we define a
\defn{left regular band of groups (LRBG)} to be a semigroup $\SS$
such that for all $s,t \in \SS$:
\begin{align}
    \tag{LRBG1}     s^\omega s &= s,     \\
    \tag{LRBG2}     s t s^\omega &= s t.
\end{align}
If every element of a LRBG is idempotent, then we obtain a semigroup called
a \defn{left regular band}: more precisely,
a semigroup $\BB$ is a \defn{left regular band (LRB)} if for all $x, y \in
\BB$:
\begin{align}
    \tag{LRB1}  x^2 &= x, \\
    \tag{LRB2}  x y x &= x y.
\end{align}
LRBs form an important family of examples of LRBGs, and we will see that every
LRBG is obtained by attaching a group to each element of a LRB (subject to some
compatibility conditions).

\begin{lemma}
    Let $\SS$ be a LRBG and let $E(\SS)$ denote the set of idempotents of $\SS$.

    Then $E(\SS)$ is a subsemigroup of $\SS$. Moreover, $E(\SS)$ is a LRB.
\end{lemma}

\begin{proof}
    We first prove $E(\SS)$ is a subsemigroup.
    Let $x, y \in E(\SS)$.
    Then $x^\omega = x$ and $y^\omega = y$,
    and so
    \begin{equation*}
        (x y)^2
        = x y x y
        = x y x^\omega y^\omega
        \overset{\cref{eq:lrbg2}}{=}
        x y y^\omega
        \overset{\cref{eq:lrbg1}}{=}
        x y.
    \end{equation*}
    It follows that $x y \in E(\SS)$.
    To see that $E(\SS)$ is a LRB, note that (LRB1) holds because every element
    of $E(\SS)$ is idempotent, and (LRB2) follows by setting $s = x$ and $t
    = y$ in \cref{eq:lrbg2}.
\end{proof}

The next result identifies the maximal subgroups of $\SS$ at an idempotent $x
\in E(\SS)$ with the set of elements in $\SS$ that are mapped to $x$ by the
function $s \mapsto s^\omega$.

\begin{lemma}
    \label{appendix:preimages-of-omega-map-are-maximal-subgroups}
    Let $\SS$ be a LRBG and define $\varphi: \SS \xrightarrow{} E(\SS)$ by
    $\varphi(s) = s^\omega$.
    Then for each $x \in E(\SS)$, the preimage
    $\varphi^{-1}(\{x\}) = \{ s \in \SS : s^\omega = x \}$ is the maximal
    subgroup $G_x$ of $\SS$ at $x$.
\end{lemma}

\begin{proof}
    We first prove $G_x \subseteq \{s \in \SS : s^\omega = x\}$.
    If $s$ is an element of the finite group $G_x$,
    then some positive power of $s$ coincides with the
    identity element of $G_x$, which is $x$.
    Hence, $s^\omega = x$.

    Conversely, suppose $s^\omega = x$.
    By \cref{eq:lrbg1}, $x s = s^\omega s = s$ and
    $s x = s s^\omega = s$, which implies $s$
    belongs to the submonoid $x \SS x$ with identity element $x$.
    Moreover, $s$ is invertible in $x \SS x$ since if $n$ is a positive integer
    such that $s^n = s^\omega = x$, then the inverse of $s$ is $s^{n-1}$.
    Hence, $s \in G_x$.
\end{proof}

It follows from \cref{appendix:preimages-of-omega-map-are-maximal-subgroups} that
$\SS$ decomposes into a disjoint union of subgroups as follows:
\begin{equation}
    \label{LRBG-decomposition-union-of-groups}
    \SS = \bigsqcup_{x \in E(\SS)} G_x,
    \qquad\text{where~} G_x = \{s \in \SS: s^\omega = x\}.
\end{equation}
Therefore, every LRBG is a \emph{union of groups}. However, it is not
necessarily a \emph{union of a band of groups} because $G_x G_y$ might not be
contained in a single $G_z$, as illustrated by the following example.

\begin{example}
    \label{ex:not-strict}
    Consider the semigroup $\SS = \{x, s, y, z\}$ with the following multiplication table:
    \begin{equation*}
        \begin{array}{c|cccc}
            \SS & x & s & y & z \\ \hline
            x & x & s & y & z \\
            s & s & x & z & y \\
            y & y & y & y & y \\
            z & z & z & z & z
        \end{array}
    \end{equation*}
    Observe that:
    \begin{enumerate*}[label=(\roman*)]
        \item
            $\SS$ is a monoid with identity element $x$;

        \item
            $E(\SS) = \{x, y, z\}$ is a LRB;

        \item
            the maximal subgroups of $\SS$ are
            $G_{x} = \{x, s\}$, $G_{y} = \{y\}$, and $G_{z} = \{z\}$; and

        \item
            left multiplication by $s$ swaps $y$ and $z$.
    \end{enumerate*}
    It is straightforward to check that $\SS$ is a LRBG, but it is not a union
    of a band of groups because the $G_{x} G_{y} = \{y, z\}$ is not contained in
    any single maximal subgroup.
\end{example}

\begin{lemma}
    Let $\SS$ be a LRBG. Then for all $s, t \in \SS$,
    \begin{equation}
        \label{(s^omega-t)^omega}
        (s^\omega t)^\omega = s^\omega t^\omega.
    \end{equation}
\end{lemma}

\begin{proof}
    Start with $s^\omega t s^\omega t \cdots s^\omega t$
    and iteratively apply (LRBG2) to delete all occurrences of $s^\omega$
    except the first: thus, $(s^\omega t)^m = s^\omega t^m$ for all positive $m
    \in \NN$.
    Taking $m$ such that $t^m = t^\omega$, one obtains
    $(s^\omega t)^m = s^\omega t^\omega \in E(\SS)$
    since $E(\SS)$ is subsemigroup of $\SS$.
    Thus, $(s^\omega t)^m$ is idempotent.
\end{proof}

\subsection{Strict Left Regular Bands of Groups}

Let $\SS$ be a LRBG and consider the function $\varphi: \SS \to E(\SS)$ defined
by $\varphi(s) = s^\omega$.
This function maps $\SS$ onto a band and the preimage of every element is a group.
By \cref{theorem-clifford-band-of-groups}, $\SS$ is the union of a band
of groups if $\varphi$ is a semigroup morphism.
We define a strict LRBG to be a LRBG for which $s \mapsto s^\omega$ is a semigroup morphism.

\begin{definition}
    A \defn{strict left regular band of groups (sLRBG)} is a LRBG $\SS$ such that
    \begin{align}
        \label{eq:strict}
        \tag{sLRBG}     (st)^\omega = s^\omega t^\omega
        \text{~for all $s,t \in \SS$.}
    \end{align}
\end{definition}

Observe that the LRBG of \cref{ex:not-strict} is not a sLRBG since
$(s y)^\omega = z \neq y = s^\omega y^\omega$.

\begin{prop}
    \label{strict-iff-BG}
    A LRBG is a union of a band of groups if and only if it is a strict LRBG.
\end{prop}

\begin{proof}
    Let $\SS$ be a sLRBG. By \cref{E(S)-is-a-LRB} and \cref{eq:strict}, the map
    $\varphi: \SS \to E(\SS)$ is a semigroup morphism from $\SS$ onto the band
    $E(\SS)$. By \cref{appendix:preimages-of-omega-map-are-maximal-subgroups},
    the preimage $\varphi^{-1}(x)$ of each $x \in E(\SS)$ is
    a group. Therefore, $\SS$ is the union of the band of groups
    $\{\varphi^{-1}(x) : x \in E(\SS)\}$.

    Conversely, suppose $\SS$ is a LRBG that is a union of a band of groups.
    Then there exists a band $\BB$ and a surjective semigroup morphism $\psi:
    \SS \xrightarrow{} \BB$ such that the preimage of each $b \in \BB$ is
    a group.
    Since $E(\SS)$ is a subsemigroup of $\SS$, the restriction
    $\psi|_{E(\SS)}: E(\SS) \to \BB$ is a semigroup morphism.
    Moreover, it is a bijection because $\psi^{-1}(b)$ for $b \in \BB$ is
    a group and thus contains a unique idempotent (namely, the identity element
    of the group).
    Therefore, the composition
    $\psi|_{E(\SS)}^{-1} \circ \psi$
    is a semigroup morphism that maps each $s \in \SS$ to the identity element
    of the group containing $s$; i.e., $s \mapsto s^\omega$.
\end{proof}

The following is an important consequence of axiom \cref{eq:strict}.

\begin{lemma}\label{lem:xf=fx}
    Let $\SS$ be a strict LRBG.
    Let $s \in \SS$ and $y \in E(\SS)$ with $s^\omega \leq y$. Then,
    \[        sy = ys.      \]
\end{lemma}

\begin{proof}
    The claim follows from the following computation, where we use that $x^\omega f = f$:
    \begin{equation*}
        s y                         \overset{\cref{eq:lrbg1}}{=}
        (s y)^\omega s y            \overset{\cref{eq:strict}}{=}
        s^\omega y s y              \overset{(s^\omega \leq y)}{=}
        y s y                       \overset{\cref{eq:lrbg2}}{=}
        y s.                        \qedhere 
    \end{equation*}
\end{proof}

In particular, we obtain a new characterization of the partial order $\leq$ on a strict LRBG.

\begin{lemma}\label{lem:sLRBG-leq}
    Let $\SS$ be a strict LRBG, $s,t \in \SS$ and $x = s^\omega \in E(\SS)$.
    Then, $s \leq t$ if and only if $t = ys = sy$ for some $y \in E(\SS)$ with $x \leq y$.
\end{lemma}

\begin{proof}
    If $s \leq t$, take $y = t^\omega$. Then, $t = s y$ by definition of $\leq$, and
    \[
        x y                 \overset{\cref{eq:strict}}{=}
        (s y)^\omega        \overset{(s \leq t)}{=}
        (t)^\omega = y,
    \]
    so $x \leq y$. Finally, by \cref{lem:xf=fx}, $t = ys = sy$.

    Conversely, if $t = ys = sy$ for some $y \in E(\SS)$ with $x \leq y$, then $t^\omega = (y s) ^\omega = y x = y$, so $s t^\omega = s t = t$ and $s \leq t$.
\end{proof}

\subsection{Presheaf construction}

In this section we present a construction to obtain a strict LRBG from a {\em presheaf of groups} over a LRB.
As it turns out, all strict LRBG arise in this manner.
When the underlying LRB is a semilattice, this construction was first
considered by Clifford; see for instance \cite[Theorem 4.11]{cp61semigroupsI}.

Let $\BB$ be a left regular band. We endow $\BB$ with the preorder $\preceq$
defined for $x, y \in \BB$ by
\begin{equation*}
    x \preceq y
    \qqiff
    y x = y.
\end{equation*}
This relation is reflexive and transitive, but not necessarily antisymmetric.
In particular, it is possible that $x \neq x'$ even though $x \preceq x'$ and $x' \preceq x$.
For example, when $\BB = \Sigma_{\A}$ is the semigroup of faces of a hyperplane
arrangement $\A$, we have ${\sf E} \preceq {\sf F}$ precisely when $\supp({\sf
E}) \subseteq \supp({\sf F})$, and any two faces ${\sf F}$ and ${\sf F'}$
with $\supp({\sf F}) = \supp({\sf F'})$ satisfy both
${\sf F} \preceq {\sf F'}$ and ${\sf F'} \preceq {\sf F}$.

Let $\mathfrak{B}$ be the thin category corresponding to this relation.
That is, $\mathfrak{B}$ is the category with objects the elements of $\BB$, and
there is exactly one arrow $x \to y$ whenever $x \preceq y$.
A \defn{presheaf of finite groups} on $\BB$ is a functor $\GG : \mathfrak{B} \to
\mathbf{gps}$, where $\mathbf{gps}$ is the category of finite groups.
More concretely, $\GG$ consists of the data of
\begin{enumerate}
    \item a finite group $\GG[x]$ for each $x \in \BB$, and
    \item a group morphism $\Delta_x^y : \GG[x] \to \GG[y]$ for each pair $x \preceq y$ in $\BB$.
    These morphisms satisfy:
    \begin{enumerate}[label=(\roman*)]
        \item
            for all $x \in \BB$, the map $\Delta_x^x$ is the identity morphism of $\GG[x]$; and
        \item
            whenever $x \preceq y$ and $y \preceq z$, one has $\Delta_y^z \circ \Delta_x^y = \Delta_x^z$.
    \end{enumerate}
\end{enumerate}
Observe that if $x \preceq x'$ and $x' \preceq x$,
then $\Delta_x^{x'}$ is an isomorphism with inverse $\Delta_{x'}^x$.

%%% New remark
\begin{remark}
    When $\BB = \Sigma_{\A}$ is the semigroup of faces of a hyperplane
    arrangement $\A$, the data of a presheaf of finite groups on $\BB$
    is equivalent to that of a cocommutative $\A$-comonoid with values
    in the category of finite groups, in the sense  of Aguiar and Mahajan's \emph{Bimonoids for hyperplane arrangements} \cite{am20}.
\end{remark}

Given a presheaf of finite groups $\GG$ on $\BB$, we construct a semigroup
$\SS$ as follows.
As a set, $\SS$ is the disjoint union of all the groups in the image of $\GG$:
\begin{equation}
    \label{presheaf-construction}
    \SS = \bigsqcup_{x \in \BB} \GG[x].
\end{equation}
We endow $\SS$ with the following product: for $s \in \GG[x]$ and $t \in \GG[y]$,
\begin{equation}\label{eq:sh-to-lrgb}
    s \cdot t = \Delta_x^{xy}(s)\Delta_y^{xy}(t).
\end{equation}
The product on the right occurs inside the group $\GG[xy]$.
If two elements $s,t \in \SS$ belong to the same group $\GG[x]$, then the product $st$ in $\SS$ agrees with the product inside the group $\GG[x]$.

\begin{example}\label{ex:presheaf-hsiao}
    Let $\Sigma_n$ be the LRB of set compositions $[n]$ from \cref{ex:braidarrangement}
    and fix a finite group $G$.
    For a set composition $(S_1, \ldots, S_k) \in \Sigma_n$, let
    \[
        \GG[ (S_1,\dots,S_k) ] = (G^{op})^k,
    \]
    where $G^{op}$ denotes the opposite group of $G$.
    If ${\sf E} \preceq {\sf F}$, then
    \[
        \Delta_{\sf E}^{\sf F} \big( (g_1,\dots,g_k) \big) = (h_1,\dots,h_\ell)
    \]
    where $h_i = g_j$ if the $i^{th}$ block of ${\sf F}$ is contained in the $j^{th}$ block of ${\sf E}$.
    For example, if ${\sf E} = ( {\red 146} \mid {\blue 23578} )$ and ${\sf F} = ({\blue 25} \mid {\red 16} \mid {\blue 37} \mid {\blue 8} \mid {\red4})$, then
    \[
        \Delta_{\sf E}^{\sf F} \big( ({\red g_1},{\blue g_2}) \big) = ({\blue g_2},{\red g_1},{\blue g_2},{\blue g_2}, {\red g_1}).
    \]
    If ${\sf E} = (S_1,\dots,S_k)$ and ${\sf F} = (S_{\sigma(1)},\dots,S_{\sigma(k)})$ for some permutation $\sigma \in \mathfrak{S}_k$, then
    \[
        \Delta_{\sf E}^{\sf F} \big( (g_1,\dots,g_k) \big) = (g_{\sigma(1)},\dots,g_{\sigma(k)}).
    \]
    Let $\SS$ be the semigroup associated to $\GG$. The map $\SS \to \hsiao$ sending $(g_1 , \dots , g_k) \in \GG[(S_1 , \dots , S_k)]$ to $\big( (S_1,g_1) , \dots , (S_k,g_k) \big)$ is an isomorphism.

\end{example}

\begin{theorem}\label{thm:presheaf<->sLRBG}
    Let $\SS$ be the semigroup associated with a presheaf of finite groups
    $\GG$ on a LRB $\BB$. Then $\SS$ is a strict LRBG. Furthermore, every
    strict LRBG arises in this manner.
\end{theorem}

\begin{proof}
    We separate the proof in three parts.
    \begin{enumerate}[wide]
        \item
    Let $\SS$ be a strict LRBG with $E(\SS) = \BB$. Define a presheaf $\GG$ on $\BB$ by
    \[
        \GG[x] = G_x = \{ s \in \SS : s^\omega = x \} \quad\text{ for all } x \in \BB
    \]
    and
    \[
        \Delta_x^y(s) = ys \quad\text{for all } x \preceq y \in \BB \text{ and } s \in G_x.
    \]
    \cref{eq:lrbg2} guarantees that the maps $\Delta_x^y$ are group morphisms: for all $g,h \in G[x]$,
    \[
        \Delta_x^y(gh) = ygh = ygyh = \Delta_x^y(g)\Delta_x^y(h).
    \]
    Moreover, $\Delta_x^x$ is the identity map, since $x$ is the unit of $G_x$; and whenever $w \preceq x \preceq y$, we have
    \[
        \Delta_x^y ( \Delta_w^x (g) ) = y (x g) = (yx) g = yg = \Delta^y_w(g)
    \]
    for all $g \in G_w$. It follows that $\GG$ is a well-defined presheaf of finite groups.

    \item
    Let $\GG$ be a presheaf of finite groups on $\BB$
    and let $\SS$ be the semigroup associated with $\GG$.
    Note that the LRB property guarantees that $x \preceq xy$ and $y \preceq xy$ for all $x,y \in \BB$.
    We verify that $\SS$ is a semigroup (i.e. the product is associative), and that it satisfies the strict LRBG axioms.

    Given $s \in \GG[x]$, $t \in \GG[y]$, and $u \in \GG[z]$, functoriality implies
    \[
        (s \cdot t) \cdot u
        = \Delta_{xy}^{xyz}( \Delta_x^{xy}(s) \Delta_y^{xy}(t) ) \Delta_z^{xyz}(u)
        = \Delta_x^{xyz}(s) \Delta_y^{xyz}(t) \Delta_z^{xyz}(u).
    \]
    One similarly deduces that $s \cdot (t  \cdot u) = \Delta_x^{xyz}(s) \Delta_y^{xyz}(t) \Delta_z^{xyz}(u)$,
    which proves associativity.

    Observe that for $s, t \in \GG[x]$, the product $s \cdot t$ in $\SS$ agrees with
    the product $st$ in $\GG[x]$, so $s^\omega = 1_{\GG[x]}$ and
    $s^\omega \cdot s = s$.
    Now, for $s \in \GG[x]$ and $t \in \GG[y]$, we have, because $xyx = xy$ in $\BB$,
    \[
        s \cdot t \cdot s^\omega
        = \Delta_x^{xyx}(s) \Delta_y^{xyx}(t) \Delta_x^{xyx}(1_{\GG[x]})
        = \Delta_x^{xy}(s) \Delta_y^{xy}(t) 1_{\GG[xy]}
        = s \cdot t.
    \]
    Thus, $\SS$ is a LRBG.

    It follows from the definition of the product in $\SS$ that $\GG[x] \cdot
    \GG[y] \subseteq \GG[xy]$, therefore $\SS$ is LRBG that is a union
    of a band of groups, and hence is a strict LRBG by \cref{strict-iff-BG}.

    \item
    Finally, we show that these constructions are inverses of one another.
    Let $\SS$ be a strict LRBG and let $\GG$ the presheaf of groups defined in
    the first part of this proof.
    We show that the product in $\SS$ agrees with the product defined in \cref{eq:sh-to-lrgb}.
    Take any $s,t \in \SS$, and let $x = s^\omega, y = t^\omega$. Then,
    \[
        s \cdot t =
        \Delta_x^{xy}(s) \Delta_y^{xy}(t) = xys xyt
        \overset{\cref{eq:lrbg2}}{=}
        xy st
        =
        s^\omega t^\omega s t
        \overset{\cref{eq:strict}}{=}
        (st)^\omega st
        \overset{\cref{eq:lrbg1}}{=}
        st.
        \qedhere
    \]
    \end{enumerate}
 \end{proof}

\begin{remark}
    The last chain of equalities in the preceding proof cannot be deduced without assuming strictness.
    If $\SS$ is the LRBG in \cref{ex:not-strict}, the {\em modified} product on $\SS$
    defined in \cref{eq:sh-to-lrgb} satisfies
    \[
        s \cdot y = \Delta_{x}^{y}(s) \Delta_y^y(y) = ysy = y \neq s y.
    \]
    For a LRBG $\SS$, we can think of this modified product as a way of ``\emph{strictifying}'' its semigroup structure.
    Following the definitions, this product can directly be defined by:
    \[  s \cdot t = s^\omega t^\omega s t. \]
\end{remark}

\newpage

\section{Notation}

\subsection*{Notation for left regular bands of groups}

\subsubsection*{Notation for a left regular band of groups $\SS$}

\begin{itemize}[wide]
    \item
        Typical elements of $\SS$ are written $s$ and $t$,
        but elements in $E(\SS)$ are written $x$ and $y$.

    \item
        $\leq$ is the partial order on $\SS$ defined by $s \leq t$ if and only if $s t^\omega = t$.

    \item
        $\sim$ is the semigroup congruence on $\SS$ defined by $s \sim t$ if and only if $s^\omega t = s$ and $t^\omega s = t$.

    \item
        $G_{x} = \{ s \in \SS : s^\omega = x \}$ is the maximal subgroup of $\SS$ at $x \in E(\SS)$.

    \item
        $\lambda_{y, x}: G_{x} \to G_{yx}$ is the group morphism induced by left-multiplication by $y \in E(\SS)$.
\end{itemize}

\subsubsection*{Notation for the quotient $\TT := \SS/{\sim}$ (a semilattice of groups)}

\begin{itemize}[wide]
    \item
        $\supp: \SS \to \TT$ is the quotient map; it is a map of semigroups and of posets.

    \item
        Typical elements of $\TT$ are written $S$ and $T$, but
        elements of $E(\TT)$ are written as $X$ and $Y$.

    \item
        $\TT$ is a semilattice of groups; in particular, $\TT$ is a LRBG with central idempotents.

    \item
        $G_{X} = \{ S \in \TT : S^\omega = X \}$ is the maximal subgroup of $\TT$ at $X \in E(\TT)$.

    \item
        $E(\SS)$ is a LRB with support map $\supp: E(\SS) \to E(\TT)$.
\end{itemize}

\subsubsection*{Additional notation when $\SS$ is a left regular band of \emph{abelian} groups}

\begin{itemize}[wide]
    \item
        $\wh{\SS} := \bigsqcup_{x \in E(\SS)} \wh{G_x}$,
        where $\wh{G_x}$ is the dual group of $G_x$;
        and similarly
        $\wh{\TT} := \bigsqcup_{X \in E(\TT)} \wh{G_X}$.

    \item
        Elements of $\wh{\SS}$ are written as $\phi$ and $\psi$;
        elements of $\wh{\TT}$ are written as $\Phi$ and $\Psi$.

    \item
        $|\phi|$ is the element of $E(\SS)$ such that $\phi \in \wh{G_{|\phi|}}$.

    \item
        $\chleq$ is the partial order on $\wh{\SS}$ defined by
        $\phi \chleq \psi$ if and only if $|\phi| \leq |\psi|$ and $\phi = \lambda_{|\psi|,|\phi|}^*(\psi)$.

    \item
        $\sim$ is the relation on $\wh{\SS}$ defined by
        $\phi \sim \psi$ if and only if
        $|\phi| \sim |\psi|$ and $\phi = \lambda_{|\psi|,|\phi|}^*(\psi)$.
\end{itemize}

\subsection*{Notation for semigroup algebras of LRBGs and LRBaGs}

\subsubsection*{Notation for $\kk E(\SS)$}

\begin{itemize}[wide]
    \item
        $\{\B{H}_{x}\}_{x \in E(\SS)}$ is the canonical basis of $\kk E(\SS)$;
        it satisfies $\B{H}_{x} \B{H}_{y} = \B{H}_{xy}$ for all $x, y \in E(\SS)$.

    \item
        $\{\uu_{X}\}_{X \in E(\TT)}$ is a homogeneous section of $\supp: E(\SS) \to E(\TT)$.

    \item
        $\{\ee^\circ_{X}\}_{X \in E(\TT)}$ is the CSoPOI of $\kk E(\SS)$
        corresponding to $\{\uu_{X}\}_{X \in E(\TT)}$.

    \item
        $\{\B{Q}^{\circ}_{x} := \B{H}_{x} \ee^\circ_{\supp(x)} \}_{x \in E(\SS)}$
        is the $\B{Q}$-basis of $\kk E(\SS)$ corresponding to $\{\uu_{X}\}_{X \in E(\TT)}$.
\end{itemize}

\subsubsection*{Notation for $\kk E(\TT)$}

\begin{itemize}[wide]
    \item
        $\{\B{H}_{X}\}_{X \in E(\TT)}$ is the canonical basis of $\kk E(\TT)$;
        it satisfies $\B{H}_{X} \B{H}_{Y} = \B{H}_{XY}$ for all $X, Y \in E(\TT)$.

    \item
        $\{\B{Q}^{\circ}_{X} := \sum_{Y \geq X} \mu(X,Y) \B{H}_Y\}_{X \in E(\TT)}$
        is the \emph{$\B{Q}$-basis} of $\kk E(\TT)$; it is the \emph{unique} CSoPOI in $\kk E(\TT)$.
\end{itemize}

\subsubsection*{Notation for $\kk \SS$}

\begin{itemize}[wide]
    \item
        $\{\B{H}_{s}\}_{s \in \SS}$ is the canonical basis of $\kk \SS$;
        it satisfies $\B{H}_{s} \B{H}_{t} = \B{H}_{st}$ for all $s, t \in \SS$.

    \item
        $\{\B{R}_{s} := \B{H}_{s} \B{Q}^\circ_{s^\omega}\}_{s \in \SS}$
        is the $\B{R}$-basis of $\kk \SS$,
        where $\{\B{Q}^\circ_{x}\}_{x \in E(\SS)}$ is the $\B{Q}$-basis of $\kk E(\SS)$.

    \item
        The $\B{R}$-basis of $\kk \SS$ contains the
        $\B{Q}$-basis of $\kk E(\SS)$:
        if $x \in E(\SS)$, then $\B{R}_{x} = \B{Q}^\circ_{x}$.

    \item
        $\{\ee_{(X, i)}\}_{X, i}$ is the CSoPOI of $\kk \SS$
        defined in \cref{thm:CSoPOI-LRBG}.
\end{itemize}

\subsubsection*{Additional notation for $\kk \SS$ when $\SS$ is a left regular band of \emph{abelian} groups}

\begin{itemize}[wide]
    \item
        $\{\B{E}_{\phi} := \frac{1}{|G_{|\phi|}|} \sum_{s \in G_{|\phi|}}
        \overline{\phi(s)} \, \B{H}_{s} \}_{\phi \in \wh{\SS}}$ is the basis of
        locally orthogonal idempotents of $\kk \SS$.

    \item
        $\{\B{Q}_{\phi} := \B{E}_{\phi} \B{Q}^\circ_{|\phi|} \}_{\phi \in \wh{\SS}}$
        is the $\B{Q}$-basis of $\kk \SS$ corresponding to
        the $\B{Q}$-basis $\{\B{Q}^\circ_{x} \}_{x \in E(\SS)}$ of $\kk E(\SS)$.

    \item
        $\{\ee_{\Phi} := \uu_{|\Phi|} \B{Q}_{\phi} \}_{\Phi \in \wh{\TT}}$
        is the CSoPOI of $\kk \SS$ corresponding to
        the $\B{Q}$-basis $\{\B{Q}_{\phi} \}_{\phi \in \wh{\SS}}$ of $\kk \SS$.
\end{itemize}

\subsubsection*{Notation for $\kk \TT$}

\begin{itemize}[wide]
    \item
        $\{\B{H}_{S}\}_{S \in \TT}$ is the canonical basis of $\kk \TT$;
        it satisfies $\B{H}_{S} \B{H}_{T} = \B{H}_{ST}$ for all $S, T \in \TT$.

    \item
        $\{\B{R}_{S} := \B{H}_{S} \B{Q}^\circ_{S^\omega}\}_{S \in \TT}$
        is the $\B{R}$-basis of $\kk \TT$,
        where $\{\B{Q}^\circ_{X}\}_{X \in E(\TT)}$ is the $\B{Q}$-basis of $\kk E(\TT)$.

    \item
        The $\B{R}$-basis of $\kk \TT$ contains the
        $\B{Q}$-basis of $\kk E(\TT)$:
        if $X \in E(\TT)$, then $\B{R}_{X} = \B{Q}^\circ_{X}$.
\end{itemize}

\subsubsection*{Relationship between bases}

\begin{equation*}
    \begin{tikzcd}[row sep=3em, column sep=4em]
        \kk \SS \arrow[r, two heads, "\supp"] & \kk \TT \\
        \kk E(\SS) \arrow[u, hook, "id"] \arrow[r, two heads, "\supp"] & \kk E(\TT) \ar[u, hook, "id"']
    \end{tikzcd}
    \qquad
    \begin{tikzcd}[row sep=3em, column sep=4em]
           \B{H}_{s} \arrow[r, mapsto, "\supp"]
           &
           \B{H}_{\supp(s)} \\
           \B{H}_{x} \arrow[u, mapsto, "id"] \arrow[r, mapsto, "\supp"]
           &
           \B{H}_{\supp(x)} \arrow[u, mapsto, "id"']
    \end{tikzcd}
    \qquad
    \begin{tikzcd}[row sep=3em, column sep=4em]
           \B{R}_{s} \arrow[r, mapsto, "\supp"]
           &
           \B{R}_{\supp(s)} \\
           \B{Q}^\circ_{x} \arrow[u, mapsto, "id"] \arrow[r, mapsto, "\supp"]
           &
           \B{Q}^\circ_{\supp(x)} \arrow[u, mapsto, "id"']
    \end{tikzcd}
\end{equation*}

\subsection*{Notation for LRBGs with symmetry}\label{s:appendix-notation}

\subsubsection*{Additional notation for $\SS$ when $\SS$ is a LRBG with symmetry}

\begin{itemize}[wide]
    \item
        $W$ acts on $\SS$ via $w \cdot (st) = (w \cdot s)(w \cdot t)$ for all $s, t \in \SS$, $w \in W$.

    \item
        $\SS^W$ is the set of $W$-orbits and $\type: \SS \to \SS^W$ maps an element to its $W$-orbit.

    \item
        Elements of $\SS^W$ are written as $p$ and $q$.

    \item
        $p \leq q$ in $\SS^W$ if and only if there are $s \leq t$ in $\SS$
        with $\type(s) = p$ and $\type(t) = q$.
\end{itemize}

\subsubsection*{Additional notation for $\kk \SS$ when $\SS$ is a LRBG with symmetry}

\begin{itemize}[wide]
    \item
        $W$ acts on $\kk \SS$ via $w \cdot \B{H}_{s} = \B{H}_{w \cdot s}$ for all $w \in W$ and $s \in \SS$.

    \item
        $\{\uu_{X}\}_{X \in E(\TT)}$ is assumed to be a homogeneous section
        satisfying $w \cdot \uu_{X} = \uu_{w \cdot X}$.

    \item
        $(\kk \SS)^W = \{ a \in \kk \SS : w \cdot a = a \text{~for all~} w \in
        W \}$ is the invariant subalgebra of $\kk \SS$.

    \item
        $\{ \mathcal{H}_p := \sum_{\type(s) = p} \B{H}_{s} \}_{p \in \SS^W}$
        is the canonical basis of $(\kk \SS)^W$.

    \item
        $\{ \mathcal{R}_p := \sum_{\type(s) = p} \B{R}_{s} \}_{p \in \SS^W}$
        is the $\mathcal{R}$-basis of $(\kk \SS)^W$.
\end{itemize}

\subsubsection*{Additional notation for $\SS$ when $\SS$ is a LRBaG with symmetry}

\begin{itemize}[wide]
    \item
        $W$ induces an action on $\wh{\SS}$ given by $(w \cdot \phi) = \phi( w^{-1} \cdot t)$ for all $t \in G_{w \cdot x}$, $w \in W$.

    \item
        $\wh{\SS}^W$ is the set of $W$-orbits and
        $\type: \wh{\SS} \to \wh{\SS}^W$ maps an element to its $W$-orbit.

    \item
        Elements of $\wh{\SS}^W$ are written as $\alpha$ and $\beta$.

    \item
        $\alpha \chleq \beta$ in $\wh{\SS}^W$
        if and only if there are $\phi \chleq \psi$ in $\wh{\SS}$
        with $\type(\phi) = \alpha$ and $\type(\psi) = \beta$.
\end{itemize}

\subsubsection*{Additional notation for $\kk \SS$ when $\SS$ is a LRBaG with symmetries}

\begin{itemize}[wide]
    \item
        $\{ \mathcal{E}_\alpha := \sum_{\type(\phi) = \alpha} \B{E}_{\phi} \}_{\alpha \in \wh{\SS}^W}$
        is the $\mathcal{E}$-basis of $(\kk \SS)^W$,
        where $\{\B{E}_{\phi} \}_{\phi \in \wh{\SS}}$ is the $\B{E}$-basis of $\kk \SS$.

    \item
        $\{ \mathcal{Q}_\alpha := \sum_{\type(\phi) = \alpha} \B{Q}_{\phi} \}_{\alpha \in \wh{\SS}^W}$
        is the $\mathcal{Q}$-basis of $(\kk \SS)^W$,
        where $\{\B{Q}_{\phi} \}_{\phi \in \wh{\SS}}$ is the $\B{Q}$-basis of $\kk \SS$.

    \item
        $\{\ee_{\lambda} := \sum_{\type(\Phi) = \lambda} \ee_{\Phi}\}_{\lambda \in \wh{\TT}^W}$
        is a CSoPOI of $(\kk \SS)^W$,
        where $\{\ee_{\Phi} \}_{\Phi \in \wh{\TT}}$ is the CSoPOI of $\kk \SS$.
\end{itemize}

\subsection*{Notation for the LRBGs and LRBaGs of $G$-compositions}

\begin{center}
\begin{tabular}{llcc}
    \toprule
    Notation                                             & Name                           & Symbols for elements & Notation for elements                                        \\
    \midrule
    $\Sigma_n \cong E(\Sigma_n[G])$                      & compositions of $[n]$          & $x,y$                & $\big( 42 \mid 1 \mid 35 \big)$                              \\
    $\Pi_n \cong E(\Pi_n[G])$                            & partitions of $[n]$            & $X,Y$                & $\big\{ 42 \mid 1 \mid 35 \big\}$                            \\
    $\Gamma_n$                                           & compositions of $n$            &                      & $\big( 2 , 1 , 2 \big)$                                      \\
    $\Lambda_n$                                          & partitions of $n$              &                      & $\big\{ 2 , 1 , 2 \big\}$                                    \\
    \midrule
    $\Sigma_n[G]$                                        & $G$-compositions of $[n]$      & $s,t$                & $\big( 42^{g_1} \mid 1^{g_2} \mid 35^{g_3} \big)$            \\
    $\Pi_n[G]$                                           & $G$-partitions of $[n]$        & $S,T$                & $\big\{ 42^{g_1} \mid 1^{g_2} \mid 35^{g_3} \big\}$          \\
    $\Gamma_n[G] \cong (\Sigma_n[G])^{\sym_n}$           & $G$-compositions of $n$        & $p,q$                & $\big( 2^{g_1} , 1^{g_2} , 2^{g_3} \big)$                    \\
    $\Lambda_n[G] \cong (\Pi_n[G])^{\sym_n}$             & $G$-partitions of $n$          &                      & $\big\{ 2^{g_1} , 1^{g_2} , 2^{g_3} \big\}$                  \\
    \midrule
    $\Sigma_n[\wh{G}]$                                   & $\wh{G}$-compositions of $[n]$ & $\phi,\psi$          & $\big( 42^{\phi_1} \mid 1^{\phi_2} \mid 35^{\phi_3} \big)$   \\
    $\Pi_n[\wh{G}]$                                      & $\wh{G}$-partitions of $[n]$   & $\Phi,\Psi$          & $\big\{ 42^{\phi_1} \mid 1^{\phi_2} \mid 35^{\phi_3} \big\}$ \\
    $\Gamma_n[\wh{G}] \cong (\Sigma_n[\wh{G}])^{\sym_n}$ & $\wh{G}$-compositions of $n$   & $\alpha,\beta$       & $\big( 2^{\phi_1} , 1^{\phi_2} , 2^{\phi_3} \big)$           \\
    $\Lambda_n[\wh{G}] \cong (\Pi_n[\wh{G}])^{\sym_n}$   & $\wh{G}$-compositions of $n$   & $\lambda$            & $\big\{ 2^{\phi_1} , 1^{\phi_2} , 2^{\phi_3} \big\}$         \\
    \bottomrule
\end{tabular}
\end{center}

\end{document}